\documentclass[smallextended,numbook,runningheads,envcountsect,envcountsame,leqno]{svjour3}     % onecolumn (second format) default one
\usepackage{amsthm}                 % with the one from amsmath - gives qed-symbol
\usepackage[american]{babel}
\usepackage{graphicx}
\usepackage{amsmath}
\usepackage{epstopdf} % needed if you have eps figures in a pdflatex manuscript
\usepackage{mathptmx}      % use Times fonts if available on your TeX system
\usepackage{booktabs}
\usepackage{amstext}
\usepackage{amsfonts}
\usepackage{amssymb}
\usepackage{amsxtra} 
\usepackage{algorithm}
\usepackage{algpseudocode}
\usepackage{bbm}
\usepackage{url}
\usepackage[dvipsnames]{xcolor}
\usepackage{tikz}
\usepackage{float} 
\usepackage{pgfplots}
\usepackage{comment}
\usepackage{subfig}

\usepackage{pgfplotstable}
\pgfplotsset{compat=newest}
\pgfplotsset{plot coordinates/math parser=false}

\usepackage[pagebackref=true,backref=true]{hyperref}
\renewcommand*{\backrefalt}[4]{%
    \ifcase #1 \footnotesize{(Not cited.)}%
    \or        \footnotesize{(Cited on page~#2)}%
    \else      \footnotesize{(Cited on pages~#2)}%
    \fi}

\newcommand{\eqdef}{\overset{\text{def}}{=}} 

\graphicspath{ {./figures/} }

\newcommand \RR {\mathbb{R}}

\newcommand \EE {\mathbb{E}}
\newcommand \PP {\mathbb{P}}

\newcommand \Ncal {\mathcal{N}}

\DeclareMathOperator{\sign}{sign}

\DeclareMathOperator{\dom}{dom}

\newcommand{\scp}[2]{\langle #1 \,,\, #2\rangle}
\newcommand{\norm}[2][]{\|#2\|_{#1}}
\newcommand{\set}[2]{\{#1\, |\, #2\}}

\newtheorem{assumption}{Assumption}
\newtheorem{cor}[theorem]{Corollary}
\journalname{}
\begin{document}

% Faster Randomized Block Sparse Kaczmarz by Averaging
% Randomized Sparse Kaczmarz with Averaging
% Faster Randomized Block Sparse Kaczmarz
% MiniBatch Randomized Sparse Kaczmarz
% Accelerated Randomized Sparse Kaczmarz
\title{Faster Randomized Block Sparse Kaczmarz by Averaging\thanks{\textbf{Funding: }The work of the authors has been supported by the ITN-ETN project TraDE-OPT funded by the European Union’s Horizon 2020 research and innovation programme under the Marie Skłodowska-Curie grant agreement No 861137. This work represents only the author’s view and the European Commission is not responsible for any use that may be made of the information it contains.
%Grants or other notes about the article that should go on the front page should be placed here. 
% General acknowledgments should be placed at the end of the article.
}}


\author{Lionel~Tondji \and Dirk~A. Lorenz  %etc.
}


\institute{Lionel Tondji
  \at Institute for Analysis and Algebra, TU Braunschweig, 38092 Braunschweig, Germany,\\
  \email{l.ngoupeyou-tondji@tu-braunschweig.de}
  \and Dirk A. Lorenz
  \at Institute for Analysis and Algebra, TU Braunschweig, 38092 Braunschweig, Germany,\\
  \email{d.lorenz@tu-braunschweig.de}
}

\date{Received: date / Accepted: date}

\maketitle

\begin{abstract}
The standard randomized sparse Kaczmarz (RSK) method is an algorithm to compute sparse solutions of linear systems of equations and uses sequential updates, and thus, does not take advantage of parallel computations. In this work, we introduce a parallel (mini batch) version of RSK based on averaging several Kaczmarz steps. Naturally, this method allows for parallelization and we show that it can also leverage large over-relaxation. We prove linear expected convergence and show that, given that parallel computations can be exploited, the method provably provides faster convergence than the standard method. This method can also be viewed as a variant of the linearized Bregman algorithm, a randomized dual block coordinate descent update, a stochastic mirror descent update, or a relaxed version of RSK and we recover the standard RSK method when the batch size is equal to one. We also provide estimates for inconsistent systems and show that the iterates converges to an error in the order of the noise level. Finally, numerical examples illustrate the benefits of the new algorithm.

%Insert your abstract here. No references or citations in abstract! 
%Include keywords and mathematical subject classification numbers as needed.
%and
%65 Numerical analysis
%65F Numerical linear algebra
%65F10 Iterative numerical methods for linear systems
%68 Computer science
%68W Algorithms in computer science
%68W20 Randomized algorithms
%68W10 Parallel algorithms in computer science
%90 Operations research, mathematical programming
%90C Mathematical programming [See also 49Mxx, 65Kxx]
%90C25 Convex programming
\end{abstract}

\section{Introduction}
\label{intro}
In this work we are concerned with the fundamental problem of approximating sparse solutions of large scale linear systems of the form
\begin{equation}
\label{eq:linear_system}
    \mathbf{A} x = b
\end{equation}
with matrix $\mathbf{A} \in \RR^{m \times n}$ and right hand side $b \in \RR^{m}$. Linear systems like (\ref{eq:linear_system}) arise in several fields of engineering and physics problems, such as sensor networks ~\cite{tropp2011improved}, signal processing~\cite{hanke1990acceleration}, partial differential equations~\cite{olshanskii2014iterative}, filtering~\cite{khan2007distributed}, computerized tomography~\cite{hounsfield1973computerized}, optimal control~\cite{patrascu2017nonasymptotic}, inverse problems~\cite{jiao2017preasymptotic,rabelo2022kaczmarzillposed} and machine learning, to name just a few.
When $\mathbf{A}$ is too large to fit in memory, direct methods for solving equation (\ref{eq:linear_system}) are not feasible and iterative methods are preferred.  As long as one can afford full matrix vector products or the system matrix fits in memory, Krylov methods including the conjugate gradient (CG) algorithms~\cite{stiefel1952methods} are the industrial standard. On the other hand, randomized methods such as the randomized (block) Kaczmarz~\cite{Kac37,SV09} and coordinate descent method~\cite{nesterov2012efficiency} are effective if a single matrix vector product is too expensive and in some situations are even more efficient than CG method (see, e.g.,~\cite{SV09} for an example). Linear convergence of the Kaczmarz method has been shown in the randomized case in~\cite{SV09} and~\cite{popa2018convrate} analyzes convergence rates in the deterministic case. Moreover block Kaczmarz methods~\cite{moorman2021randomized,necoara2019faster,needell2014paved,P15,richtarik2020stochastic} have received much attention for their high efficiency for solving (\ref{eq:linear_system}) and distributed implementations.
In this work we propose and analyze the randomized sparse Kaczmarz method~\cite{SL19} and show that parallel computations and averaging as in~\cite{moorman2021randomized} lead to faster convergence.

\subsection{Related Work}
\label{sec:related_work}
\paragraph{\textbf{Randomized Kaczmarz}.}
In the large data regime, the randomized Kaczmarz (RK) is a popular iterative method for solving linear systems. In each iteration\footnote{We use subscript indices for components of a vector, columns or rows of a matrix, and also as iteration indices. But the meaning should always be clear from the context.}, a row vector $a_i^T$ of $\mathbf{A}$ is chosen at random from the system (\ref{eq:linear_system}) and the current iterate $x_k$ is projected onto the solution space
of that equation to obtain $x_{k+1}$.  Geometrically, at each iteration
\begin{equation*}
    x_{k+1} = \operatorname*{arg min}_{x \in \RR^n}\, \|x - x_k\|_2^2 \quad \text{s.t.} \quad \langle a_i, x \rangle = b_i.
\end{equation*}
It has been observed that the convergence of RK method can be accelerated by introducing relaxation.
In a relaxed variant of RK, a step is taken in the direction of this projection with the size of the step depending on a relaxation parameter.
Explicitly, the relaxed RK update is given by 
\begin{equation}
\begin{aligned}
\label{eq:rk}
x_{k+1} = x_k - w_{k,i} \tfrac{\scp{a_i}{x_k}-b_i}{\norm[2]{a_i}^2} \cdot a_i \,,
\end{aligned}
\end{equation}
with initial values $x_0=0$, where the $w_{k,i}, \, i \in \{1, \dots, m\}$ are relaxation parameters. Note that this update rule requires low cost per iteration
and storage of order $\mathcal{O} (n)$. For consistent systems
the relaxation parameters must satisfy
\begin{equation}
    0 < \liminf_{k \rightarrow \infty}w_{k,i} \leq \limsup_{k \rightarrow \infty}w_{k,i} < 2
\end{equation}
to ensure convergence~\cite{10.1145/359340.359351}. Fixing the relaxation parameters $w_{k,i} = 1$ for all iterations $k$ and indices $i$ leads to the standard RK method. In~\cite{necoara2019faster}, a block Kaczmarz variant under the name randomized block Kaczmarz (RBK) has been analyzed.  Linear convergence in expectation was shown for consistent systems of equations, with a rate depending on the geometric properties of the matrix, its submatrices, and on the size of the blocks. The convergence rate given in ~\cite{necoara2019faster} depends on the block size and the stochastic conditioning parameter of the most ill-conditioned block of the partition when a partition is used and on the most ill-conditioned block of the entire matrix $\mathbf{A}$ when the indices are sampled i.i.d. with replacement. The paper~\cite{needell2014paved} considers more general sampling strategies such as sampling from a partition of the rows of the matrix. In~\cite{du2020randomized} the authors investigate an extension of the randomized averaged block Kaczmarz method that can solve least squares problems. A parallel version of RK where a weighted average of independent updates is used was studied in~\cite{moorman2021randomized}. They showed that as the number of threads increases, the rate of convergence improves and the convergence horizon for inconsistent systems decreases. Another more general class of block methods are sketch-and-project methods~\cite{richtarik2020stochastic,gower2019adaptive}. For a linear system $\mathbf{A} x = b$, sketch-and-project methods iteratively project the current iterate onto the solution space of a sketched subsystem $\mathbf{S}^T \mathbf{A}x = \mathbf{S}^T b$. In particular, RK is a sketch-and-project method with $\mathbf{S}$ being rows of the identity matrix.

\paragraph{\textbf{Randomized Sparse Kaczmarz.}}
Recently, a new variant of the standard RK method namely the randomized sparse Kaczmarz method (RSK)~\cite{SL19,LWSM14,P15} with almost the same low cost and storage requirements has shown good performance in approximating sparse solutions of large consistent linear systems. It uses two variables $x_k^*$ and $x_k$ and the relaxed RSK update is given by
\begin{equation}
\begin{aligned}
\label{eq:rsk}
x_{k+1}^* &= x_k^* - w_{k,i} \tfrac{\scp{a_i}{x_k}-b_i}{\norm[2]{a_i}^2} \cdot a_i \,,\\
x_{k+1} &= S_{\lambda}(x_{k+1}^*)
\end{aligned}
\end{equation}
with initial values $x_0=x_0^*=0$, $\lambda >0$,
and the soft shrinkage operator 
\[
S_{\lambda}(x) = \max\{|x|-\lambda,0\} \cdot \sign(x)\,.
\]
Fixing the relaxation parameters $w_{k,i} = 1$ for all iterations $k$ and indices $i$ lead to the standard RSK method. 
For consistent systems the iterates of the standard RSK method converge in expectation to the solution of the regularized \emph{Basis Pursuit Problem}
\begin{equation}
\label{eq:RBPP}
\min_{x \in \RR^n} \lambda \cdot \norm[1]{x} + \tfrac{1}{2} \cdot \norm[2]{x}^{2} \quad \mbox{s.t.}\quad \mathbf{A}x=b \,.
\end{equation}
The advantage of \eqref{eq:RBPP} is the strong convexity property of the objective function. It is important to note that (see ~\cite{friedlander2008exact}) for a large but finite parameter $\lambda >0$, the solution of \eqref{eq:RBPP} gives a solution of 
\begin{equation}
\label{eq:BPP}
\min_{x \in \RR^n} \norm[1]{x} \quad \mbox{s.t.}\quad \mathbf{A}x=b \,,
\end{equation}
which is the famous \emph{Basis Pursuit Problem}~\cite{chen1998basispursuit}.
The $\ell_1$-norm has been used in many applications, with the goal of obtaining sparse or even sparsest solutions of underdetermined systems of linear equations and least-squares problems which is the basis of the theory of compressed sensing~\cite{candes2006compressive,donoho2006compressed}. The sparsest solution is given by minimizing the so-called zero-norm, $\norm[0]{x}$ counting the number of nonzero components in $x$. However, it is computationally intractable even for the simplest instances due to its combinatorial nature (it is even strongly NP-complete, see problem [MP5] in the seminal reference~\cite{garey1979NP}). In ~\cite{candes2006robust,donoho2005sparse} reasonable conditions are given under which a solution of \eqref{eq:BPP} is a sparsest solution. In~\cite{schopfer2022extended}, an extension of the RSK with linear expected convergence has been proposed for solving sparse least squares and impulsive noise problems while requiring only one additional column of the system matrix in each iteration.
    
\paragraph{\textbf{Block Sparse Kaczmarz Methods.}}
In this setting, a subset of rows $\mathbf{A}_{\tau_k}$ is used at each iteration, with $\tau_k \subseteq \left\{ 1,\dots,m \right\}$ and $|\tau_k| >1$ where $|\tau_k|$ denote the cardinality of the set of indices $\tau_k$. We usually have two approaches. The first variant is simply a \emph{block generalization} of the basic sparse Kaczmarz and the update is given by
\begin{equation}
\begin{aligned}
\label{eq:bsk}
x_{k+1}^* &= x_k^* - w_{k,i} \tfrac{ \mathbf{A}_{\tau_k}^{T}(\mathbf{A}_{\tau_k}x_k - b_{\tau_k})}{\|\mathbf{A}_{\tau_k}\|^2_2} \,,\\
x_{k+1} &= S_{\lambda}(x_{k+1}^*)
\end{aligned}
\end{equation}
Such block variants are considered, e.g., in~\cite{necoara2019faster,needell2014paved} (with $\lambda=0$) and in \cite{LSW14,P15} for $\lambda\geq 0$ and we refer to these iterative process as \emph{block sparse Kaczmarz} method. When $|\tau_k| = m$, we refer to (\ref{eq:bsk}) as the \emph{linearized Bregman method}~\cite{cai2009convergence,yin2010analysis,LSW14} and when  $|\tau_k| = 1$ as the \emph{randomized sparse Kaczmarz method}~\cite{LWSM14,SL19}. The main drawback of (\ref{eq:bsk}) is that it is not adequate for distributed implementations. The second variant of block sparse Kaczmarz can take advantage of distributed computing: each iteration takes $\eta$ steps of the relaxed randomized sparse Kaczmarz, independently in parallel, averages the results and applies the soft shrinkage to form the next iterate. This leads to the following iteration:
\begin{equation}
\begin{aligned}
\label{eq:rska}
x_{k+1}^* &= x_k^* - \frac{1}{\eta} \sum_{i \in \tau_k} w_i\tfrac{\scp{a_i}{x_k}-b_i}{\norm[2]{a_i}^2} \cdot a_i \,,\\
x_{k+1} &= S_{\lambda}(x_{k+1}^*)
\end{aligned}
\end{equation}
with initial values $x_0=x_0^*=0$, where $\tau_k \subseteq \left\{ 1,\dots,m \right\}$ denotes a random set of $\eta$ row indices sampled with replacement (and the $i$-th row is chosen with probability $p_{i}$) and $w_i$ represents the weight corresponding to the $i$-th row. Such block variants are considered, e.g., in ~\cite{richtarik2020stochastic,gower2019adaptive,moorman2021randomized,necoara2019faster,miao2022greedy} (with $\lambda=0$). Method~\eqref{eq:rska} is the main method presented and analyzed in this paper and we refer to it as the \emph{randomized sparse Kaczmarz with averaging} (RSKA) with more details in Algorithm \ref{alg:RSKA}. Note that the update (\ref{eq:rska}) is easy to implement on distributed computing units, and it is comparable in terms of cost per iteration to the basic sparse Kaczmarz update i.e., of order $\mathcal{O} (\eta n)$. If $\tau_k$ is a set of one index, we recover the relaxed RSK method and if in addition the weights are chosen as $w_i = 1$ for $i \in \left\{ 1,\dots,m \right\}$, we recover the standard RSK method.

\begin{algorithm}[ht]
  \caption{Randomized Sparse Kaczmarz with Averaging (RSKA)}
  \label{alg:RSKA}
  \begin{algorithmic}[1]
    \Require{starting points $x_0=x_{0}^{*}=0\in \RR^n$, matrix $\mathbf{A} \in \RR^{m\times n}$ with rows $0 \not=a_{i}^T\in \RR^{n}$ and vector  $b\in \RR^{m}$, batch size $\eta$, weights $\{w_i\}_{i=1}^{m} \in \RR$ and probabilities $p_{i}$}
    \Ensure{(approximate) solution of $\min_{x \in \RR^n} \lambda \|x\|_1 + \frac{1}{2}\|x\|_2^{2}$ s.t. $\mathbf{A} x=b$}
    \State initialize $k =  0$
    \Repeat
    \State $\tau_k \leftarrow \eta$ indices sampled from $\left\{ 1,\dots,m \right\}$
    \State Compute $\delta_k = \tfrac{1}{\eta} \sum_{i \in \tau_k} w_i \tfrac{\langle a_{i}, x_k \rangle - b_{i}}{\|a_{i}\|^2_2} \cdot a_{i}$ \Comment{In parallel}
    \State Update $x_{k+1}^* = x_k^{*} - \delta_k$
    \State Update $x_{k+1} = S_{\lambda}(x^{*}_{k+1})$
    \State Increment $k = k+1$
    \Until{a stopping criterion is satisfied}
  \end{algorithmic}
\end{algorithm}

\subsection{ Contribution and Organization.}    
\label{subsec:contribution}

To the best of our knowledge, the proposed block variant (\ref{eq:rska}) have not yet been proposed and analyzed for the randomized sparse Kaczmarz method. In this work, we make the following contributions:
\begin{itemize}
    \item We propose a mini batch version termed \emph{RSKA} of the randomized sparse Kaczmarz method, a general algorithm that unifies a variety of other methods such as the randomized Kaczmarz and the randomized sparse Kaczmarz with their relaxed variants. It is theoretically well-motivated, can exploit parallel computation and converges linearly in expectation.
    
    \item We prove that our proposal leads to faster convergence than its standard counterpart. We also validate this empirically and we provide implementations of our algorithm in \texttt{Python}.
\end{itemize}

The remainder of the paper is organized as follows. Section \ref{sec:basicnotions} provides a brief overview on convexity and Bregman distances. In section \ref{sec:interpretations} we give several interpretations of our method. Section \ref{sec:convergence} provides convergence guarantees for our proposed method.  In Section \ref{sec:numerical_experiments}, numerical experiments demonstrate the effectiveness of RSKA and provides several insights regarding its behavior and its hyper-parameters. Finally, Section \ref{sec:conclusion} draws some conclusions.

\subsection{Notation}
\label{subsec:notation}

%\section{Preliminaries}
%\label{sec:preliminaries}

In this section we introduce notation that will be used throughout. The first $m$ integers are denoted by $[m] \eqdef \{1, 2, \dots, m \}$. Given a symmetric positive definite matrix $\mathbf{B},$ we equip the space $\RR^{n}$ with the Euclidean inner product defined by 
\begin{equation*}
   \langle x, y \rangle_{\mathbf{B}} \eqdef  \langle x, \mathbf{B} y \rangle = \sum_{i,j \in [n]} x_i \mathbf{B}_{ij} y_j , \quad x, y \in \RR^{n}
\end{equation*}
We also define the induced norm: $\|.\|_{\mathbf{B}}^2 \eqdef \langle \cdot, \cdot \rangle_{\mathbf{B}}$ and use the short-hand notation $\|.\|$ to mean $\|.\|_{\mathbf{I}}$ to denote the standard 2-norm. 
\noindent
Let $\mathbf{A} \in \RR^{m \times n}$ be a real matrix. By $\mathbf{Range}(\mathbf{A}), \|\mathbf{A} \|_F$ and $a_i^T$ we denote its range space, Frobenius norm and $i$th row respectively, and by $A^{\dag}$ we denote the (Moore-Penrose) pseudo inverse. The indicator function of a set $C$ is denoted by 
\[
\delta_{C}(x) \eqdef 
\begin{cases}
    0, & \text{if } x \in C\\
    +\infty, & \text{if } x \not \in C
  \end{cases} 
\]%
By $e_i$ we denote the $i$th column of the identity matrix $\mathbf{I}_n \in \RR^{n \times n}$.
%\textcolor{red}{Let $\mathcal{D}$ be a user-defined distribution describing a set of \emph{sketching matrices} $\mathbf{S}_i \in \RR^{m \times s}$ for $i \in [m]$. The number of rows $s \in \NN$ is called the \emph{the sketch size}. In the standard randomized Kaczmarz method, the sketching matrices are $\mathbf{S} = e_i$ with probability $p_i = \frac{\norm[2]{a_i}^2}{\|A\|^2_F}$, for example, and in this paper we will use sketching matrices that consist of a subset of the rows of the identity matrix}. 
%By $\Delta_m$ we denote the simplex in $\RR^{m}$, that is 
%\[
%\Delta_m \eqdef \{p \in \RR^{m}_{+}: \scp{\mathbf{1}}{p} = 1  \}
%\]
%where we denote $\mathbf{1}$ as the all-ones vector with dimension inferred from context.
For a random vector $x_i$ that depends on a random index $i \in [q]$ (where $i$ is chosen with probability $p_{i}$) we denote $\EE _{i \sim p} [x_i] \eqdef \sum_{i\in [q]} p_ix_i$ and we will just write $\mathbb E [x_i]$ when the probability distribution is clear from the context.
% Let $\tau_k$ denote the set of indices sampled at the $k^{\text{th}}$ iteration \textcolor{red}{of Algorithm \eqref{alg:RSKA}} independently from $[m]$ with replacement and $\eta=|\tau_{k}|$ be the size of $\tau_k$ and is independent of the iteration $k$. The 
% $\eta$ sketching matrices $\mathbf S_{i_k}^1, \dots,  \mathbf S_{i_k}^{\eta}$ are sampled independently from $\Dcal$ at iteration $k$  with probability 
% \begin{equation}
%  \PP [\mathbf{S}_{i_k}=\mathbf{S}_j|x_k] =  \PP [i_k= j|x_k] = p_j^k, \quad \text{for}  \quad j \in [m]  
% \end{equation}      
% where $(p_1^k, \dots, p_m^k) \eqdef p^k$. The superscript $k$ is omitted when the probabilities do not depend on the iteration.
%We denote by $\EE_k \left[ \cdot \right] \eqdef \EE \left[ \cdot |\tau_{k-1}, \cdots, \tau_{0} \right]$ the expectation conditioned on the samples from iterations $0, 1,\dots, k - 1$. 
Let $\sigma_{i}(\mathbf{A})$ be the $i$th singular value of $A$ (when ordered decreasingly)  and $\sigma_{\min}(\mathbf{A})$ and $\sigma_{\max}(\mathbf{A})$ be the smallest and largest singular values of $\mathbf{A}$, respectively. They are given by
\begin{equation}
    \sigma_{\min}(\mathbf{A}) \eqdef \operatorname*{\min}_{x \in \RR^n,\ x\neq0} \, \frac{\|\mathbf{A} 
    x\|_2}{\|x\|_2} \quad \text{and} \quad \sigma_{\max}(\mathbf{A}) \eqdef \sigma_{1}(\mathbf{A}) \eqdef \operatorname*{\max}_{x \in \RR^n,\ x\neq0} \, \frac{\|\mathbf{A} 
    x\|_2}{\|x\|_2}
\end{equation}
Finally, a result we will need, if $\mathbf{A}$ is a symmetric positive semi-definite matrix the
largest singular value of $\mathbf{A}$ (the $L_2$ induced matrix norm or the spectral norm ) can be defined instead as
\begin{equation}
\label{eq:lsv}
        \|\mathbf{A}\|_2 \eqdef \sigma_{\max}(\mathbf{A}) = \max_{x \in \RR^n,\ x\neq0} \, \frac{|\langle \mathbf{A} x, x \rangle|}{\|x\|_2^2} = \max_{x \in \RR^n,\ x\neq0} \, \frac{\|\mathbf{A} x\|_2}{\|x\|_2}.
\end{equation}
Thus clearly
\[
  \frac{|\langle x, x \rangle_{\mathbf{ A}}|}{\|x\|_2^2} \leq \sigma_{\max}(\mathbf{A}).
\]
\noindent
Let $\tilde{\sigma}_{\min}(\mathbf{A}) \eqdef \min\{\sigma_{\min}(\mathbf{A}_J) \mid J\subseteq [n], \mathbf{A}_J\neq 0 \}$
where	$\mathbf{A}_J$ denotes the submatrix of $\mathbf{A}$ that is built up by the columns indexed by $J$ and $| x|_{\mathrm{min}} \eqdef \min\{|x_j| \mid x_j\neq 0\}$. 

\section{Basic notions}
\label{sec:basicnotions}
	
At first we recall some well known concepts and properties of convex functions and Bregman distances and later give upper-bounds of singular values of sum of matrices.
%Now we collect some basic notions on convexity and the Bregman distance.
Let $f:\RR^n \to \RR$ be convex (note that we assume that $f$ is finite everywhere, hence also continuous). The \emph{subdifferential} of $f$ is defined by
\[
\partial f(x) \eqdef \set{x^* \in \RR^n}{ f(y) \ge f(x) + \scp{x^*}{y-x}\: \mbox{for all $y \in \RR^n$}}
\]
at any $x \in \RR^n$ is nonempty, compact and convex.

The function  $f:\RR^n \to \RR$ is said to be \emph{$\alpha$-strongly convex}, if for all $x,y \in \RR^n$ and subgradients $x^* \in \partial f(x)$ we have
\[
f(y) \ge f(x) + \scp{x^*}{y-x} + \tfrac{\alpha}{2} \cdot \norm[2]{y-x}^2 \,.
\]
If $f$ is $\alpha$-strongly convex then $f$ is coercive, i.e.
\[
\lim_{\norm[2]{x} \to \infty} f(x)=\infty \,,
\]
and its \emph{Fenchel conjugate} $f^{*}:\RR^n \to \RR$ given by 
\[
f^*(x^*)\eqdef \sup_{y \in \RR^n} \scp{x^*}{y} - f(y)
\]
is also convex, finite everywhere and coercive.

Additionally, $f^*$ is differentiable with a \emph{Lipschitz-continuous gradient} with constant $L_{f^*}=\frac{1}{\alpha}$, i.e. for all $x^*,y^* \in \RR^n$ we have
\[
\norm[2]{\nabla f^*(x^*)-\nabla f^*(y^*)} \le L_{f^*} \cdot \norm[2]{x^*-y^*} \,,
\]
which implies the estimate
\begin{equation} \label{eq:Lip}
f^*(y^*) \le f^*(x^*) -\scp{\nabla f^*(x^*)}{y^*-x^*} + \tfrac{L_{f^*}}{2} \cdot \norm[2]{x^*-y^*}^2 \,.
\end{equation}

\begin{example}
\label{exmp:f}
The objective function
\begin{equation} \label{eq:spf}
f(x) \eqdef \lambda \cdot \norm[1]{x} + \tfrac{1}{2} \cdot \norm[2]{x}^{2}
\end{equation}
is strongly convex with constant $\alpha=1$ and its conjugate function can be computed with the soft shrinkage operator as
\[
f^{*}(x^{*}) = \tfrac{1}{2} \cdot \norm[2]{S_{\lambda}(x^{*})}^{2} \quad \mbox{with} \quad \nabla f^{*}(x^{*}) = S_{\lambda}(x^{*}) \,.
\]
\end{example}

\begin{definition} \label{def:D}
The \emph{Bregman distance} $D_f^{x^*}(x,y)$ between $x,y \in \RR^n$ with respect to $f$ and a subgradient $x^* \in \partial f(x)$ is defined as
\[
D_f^{x^*}(x,y) \eqdef f(y)-f(x) -\scp{x^*}{y - x}\,.
\]
\end{definition}

Fenchel's equality states that $f(x) + f^*(x^*) = \scp{x}{x^*}$ if $x^*\in\partial f(x)$ and implies that the Bregman distance can be written as
\[
D_f^{x^*}(x,y) = f^*(x^*)-\scp{x^*}{y} + f(y)\,.
\]

\begin{example}[cf.~\cite{SL19}]  \label{exmp:D}
For $f(x)=\frac{1}{2} \cdot \norm[2]{x}^2$ we just have $\partial f(x) = \{x\}$ and  $D_f^{x^*}(x,y)=\frac{1}{2}\norm[2]{x-y}^2$.
For $f(x) = \lambda \cdot \norm[1]{x} + \tfrac{1}{2} \cdot \norm[2]{x}^{2}$ and any $x^*=x+\lambda \cdot s \in \partial f(x)$ we have
\[
D_f^{x^*}(x,y)=\frac{1}{2} \cdot \norm[2]{x-y}^2 + \lambda \cdot(\norm[1]{y}-\scp{s}{y}) \,.
\]
\end{example}

The following properties are crucial for the convergence analysis of the randomized algorithms.
They immediately follow from the definition of the Bregman distance and the assumption of strong convexity of $f$, cf.~\cite{LSW14}.
For all $x,y,z \in \RR^n$ and $x^* \in \partial f(x)$, $y^* \in \partial f(y)$, $z^* \in \partial f(z)$ we have
\begin{equation} \label{eq:D}
\frac{\alpha}{2} \norm[2]{x-y}^2 \le  D_f^{x^*}(x,y) \le \scp{x^*-y^*}{x-y} \le \norm[2]{x^*-y^*} \cdot \norm[2]{x-y}
\end{equation}
\begin{equation} \label{eq:E}
 D_f^{x^*}(x,y) +  D_f^{y^*}(y,z) -  D_f^{x^*}(x,z) = \scp{x^{*} - y^{*}}{z-y} 
\end{equation}

Note that if $f$ is differentiable with a Lipschitz-continuous gradient, then we also have the (better) upper estimate $D_f^{x^*}(x,y) \le L_f \cdot \norm[2]{x-y}^2$, but in general this need not be the case. 

The following Theorem, will be use in the convergence analysis more precisely in Lemma \eqref{lm:svt}.
\begin{theorem}[{\cite[Theorem 3.3.16(c)]{horn1994topics}}]
	\label{thm:svi}
	Let $\mathbf{A}, \mathbf{B} \in \RR^{m\times n}$ be given and let $p = \min\{m,n\}$. Then it holds for the decreasingly ordered singular values of $\mathbf{A},  \mathbf{B}, \mathbf{A+B}$ that
	\[
          |\sigma_{i}(\mathbf{A+B}) - \sigma_{i}(\mathbf{A}) | \leq \sigma_{1}(\mathbf{B}), \quad \text{for} \quad i \in [p].
        \]
	In particular we have
	\[
          \sigma_{i}(\mathbf{A}+\mathbf{B}) \geq \sigma_{i}(\mathbf{A}) - \sigma_{1}(\mathbf{B}), \quad \text{for} \quad i \in [p].
        \]
\end{theorem}

\section{Interpretations}
\label{sec:interpretations}
We can view the randomized sparse Kaczmarz with averaging algorithm as an optimization method for solving a specific primal or dual optimization problem. More precisely, the RSKA algorithm is a particular case of the following.
\subsection{Randomized Block/Parallel Coordinate Descent}
Considering the regularized Basis Pursuit Problem as primal problem
\begin{equation}
\label{eq:bpp}
\min_{x \in \RR^n} \lambda \cdot \norm[1]{x} + \tfrac{1}{2} \cdot \norm[2]{x}^{2} \quad \mbox{s.t.}\quad \mathbf{A}x=b \,.
\end{equation}
\noindent
The dual of optimization problem (\ref{eq:bpp}) takes the form of a quadratic program:
\begin{equation}
\label{eq:dualbpp}
\min_{y \in \RR^m} \tfrac{1}{2} \cdot \norm[2]{S_{\lambda}(\mathbf{A}^Ty)}^{2} - \scp{b}{y} \,.
\end{equation}
where the primal variable $x$ and the dual variable $y$ are related through the relation
$x = S_{\lambda}(\mathbf{A}^T y)$. Let us define the primal and dual objective functions 
\begin{equation*}
    f(x) = \lambda \cdot \norm[1]{x} + \tfrac{1}{2} \cdot \norm[2]{x}^{2} + \delta_{\{0\}}(b-\mathbf{A}x)
\end{equation*}
and
\begin{equation*}
    g(y) = \tfrac{1}{2} \cdot \norm[2]{S_{\lambda}(\mathbf{A}^Ty)}^{2} - \scp{b}{y},
\end{equation*}
respectively. One iteration of the RSKA algorithm can be viewed as one step of the randomized block coordinate descent (RBCD) applied to the dual problem (\ref{eq:dualbpp}) when the weights $w_i$ are chosen in a particular form~\cite{P15}$^{}$. 
More formally  a negative gradient step in the random $i$-th component of $y$ having $\nabla_{i_k} g(y) = a_{i_k}^TS_{\lambda}(\mathbf{A}^T y) - b_{i_k}$ with step size $t_k = \tfrac{1}{\|a_{i_k}\|^2_2}$ yields
\begin{align*}
    y_{k+1} &= y_{k} - \tfrac{1}{\|a_{i_k}\|^2_2} \nabla_{i_k} g(y_k) e_{i_k}
\end{align*}
We easily recover (\ref{eq:rsk}) by simply multiplying this update with $\mathbf{A}    ^T$ and using the
relation between the primal and dual variables given by $x = S_{\lambda}(\mathbf{A}^T y)$. Consider the dual problem (\ref{eq:dualbpp}). If we choose the particular weights $w_{k,i} = \tfrac{\|a_i\|^2_2}{\sum_{i \in \tau_k}\|a_i\|^2_2}$, the block coordinate
descent method applied to $g$ from (\ref{eq:dualbpp}) reads
\begin{equation*}
\begin{aligned}
x_{k+1}^* &= x_k^* - \frac{1}{\eta\sum_{i \in \tau_k}\|a_i\|^2_2} \sum_{i \in \tau_k} (\scp{a_i}{x_k}-b_i) \cdot a_i \,,\\
x_{k+1} &= S_{\lambda}(x_{k+1}^*)
\end{aligned}
\end{equation*}
Parallel coordinate descent ~\cite{richtarik2016parallel} applied to the dual problem \eqref{eq:dualbpp} with learning rate $t_k = \frac{1}{\eta\|a_i\|^2_2}$ yields
\begin{equation}
\label{eq:pcdm}
y_{k+1} = y_k - \sum_{i \in \tau_k} \frac{1}{\eta\|a_i\|^2_2} \nabla_{i} g(y_k) e_{i}    
\end{equation}
In Update \eqref{eq:pcdm} only coordinate $i \in \tau_k$ are updated in $y_{k+1}$ and the remaining coordinate are unchanged. Multiplying this update with $\mathbf{A}^T$ and using the
relation between the primal and dual variables, we recover (\ref{eq:rsk}) with particular weights $w_i = 1$. However, for general weights $w_k$, the RSKA algorithm cannot be interpreted in these ways, and thus our scheme is more general.
It is important to note that in~\cite{P15} for the randomized block sparse Kaczmarz method of type $(\ref{eq:bsk})$ sublinear convergence rates have been obtained by identifying the iteration as a randomized block coordinate gradient descent method applied to the objective function $g$ of the unconstrained dual of $f$. However, the rates given in~\cite{P15} are in terms of the dual objective function $g$, and not of the primal iterates only, although, as mentioned there in the conclusions, the experimental results indicate that such rates also hold for the primal iterates.

\subsection{Stochastic Mirror Descent with Stochastic Polyak Stepsize}

The stochastic mirror descent (SMD) method and its variants ~\cite{beck2003mirror,lan2012validation,nemirovski2009robust} is one of the most
widely used family of algorithms in stochastic optimization  for non-smooth, Lipschitz continuous– convex and non-convex functions. Starting
with the orginal work of ~\cite{nemirovskij1983problem}, SMD has been studied in the context of
convex programming ~\cite{nemirovski:hal-00976649}, saddle-point problems ~\cite{mertikopoulos2018optimistic}, and monotone variational inequalities ~\cite{mertikopoulos2018stochastic}. Now we draw a connection of Algorithm~\ref{alg:RSKA} to the stochastic mirror descent method using stochastic Polyak stepsizes.
We consider a set of sketching matrices $\mathbf{S}_{i}\in \RR^{m\times s}$ ($i\in[m]$), define $\mathbf{Z}_{i} \eqdef  \mathbf{S}_{i}(\mathbf{S}_{i}^T \mathbf{A} \mathbf{A}^T \mathbf{S}_{i})^{\dagger}\mathbf{S}_{i}^T$ and consider the stochastic convex quadratic $h_{\mathbf{S}_{i}}$
\begin{equation}
  h_{\mathbf{S}_{i}}(x) \eqdef \frac{1}{2}\|\mathbf{A}x-b\|_{\mathbf{Z}_{i}}^2 = \frac{1}{2} (\mathbf{A}x-b)^T \mathbf{Z_{i}} (\mathbf{A}x-b),  
\end{equation}
(recall that $\mathbf{A} \in \RR^{m \times n}$, $b \in \RR^{m}$).

The general sketched mirror descent update with learning rate $t_k$ and a mirror map $f$ works as follows: For a given iterate $x_{k}$ draw a sketching matrix $\mathbf{S}_{i_{k}}$ (actually, one draws an index $i_{k}$) at random and update
\begin{equation}
\label{eq:smd}
    x_{k+1} = \operatorname*{arg\,min}_{x \in \RR^{n}} \, \Bigg \{\langle \nabla h_{\mathbf{S}_{i_{k}}}(x_k), x - x_k \rangle  + \frac{1}{t_k} D^{x_k^{*}}_{f}(x_k, x) \Bigg\},\quad x_k^{*} \in \partial f(x_k),
\end{equation}
which yields to the following update:
\begin{equation}
\begin{aligned}
\label{eq:smd1}
x_{k+1}^* &= x_k^* - t_k \nabla h_{\mathbf{S}_{i_{k}}}(x_k) \,,\\
x_{k+1} &= \nabla f^{*}(x_{k+1}^*).
\end{aligned}
\end{equation}
where $f^*$ denote the Fenchel conjugate of $f$.
% The above update \eqref{eq:smd} of stochastic mirror descent is quite general.
% The flexibility of
% selecting a distribution $\Dcal$ allow us to obtain different stochastic reformulations of the linear system (\ref{eq:linear_system}).
% \textcolor{red}{
% \begin{example}
% \label{eg:example1}
% Several iterative methods can be seen as special case of the above update \eqref{eq:smd1}. For example choosing 
% \begin{itemize}
%     \item  $f(x) = \frac{1}{2} \norm[2]{x}^2$ , $\mathbf{B} = \mathbf{I}$, $\mathbf{S} = e_i$ and $t_k = 1$ yields to randomized Kaczmarz.
%     \item $f(x) = \frac{1}{2} \norm[2]{x}^2$ , $\mathbf{B} = \mathbf{A}^T\mathbf{A}$, $\mathbf{S} = \mathbf{A}e_i = $ and $t_k = 1$ yields to coordinate descent.
%     \item $f(x) = \lambda \cdot \norm[1]{x} + \frac{1}{2} \cdot \norm[2]{x}^2$ , $\mathbf{B} = \mathbf{I}$, $\mathbf{S} = e_i$ and $t_k = 1$ yields to randomized sparse Kaczmarz.
% \end{itemize}
% \end{example}
% }

One can show that 
\begin{align}
\label{eq:smdpd}
\footnotesize
 D^{x_{k+1}^{*}}_{f}(x_{k+1}, \hat x) \leq D^{x_k^{*}}_{f}(x_k, \hat x) -   t_k\langle \nabla h_{\mathbf{S}_{i_{k}}}(x_k), x_k - \hat x \rangle + \frac{t_k^2}{2}\|\nabla h_{\mathbf{S}_{i_{k}}}(x_k)\|^2
\end{align}
(actually, this also follows from Lemma \ref{lm:mp} below with $\varPhi(x) = \langle \nabla h_{\mathbf{S}_{i_{k}}}(x_k), x - x_k \rangle$, and a strongly convex function $f$).
If we select $t_k$ such that the RHS of inequality (\ref{eq:smdpd}) is minimized, we obtain 
\begin{equation}
\label{eq:mps}
    t_k = \frac{\langle \nabla h_{\mathbf{S}_{i_{k}}}(x_k), x_k - \hat x \rangle}{\|\nabla h_{\mathbf{S}_{i_{k}}}(x_k)\|^2}.
  \end{equation}
Since $\nabla h_{\mathbf{S}}(x) = \mathbf{A}^T \mathbf{Z}(\mathbf{A}x-b)$ we get
(cf.~\cite{richtarik2020stochastic}) 
\begin{equation}
\label{eq:psr}
   h_{\mathbf{S}_{i_{k}}}(x) - h_{\mathbf{S}_{i_{k}}}(\hat x) \overset{h_{\mathbf{S}_{i_{k}}}(\hat x) = 0}{=}  h_{\mathbf{S}_{i_{k}}}(x) = \tfrac{1}{2} \|\nabla h_{\mathbf{S}_{i_{k}}}(x)\|^2 = \tfrac12\langle \nabla h_{\mathbf{S}_{i_{k}}}(x_k), x - \hat x \rangle
 \end{equation}
 we get that the optimal step size is in fact simply
\begin{equation}
\label{eq:mps-polyak}
    t_k =   \frac{2\big[h_{\mathbf{S}_{i_{k}}}(x_k) - h_{\mathbf{S}_{i_{k}}}(\hat x)\big]}{\|\nabla h_{\mathbf{S}_{i_{k}}}(x_k)\|^2}.
\end{equation}
This quotient is known as \emph{stochastic mirror Polyak stepsize}~\cite{d2021stochastic}, (note that in this particular case, we always get $t_{k}=1$).
The randomized Kaczmarz (\ref{eq:rsk}) is
equivalent to one step of the stochastic mirror descent (\ref{eq:smd}) with the mirror Polyak stepsize $t_k$ given in (\ref{eq:mps}), whereas the mirror Polyak stepsize in it general form~\cite{d2021stochastic,loizou2021stochastic} is given by $t_k = \frac{h_{\mathbf{S}_{i_{k}}}(x_k) - h_{\mathbf{S}_{i_{k}}}(\hat x)}{c\|\nabla h_{\mathbf{S}}(x_k)\|^2_2}.$ The parameter $0 < c \in \RR $ in the step size is an important quantity which can be set theoretically based on the properties of the function under study. In~\cite{loizou2021stochastic} it is suggested that, for optimal convergence, one should select $c = 1/2$  for strongly convex functions and $c=0.2$ for non-convex functions. 
%Moreover, the general randomized sparse Kaczmarz with averaging can be written using this framework with $\mathbf{Z} \eqdef \frac{1}{\eta} \sum_{i\in \tau} w_i \mathbf{S}_{i}(\mathbf{S}_{i}^T \mathbf{A} \mathbf{A}^T \mathbf{S}_{i})^{\dagger}\mathbf{S}_{i}^T$ and we recover the standard randomized sparse Kaczmarz with averaging when the sketching matrices $\mathbf{S}_{i}$ are sampled over the standard basis vectors $e_i$. 
Morever, the general RSKA method (see Algorithm \ref{alg:RSKA}) falls into the general sketched-and-project framework in the context of mirror descent with $\mathbf{Z} \eqdef \frac{1}{\eta} \sum_{j\in \tau} w_j \mathbf{S}_{j}(\mathbf{S}_{j}^T \mathbf{A} \mathbf{A}^T \mathbf{S}_{j})^{\dagger}\mathbf{S}_{j}^T$ with $\mathbf{S}_{j} = e_{j}$, $t_{k}=1$ and $f(x) = \lambda \cdot \|x\|_{1} + \tfrac12 \cdot \|x\|_{2}^{2}$. In fact, the sketch-and-project form of updates (\ref{eq:rska}) (resp. \ref{eq:smd1}) is given by:
\begin{equation}
\begin{aligned}
\label{eq:rska1}
x_{k+1}^* &= x_k^* - \mathbf{A}^T \mathbf{Z} (\mathbf{A}x_k - b),\\
x_{k+1} &= \nabla f^*(x_{k+1}^*),
\end{aligned}
\end{equation}
with some random $\tau_{i}\subset [m]$ with cardinality $\eta$.

\section{Convergence analysis}
\label{sec:convergence}
In this section we show expected linear convergence for the randomized sparse Kaczmarz with averaging method. 
%Before that we give several reformulations of our proposed method.
Before that we give the update satisfy by the iterate $x_k^*$ in Algorithm \eqref{alg:RSKA}. From line 5 of Algorithm \eqref{alg:RSKA}, it holds that:
\begin{align*}
    x_{k+1}^* &= x_k^*  - \frac{1}{\eta} \sum_{i \in \tau_k} w_i\tfrac{\scp{a_i}{x_k}-b_i}{\norm[2]{a_i}^2} \cdot a_i \\
    &= x^{*}_k - \tfrac{1}{\eta} \sum_{i \in \tau_k} w_i (e_{i}^T\mathbf{A})^T \cdot \tfrac{e_{i}^T(\mathbf{A} x_k - b)}{\|a_{i}\|^2_2}\\
    &= x^{*}_k - \mathbf{A}^T\tfrac{1}{\eta} \sum_{i \in \tau_k} w_i  \cdot \tfrac{e_{i} e_{i}^T}{\|a_{i}\|^2_2}(\mathbf{A} x_k - b)
\end{align*}
To simplify notation, we define the following matrices.
\begin{definition} 
\label{def:definition1}
Let $\mathbf{Diag}(d_1, d_2, \dots, d_m)$ denote the diagonal matrix with $d_1, d_2, \dots, d_m$ on the diagonal. We define the following matrices:
\begin{itemize}
    \item Weighted sampling matrix: 
    \begin{align*}
    \mathbf{M}_k = \tfrac{1}{\eta} \sum_{i \in \tau_k} w_i \tfrac{e_i e_i^T}{\|a_{i}\|^2_2}
    \end{align*}
    \item Normalization matrix:
    \begin{align*}
    \mathbf{D} = \mathbf{Diag} (\|a_{1}\|, \|a_{2}\|, \dots, \|a_{m}\|)
    \end{align*}
    so that the matrix $\mathbf{D}^{-1} \mathbf{A}$ has rows with unit norm.\\
    \item Probability matrix:
    \begin{align*}
    \mathbf{P}=\mathbf{Diag}(p_1,p_2,\dots,p_m)
    \end{align*}
    where $p_j = \PP (i=j)$.\\
    \item Weight matrix:
    \begin{align*}
    \mathbf{W} = \mathbf{Diag}(w_1,w_2,\dots,w_m)
    \end{align*}
    where $w_i$ represents the weight
corresponding to the $i$-th row.
\end{itemize}
\end{definition}
The following lemma and proof are taken from \cite[Lemma 1]{moorman2021randomized} and we include the full proof for completeness. The lemma gives the first and second moment of the random matrices $\mathbf{M}_k$, $\mathbf{A}^T\mathbf{M}_k$ respectively and will be use in our convergence analysis. 
\begin{lemma}
\label{lemma1}
Let $\mathbf{M}_k, \mathbf{P}, \mathbf{W}$ and $\mathbf{D}$ be defined as in Definition \ref{def:definition1}. Then
\begin{align*}
    \EE_k \left[ \mathbf{M}_k\right] &= \mathbf{P}\mathbf{W}\mathbf{D}^{-2}\quad\text{and}\\
   \EE_k \left[ (\mathbf{A}^T\mathbf{M}_k)^T \cdot (\mathbf{A}^T\mathbf{M}_k)\right] & = \tfrac{1}{\eta} \mathbf{PW^2D^{-2}} + (1-\tfrac{1}{\eta}) \mathbf{PWD^{-2}AA^TPWD^{-2}}.
\end{align*}
\end{lemma}
\begin{proof}
Let $\mathbb{E}_{i}[\cdot]$ denote $\mathbb{E}_{i \sim p}[\cdot]$. From the definition of the weighted sampling matrix $\mathbf{M}_k$ as the weighted average of the i.i.d. sampling matrices $\tfrac{e_i e_i^T}{\|a_{i}\|^2_2},$ we see that
\begin{align*}
    \EE_k \left[ \mathbf{M}_k\right] = \EE_k \left[ \tfrac{1}{\eta} \sum_{i \in \tau_k} w_i \tfrac{e_i e_i^T}{\|a_{i}\|^2_2}\right] = \EE_i \left[  w_i \tfrac{e_i e_i^T}{\|a_{i}\|^2_2}\right] = \sum_{i=1}^{m} p_i w_i \tfrac{e_i e_i^T}{\|a_{i}\|^2_2} = \mathbf{P}\mathbf{W}\mathbf{D}^{-2}.
\end{align*}
In the same way, we have
\begin{align*}
  & \EE_k \left[ (\mathbf{A}^T\mathbf{M}_k)^T \cdot (\mathbf{A}^T\mathbf{M}_k)\right] \\
  &= \EE_k \left[ (\tfrac{1}{\eta} \sum_{i \in \tau_k} w_i \tfrac{e_i e_i^T}{\|a_{i}\|^2_2}) \mathbf{A}  \cdot \mathbf{A}^T \cdot (\tfrac{1}{\eta} \sum_{j \in \tau_k} w_j \tfrac{e_j e_j^T}{\|a_{j}\|^2_2})\right] \\
  &= \EE_k \left[ (\tfrac{1}{\eta} \sum_{i \in \tau_k} w_i \tfrac{e_i a_i^T}{\|a_{i}\|^2_2}) \cdot (\tfrac{1}{\eta} \sum_{j \in \tau_k} w_j \tfrac{a_j e_j^T}{\|a_{j}\|^2_2})\right] \\
  &= \tfrac{1}{\eta} \EE_i \left[ ( w_i \tfrac{e_i a_i^T}{\|a_{i}\|^2_2}) \cdot (w_i \tfrac{a_i e_i^T}{\|a_{i}\|^2_2})\right] + (1 - \tfrac{1}{\eta})\EE_i \left[  w_i \tfrac{e_i a_i^T}{\|a_{i}\|^2_2}\right] \EE_i \left[ w_i \tfrac{a_i e_i^T}{\|a_{i}\|^2_2} \right] \\
  &= \tfrac{1}{\eta} \EE_i \left[ w_i^2 \tfrac{e_i e_i^T}{\|a_{i}\|^2_2} \right] + (1 - \tfrac{1}{\eta})\EE_i \left[  w_i \tfrac{e_i e_i^T}{\|a_{i}\|^2_2}\right] \mathbf{A} \mathbf{A}^T \EE_i \left[ w_i \tfrac{e_i e_i^T}{\|a_{i}\|^2_2} \right] \\
  &= \tfrac{1}{\eta}\mathbf{P}\mathbf{W}^{2}\mathbf{D}^{-2} + (1 - \tfrac{1}{\eta})\mathbf{P}\mathbf{W}\mathbf{D}^{-2} \mathbf{A} \mathbf{A}^T \mathbf{P}\mathbf{W}\mathbf{D}^{-2},
\end{align*}
by separating the cases where $i = j$ from those where $i \neq j$ and utilizing the independence of the indices sampled in $\tau_k$ .
\end{proof}

We now present convergence results for the proposed method. We start our analysis by characterizing the error bound between two consecutive iterates and the error bound between the Bregman distance of the iterates, the solution and the residual in the following lemma.

\begin{lemma}
\label{lm:mp} 
Let $f, \varPhi: \RR^n \rightarrow \RR \cup \{+\infty\}$ be convex, where $\dom(f)=\RR^n$ and $\dom(\varPhi)\neq\emptyset$.
Let $ \mathcal{X} \subseteq \mathrm{dom}(\varPhi)$ be nonempty and convex, $x_k \in \RR^n$, $x_k^*\in\partial f(x_k)$. Assume that

\begin{equation*}
    x_{k+1} \in \operatorname*{arg\,min}_{x \in \mathcal{X}} \, \Bigg \{\varPhi(x)  + D^{x_k^{*}}_{f}(x_k, x)\Bigg\}.
\end{equation*}
Then there exist subgradient $x_{k+1}^{*}\in\partial f(x_{k+1})$ such that it holds
\begin{align*}
& \varPhi(y) + D^{x_k^{*}}_{f}(x_k, y)
 \geq \varPhi(x_{k+1}) + D^{x_k^{*}}_{f}(x_k, x_{k+1})  +  D^{x_{k+1}^{*}}_{f}(x_{k+1}, y)
\end{align*}
for any $y \in \mathcal{X}$.
\end{lemma}
\begin{proof}
Let denote by $J(x) = \varPhi(x)  + D^{x_k^{*}}_{f}(x_k, x)$. Since $J$ and $\mathcal X$ are convex and $x_{k+1}$
minimizes $J$ over $\mathcal{X}$, there exists a subgradient $d \in \partial J(x_{k+1})$ such that $$\langle d, y - x_{k+1} \rangle \geq 0 \,\,  \quad \forall \, y \in \mathcal{X}.$$
Since $f$ is finite everywhere, we have $\dom(D_{f}(\cdot,u))=\RR^n$ for all $u \in \RR^n$. Since $\dom(\varPhi)$ is nonempty and convex, $\varPhi$ has nonempty relative interior. So the subgradient sum rule applies and we obtain that
\[ \partial J(x_{k+1}) = \partial \varPhi(x_{k+1}) + (\partial f(x_{k+1}) - x_k^*). \]
Hence, there exist subgradients $g \in \partial \varPhi(x_{k+1})$, $x_{k+1}^{*}\in \partial f(x_{k+1})$ such that
\[\langle g + (x_{k+1}^{*} - x_{k}^{*}), \ y - x_{k+1} \rangle \geq 0 \,\,  \quad \forall \, y \in \mathcal{X}.\]
Therefore using the property of the subgradient and \eqref{eq:E}, we have for all $y \in \mathcal{X}$
\begin{align*}
\varPhi(y)  \geq& \ \varPhi(x_{k+1}) + \langle g, y-x_{k+1} \rangle \\
\geq& \ \varPhi(x_{k+1}) + \langle x_k^{*} - x_{k+1}^{*}, \ y-x_{k+1} \rangle \\
\overset{\text{\tiny \eqref{eq:E}}}{=}& \ \varPhi(x_{k+1}) + D^{x_k^{*}}_{f}(x_k, x_{k+1}) - D^{x_k^{*}}_{f}(x_k, y) + D^{x_{k+1}^{*}}_{f}(x_{k+1}, y).
\end{align*}
\end{proof}

The following lemma provides an error bound for the Bregman distance.
\begin{lemma}[\cite{SL19}]
\label{lm:Gamma_for_Bregman_Ax_bound}
Let $\tilde{\sigma}_{\min}(\mathbf{A})$ and $|\hat x|_{\mathrm{min}}$ be defined as in subsection \ref{subsec:notation}. Then for any $x \in  \RR^n$ with $\partial f(x) \cap \mathbf{Range}(\mathbf{A}^T) \neq 0$ and for all $\hat x = \mathbf{A}^T y \in \partial f(x) \cap \mathbf{Range}(\mathbf{A}^T)$, we have 
\begin{equation}
	\label{eqn:Bregman_Ax_bound}
	D_f^{x_k^*}(x_k,\hat x) \leq \gamma \cdot  \|\mathbf{A}x_k-b\|_2^2
\end{equation}
where 
\begin{equation}
	\label{eqn:Gamma}
	\gamma = \frac{1}{\tilde{\sigma}^2_{\text{min}}(\mathbf{A})} \frac{ |\hat x|_{\mathrm{min}} + 2\lambda }{|\hat x|_{\mathrm{min}}} 
\end{equation}
\end{lemma}

To effectively use Lemma \ref{lm:Gamma_for_Bregman_Ax_bound} we need the following assumption which characterize the coupling between the weight matrix and the probability matrix.

\begin{assumption}
\label{as:coupling}
The weight matrix $\mathbf{W}$ and the probability matrix $\mathbf{P}$ are linked by the following coupling
\[
\mathbf{P} \mathbf{W} \mathbf{D}^{-2} = \frac{\alpha}{\|\mathbf{A}\|_{F}^2} \mathbf{I} 
\]
some  scalar relaxation parameter $\alpha > 0$.
\end{assumption}
Assumption \eqref{as:coupling} has been used in ~\cite{moorman2021randomized} for the inconsistent case (i.e. $b - \mathbf{A} \hat x \neq 0$) and for $\lambda=0$. We include a motivation for completeness. As the batch size $\eta$ goes to $\infty$ (recall that we sample with replacement) we have 
\[
\lim_{\eta \to \infty}\mathbf{M}_k = \EE_{i \sim p} \bigg[w_i \tfrac{e_i e_i^T}{\|a_{i}\|^2_2}\bigg] = \mathbf{P} \mathbf{W} \mathbf{D}^{-2}
\]
Therefore the averaged RSK update of Eq. \eqref{eq:rska} approaches the deterministic update:
\begin{align*}
 x_{k+1} &= (\mathbf{I} - \mathbf{A}^T\mathbf{P} \mathbf{W} \mathbf{D}^{-2}\mathbf{A})x_k + \mathbf{A}^T\mathbf{P} \mathbf{W} \mathbf{D}^{-2}b   \\
 x_{k+1} - \hat x &= (\mathbf{I} - \mathbf{A}^T\mathbf{P} \mathbf{W} \mathbf{D}^{-2}\mathbf{A})(x_k - \hat x) + \mathbf{A}^T\mathbf{P} \mathbf{W} \mathbf{D}^{-2}(b - \mathbf{A} \hat x)
\end{align*}
In order to have that $(x_{k+1} - \hat x)$ goes to zero in the limit we
should require that this limiting error update has the zero vector as a fixed point, i.e.
\[
0 = \mathbf{A}^T\mathbf{P} \mathbf{W} \mathbf{D}^{-2}(b - \mathbf{A} \hat x)
\]
This is guaranteed if $\mathbf{P} \mathbf{W} \mathbf{D}^{-2} = \beta \mathbf{I}$. But for $\lambda \neq 0$, we do not have a bound of the form $D_f^{x_k^*}(x_k, \hat x) \leq \gamma \cdot \|Ax_k - b\|^2_{\mathbf{P} \mathbf{W} \mathbf{D}^{-2}}$ that is the reason why for that case we need to assume $\mathbf{P} \mathbf{W} \mathbf{D}^{-2} = \beta \mathbf{I}$.

Since in this case $\mathbf{P}\mathbf{W}^2\mathbf{D}^{-2} = \mathbf{P}\mathbf{W}\mathbf{D}^{-2}\mathbf{W}
$, Lemma~\ref{lemma1} becomes:
\begin{lemma}
\label{lemma2}
Let $\mathbf{M}_k, \mathbf{P}, \mathbf{W}$ and $\mathbf{D}$ be defined as in Definition \ref{def:definition1} and let Assumption~\ref{as:coupling} hold. Then
\begin{equation*}
    \EE_k \left[ \mathbf{M}_k\right] =  \frac{\alpha \mathbf{I}}{\|\mathbf{A}\|^2_F}
\end{equation*}
and 
\begin{equation*}
    \EE_k \left[ \mathbf{M_k^TAA^TM_k} \right] = \tfrac{1}{\eta}  \frac{\alpha \mathbf{W}}{\|\mathbf{A}\|^2_F} + \alpha^2 (1-\tfrac{1}{\eta})  \frac{\mathbf{AA^T}}{\|\mathbf{A}\|^4_F}
\end{equation*}
\end{lemma}

\begin{lemma}
\label{lm:bd}
Under Assumption \ref{as:coupling}, for the iterates $x_k$ of Algorithm \ref{alg:RSKA}, it holds that:
\begin{align*}
	\mathbb E_k \left[ D_f^{x_{k+1}^*}(x_{k+1},\hat x) \right] \leq D_f^{x_k^*}(x_k,\hat x) - \frac{\alpha }{\|\mathbf{A}\|^2_F} (1 - \sigma_{\max}(\mathbf{T})) \|\mathbf{A}x_k - b\|^2_2 .
\end{align*}
with
\begin{equation*}
    \mathbf{T} = \tfrac{1}{2\eta}  \mathbf{W} + \tfrac{\alpha}{2} (1-\tfrac{1}{\eta})  \frac{\mathbf{A} \mathbf{A}^T}{\|\mathbf{A}\|^2_F}
\end{equation*}
\end{lemma}

\begin{proof}
Using Lemma \ref{lm:mp} with $f(x) = \lambda \|x\|_1 + \frac{1}{2}\|x\|_2^{2}$ and $\varPhi(x) = \langle \mathbf{A}^T \mathbf{M}_k(\mathbf{A} x_k - b), x - x_k \rangle$, $y= \hat x$, it holds that:
\begin{align*}
 & D^{x_{k+1}^{*}}_{f}(x_{k+1}, \hat x) \\
 & \leq D^{x_k^{*}}_{f}(x_k, \hat x) +  \varPhi(\hat x) - \varPhi(x_{k+1}) - D^{x_k^{*}}_{f}(x_k, x_{k+1})  \\
 &= D^{x_k^{*}}_{f}(x_k, \hat x) -   \langle \mathbf{A}^T \mathbf{M}_k(\mathbf{A} x_k - b), x_k - \hat x \rangle + \langle \mathbf{A}^T \mathbf{M}_k(\mathbf{A} x_k - b),  x_k - x_{k+1} \rangle - D^{x_k^{*}}_{f}(x_k, x_{k+1})  \\
 & \leq D^{x_k^{*}}_{f}(x_k, \hat x) -   \langle  \mathbf{A}^T \mathbf{M}_k(\mathbf{A} x_k - b), x_k - \hat x \rangle + \| \mathbf{A}^T \mathbf{M}_k(\mathbf{A} x_k - b)\| \cdot \|x_k - x_{k+1}\| - \tfrac{1}{2}\|x_k - x_{k+1}\|^2 \\
 & \leq D^{x_k^{*}}_{f}(x_k, \hat x) -   \langle  \mathbf{A}^T \mathbf{M}_k(\mathbf{A} x_k - b), x_k - \hat x \rangle + \tfrac{1}{2}\| \mathbf{A}^T \mathbf{M}_k(\mathbf{A} x_k - b)\|^2
\end{align*}
We have:
\begin{align*}
	\EE_k \left[\Bigl \langle \mathbf{A}^T \mathbf{M}_k(\mathbf{A} x_k - b), x_k - \hat x \Bigr\rangle \right] &= \EE_k \left[\Bigl \langle \mathbf{M}_k \cdot (\mathbf{A} x_k - b) ,  \mathbf{A} x_k - b \Bigr\rangle \right]   \\
	&\overset{\text{\tiny Lemma~\ref{lemma1}}}{=} \Bigl\langle  \mathbf{A} x_k - b ,  \mathbf{A} x_k - b \Bigr\rangle_{\mathbf{P W D}^{-2}}  \overset{\text{\tiny Lemma~\ref{lemma2}}}{=} \frac{\alpha }{\|\mathbf{A}\|^2_F} \|\mathbf{A} x_k - b\|^2_2
\end{align*}	
and with $\mathbf{T} = \tfrac{1}{2\eta}  \mathbf{W} + \tfrac{\alpha}{2} (1-\tfrac{1}{\eta})  \frac{\mathbf{A} \mathbf{A}^T}{\|\mathbf{A}\|^2_F}$ we get 
\begin{align*}
\footnotesize
	\EE_k \left[ \Bigl\| \mathbf{A}^T \cdot \mathbf{M}_k \cdot (\mathbf{A} x_k - b) \Bigr\|_2^2  \right] &= \Bigl\langle  \mathbf{A} x_k - b ,  \EE_k \left[ \mathbf{M}_k^T \mathbf{A} \mathbf{A}^T \mathbf{M}_k \right]\cdot (\mathbf{A} x_k - b) \Bigr\rangle \\
	&\overset{\text{\tiny Lemma~\ref{lemma2}}}{=} \Bigl\langle   \mathbf{A} x_k - b ,  \Bigg(\tfrac{1}{\eta}  \frac{\alpha \mathbf{W}}{\|\mathbf{A}\|^2_F} + \alpha^2 (1-\tfrac{1}{\eta})  \frac{\mathbf{A} \mathbf{A}^T}{\|\mathbf{A}\|^4_F}\Bigg)\cdot (\mathbf{A} x_k - b) \Bigr\rangle\\
	&= \tfrac{2\alpha}{\|\mathbf{A}\|^2_F}\Bigl\langle  \mathbf{A} x_k - b ,   \mathbf{A} x_k - b \Bigr\rangle_{\mathbf{T}} \\
	&\leq \tfrac{2\alpha}{\|\mathbf{A}\|^2_F} \sigma_{\max}(\mathbf{T}) \|\mathbf{A} x_k - b\|^2_2.
\end{align*}
Thus combining everything together gives us	
\begin{align*}
%\footnotesize
	& \EE_k \left[ D_f^{x_{k+1}^*}(x_{k+1}, \hat x) \right] \\
	&\leq 
    D_f^{x_k^*}(x_k, \hat x) - \EE_k \left[\Bigl \langle \mathbf{A}^T \mathbf{M}_k(\mathbf{A} x_k - b), x_k - \hat x \Bigr\rangle \right] + \frac{1}{2} \EE_k \left[\Bigl\| \mathbf{A}^T \mathbf{M}_k (\mathbf{A} x_k - b) \Bigr\|_2^2 \right] \\
    &\leq  D_f^{x_k^*}(x_k, \hat x) - \frac{\alpha }{\|\mathbf{A}\|^2_F} \|\mathbf{A} x_k - b\|^2_2 + \tfrac{1}{2} \tfrac{2 \alpha}{\|\mathbf{A}\|^2_F} \sigma_{\max}(\mathbf{T}) \|\mathbf{A} x_k - b\|^2_2 \\
    &\leq   D_f^{x_k^*}(x_k, \hat x) - \frac{\alpha }{\|\mathbf{A}\|^2_F} (1 - \sigma_{\max}(\mathbf{T})) \|\mathbf{A} x_k - b\|^2_2.
\end{align*}
\end{proof}
The following lemma gives an upper and a lower bound for the largest singular value of  $\mathbf{T}$ which will be use in the convergence of RSKA iterates.
\begin{lemma}
\label{lm:svt}
Let 
\begin{equation*}
    \mathbf{T} = \tfrac{1}{2\eta}  \mathbf{W} + \tfrac{\alpha}{2} (1-\tfrac{1}{\eta})  \frac{\mathbf{A} \mathbf{A}^T}{\|\mathbf{A}\|^2_F}
\end{equation*}
Then the largest singular value of $\mathbf{T}$ satisfies:
\begin{align*}
  \tfrac{1}{2\eta} \bigg(\sigma_{\max}(\mathbf{W}) - \tfrac{\alpha}{\|\mathbf{A}\|^2_F} (\eta-1) \sigma_{\max}^2(\mathbf{A})\bigg) & \leq \sigma_{\max}(\mathbf{T})\\
  & \leq \tfrac{1}{2\eta} \bigg(\sigma_{\max}(\mathbf{W}) + \tfrac{\alpha}{\|\mathbf{A}\|^2_F} (\eta-1) \sigma_{\max}^2(\mathbf{A})\bigg)
\end{align*}
\noindent
In addition, If $\mathbf{W} = \alpha \mathbf{I}$, then $\mathbf{T}$ is positive semi-definite and 
\begin{equation*}
    \sigma_{\max}(\mathbf{T}) = \tfrac{1}{2\eta} \bigg(\alpha + \tfrac{\alpha}{\|\mathbf{A}\|^2_F} (\eta-1) \sigma_{\max}^2(\mathbf{A})\bigg)
\end{equation*}
\end{lemma}

\begin{proof}
The first part of the proof follows easily from Theorem \ref{thm:svi}. If $\mathbf{W} = \alpha \mathbf{I}$, we have: If $\lambda$ is an eigenvalue of $\mathbf{A}\mathbf{A}^{T}$ then $\tfrac\alpha\eta + \tfrac\alpha2(1-\tfrac1\eta)\tfrac\lambda{\|\mathbf{A}\|_{F}^{2}}$ is an eigenvalue of $\mathbf{T}$. From this we deduce the last equality as well as that $\mathbf{T}$ is positive semi-definite.
\end{proof}

\subsection{General Convergence Result}
In this part, we present general convergence results for the iterates from RSKA method.
\begin{theorem}[Noiseless case]
\label{th:RSKA} Consider $\eta>1$, $\gamma$ as defined in equation \eqref{eqn:Gamma} (Lemma \eqref{lm:Gamma_for_Bregman_Ax_bound}), let Assumption~\ref{as:coupling} hold and assume that 
\begin{align}\label{eq:upper-bound-alpha}
0<\alpha < 2\frac{(\eta - \tfrac12 \sigma_{\max}(\mathbf{W}))\|\mathbf{A}\|_{F}^{2}}{\sigma_{\max}(\mathbf{A})^{2}(\eta-1)}.
\end{align}Then the random iterates $x_k$ produced by Algorithm \ref{alg:RSKA} converge in expectation  with a linear rate to the unique solution $\hat x$ of $\operatorname*{\min}_{\mathbf{A}x=b}\lambda \|x\|_1 + \frac{1}{2}\|x\|_2^{2}$, more precisely, with 
\begin{align}
\label{eq:cf_rska}
q = 1 - \frac{1}{\gamma } \cdot \frac{  L(\alpha)}{\|\mathbf{A}\|^2_F} \, \in (0,1),
\end{align}
and
\begin{align*}
    L(\alpha) = \alpha -  \tfrac{\alpha}{2\eta} \bigg( \tfrac{\alpha}{\|\mathbf{A}\|^2_F} (\eta-1) \sigma^2_{\max}(\mathbf{A}) + \sigma_{\max}(\mathbf{W})\bigg),
\end{align*}
it holds that
\begin{align*}
\EE \left[ D_f^{x_{k+1}^*}(x_{k+1},\hat x) \right] &\leq q \cdot \EE \left[ D_f^{x_k^*}(x_k,\hat x) \right] \, \\
\EE \bigg[\|x_{k}-\hat x\|_2^2 \bigg]   &\le 2 \cdot q^{k}  \cdot f(\hat{x}).
\end{align*}
\end{theorem}

\begin{proof}
  Combining Lemma~\ref{lm:bd} with equation (\ref{eqn:Bregman_Ax_bound}) gives 
\begin{align*}
  \EE_k \left[ D_f^{x_{k+1}^*}(x_{k+1},\hat x) \right] &\leq D_f^{x_k^*}(x_k,\hat x)  - \tfrac{\alpha}{\|\mathbf{A}\|_{F}^{2}\gamma}(1-\sigma_{\max}(\mathbf{T}))D_f^{x_k^*}(x_k,\hat x)\\
  & \leq \left(1-\tfrac{\alpha}{\|\mathbf{A}\|_{F}^{2}\gamma}(1-\sigma_{\max}(\mathbf{T}))\right)D_f^{x_k^*}(x_k,\hat x).
\end{align*}
Using the rule of total expectation we get
\begin{align*}\EE \left[ D_f^{x_{k+1}^*}(x_{k+1},\hat x) \right] &\leq \left(1- \tfrac{1}{\gamma} \cdot \tfrac{\alpha \cdot (1-\sigma_{\max}(\mathbf{T}))}{\|\mathbf{A}\|_{F}^{2}}\right)\EE \left[ D_f^{x_k^*}(x_k,\hat x) \right].
\end{align*}
To get a rate $q\in (0,1)$, we need that $(1-\sigma_{\max}(\mathbf{T}))>0$, i.e. $\sigma_{\max}(\mathbf{T})<1$ which holds true from \eqref{eq:upper-bound-alpha}.
From Lemma \ref{lm:svt}, it follows that
$L(\alpha) \leq \alpha (1 - \sigma_{\max}(\mathbf{T}))$
%:
%\begin{multline*}
%  1-\tfrac{1}{2\eta} \bigg(\sigma_{\max}(\mathbf{W}) + \tfrac{\alpha}{\|\mathbf{A}\|^2_F} (\eta-1) \sigma_{\max}^2(\mathbf{A})\bigg)\\ \leq 1 - \sigma_{\max}(\mathbf{T})\\
%   \leq 1 +  \tfrac{1}{2\eta} \bigg( \tfrac{\alpha}{\|\mathbf{A}\|^2_F} (\eta-1) \sigma_{\max}^2(\mathbf{A}) - \sigma_{\max}(\mathbf{W})\bigg)
%\end{multline*}
and thus we  get $\EE \left[ D_f^{x_{k+1}^*}(x_{k+1},\hat x) \right] \leq q \cdot \EE \left[ D_f^{x_k^*}(x_k,\hat x) \right]$, with $q = 1 - \frac{1}{\gamma } \cdot \frac{  L(\alpha)}{\|\mathbf{A}\|^2_F}$.
 The inequality in terms of $\norm[2]{x_k - \hat x}^2$ is obtained by using the first inequality of equation (\ref{eq:D}) since $f$ is $1$-strongly convex.
\end{proof}

From~\eqref{eq:cf_rska} we see that we want to choose $\alpha$ such that $L(\alpha)$ is as large as possible:
\begin{cor}
\label{cor:linear_convergence}
Let Assumption \ref{as:coupling}  hold true. Then the relaxation parameter $\alpha$ and the constant $L$ which yields the fastest convergence rate guarantee in Theorem \ref{th:RSKA} are as follows:
\begin{enumerate}
    \item General Weights: If $\eta> \max(1,\sigma_{\max}(\mathbf{W})/2)$ then
        \begin{align*}
    \alpha^{*} = 
\frac{\|\mathbf{A}\|^2_F}{\sigma^2_{\max}(\mathbf{A})(\eta-1)}(\eta - \tfrac{\sigma_{\max}(\mathbf{W})}{2}),
\end{align*}
and
\begin{align*}
L(\alpha^{*}) =
\frac{\|\mathbf{A}\|^2_F}{8\sigma^2_{max}(\mathbf{A})} \cdot \frac{(2\eta - \sigma_{max}(\mathbf{W}))^2}{\eta(\eta-1)}.
\end{align*}
     \item Uniform Weights i.e $\mathbf{W} = \alpha \mathbf{I}$: 
     \begin{equation*}
   \alpha^{*} = \frac{\eta}{1 + (\eta-1)\tfrac{\sigma^2_{\max}(\mathbf{A})}{\|\mathbf{A}\|^2_F}}
\end{equation*}
and
\begin{equation*}
    L(\alpha^{*}) = \frac{\eta}{2 + 2(\eta-1)\tfrac{\sigma^2_{\max}(\mathbf{A})}{\|\mathbf{A}\|^2_F}}.
\end{equation*}
\end{enumerate}
\end{cor}
\begin{proof}
  In the case (a) of general weights, in order to get the tightest lower bound, we maximized the concave function $L(\alpha)$  and obtain
  \begin{align*}
  \alpha = \|\mathbf{A}\|_{F}^{2}(\eta - \sigma_{\max}(\mathbf{W})/2)/((\eta-1)\sigma_{\max}(\mathbf{A})^{2}
  \end{align*}
  which fulfills~\eqref{eq:upper-bound-alpha} and thus gives the best $q$ is Theorem~\ref{th:RSKA}. Since we need $\alpha\geq 0$ we have to assume $\eta  \geq \sigma_{\max}(\mathbf{W})/2$.
This gives $\alpha^{*}$ and plugging this into $L(\alpha)$ give us $L(\alpha^{*})$. In the case (b) of uniform weights, we maximized $L(\alpha)$ for all $\eta$ with $\mathbf{W} = \alpha \mathbf{I}$.
\end{proof}

A few remarks about the interpretation of the above theorem are in order:
\begin{remark}[Overrelaxation]
\label{rmk:rmk_lc}
\begin{itemize}
\item When a single thread $\eta = 1$ is used in the case (b) of uniform weight, we see that our optimal relaxation parameter
  is $\alpha^{*} = 1$. Whereas, when multiple threads $\eta > 1$ are used, we see
  \begin{equation*}
    1 < \alpha^{*} \leq \eta,
  \end{equation*}
  i.e. the method allows for large over-relaxation when the number of threads is high and we will see in Section~\ref{sec:effect-relaxation} that this does indeed lead to faster convergence.

  \item Our relaxation parameter $\alpha^{*}$ in the case of uniform weights is the same as the relaxation parameter $\alpha^{RT}$ suggested in~\cite{richtarik2020stochastic} although they do not treat the sparse case and only consider uniform weights.

  \item Finally, for $\eta=1$ in the case of general weights we have, due to the coupling in Assumption~\ref{as:coupling}, that the weights fulfill $w_{i} = \tfrac{\alpha\|a_{i}\|^{2}}{p_{i}\|\mathbf{A}\|_{F}^{2}}$. We could estimate $\sigma_{\max}(\mathbf{W}) \leq \tfrac{\alpha}{\min_{i}p_{i}}$ but this would be a quite crude estimate. By choosing the classical probabilities $p_{i} = \|a_{i}\|^{2}/\|\mathbf{A}\|_{F}^{2}$ we would get $\sigma_{\max}(\mathbf{W}) = \alpha$ and get that the relaxation parameter need to fulfill $\alpha\in (0,2)$ and that $\alpha^{*}=1$ is the optimal relaxation parameter. 
\end{itemize}
\end{remark}

\begin{remark}[Relation to standard randomized sparse Kaczmarz]
  In the case $\mathbf{W} = \mathbf{I}, \eta=1, \alpha = 1$, we have $\mathbf{T} = \tfrac{1}{2}\mathbf{I}$ and we recover the rate of the standard RSK.
  It holds
  \begin{align*}
    \EE_k \left[ D_f^{x_{k+1}^*}(x_{k+1},\hat x) \right] \leq D_f^{x_k^*}(x_k,\hat x) - \frac{1 }{2\|\mathbf{A}\|^2_F} \|\mathbf{A}x_k - b\|^2_2 
  \end{align*}
  which is obtained in \cite{SL19} and it is shown that this leads to a linear convergence rate in expectation,
  \begin{equation}
    \label{eq:cv_rsk}
    \EE \bigg[\|x_{k}-\hat x\|_2^2 \bigg]   \le 2 \cdot \bigg(1 - \frac{1 }{2\gamma\|\mathbf{A}\|^2_F}\bigg)^{k}  \cdot f(\hat{x}).
  \end{equation}
  which implies that we reach accuracy $\EE \bigg[\|x_{k}-\hat x\|_2^2 \bigg]   \le 2 \cdot \epsilon  \cdot f(\hat{x})$ in at most $k \geq 2\gamma \|A\|^2_F \log(\tfrac{1}{\epsilon})$ iterations.
\end{remark}
\begin{remark}[Convergence rate for the linearized Bregman method]
  Similarly to Lemma \ref{lm:bd} we can show that the linearized Bregman algorithm~\cite{cai2009linearized,yin2010analysis,LSW14}
  \begin{equation}
    \begin{aligned}
      x_{k+1}^* &= x_k^* - \frac{A^T(Ax_k - b)}{\|A\|^2_2},\\
      x_{k+1} &= S_{\lambda}(x_{k+1}^*)
    \end{aligned}
  \end{equation}
  has linear convergence rate given by
  \begin{equation}
    \label{eq:cv_lb}
    \big\|x_{k}-\hat x\big\|_2^2   \le 2 \cdot \bigg(1 - \frac{1 }{2\gamma\|\mathbf{A}\|^2_2}\bigg)^{k}  \cdot f(\hat{x}).
  \end{equation}
  Although we suspect that this result (\ref{eq:cv_lb}) is not new we could not find it in the literature.
\end{remark}

Inspired by~\cite{richtarik2020stochastic}, the following remarks apply to the uniform weight case. 
\begin{remark}[Mini-batch vs. full-batch]
  Let $H(\eta)= \frac{1}{L(\alpha^{*})} = \frac{2}{\eta} + 2(1-\frac{1}{\eta})\tfrac{\sigma^2_{\max}(\mathbf{A})}{\|\mathbf{A}\|^2_F} $ be the inverse of $L(\alpha^{*})$. Recall from Theorem~\ref{th:RSKA} that $L(\alpha^{*})$ influences the convergence rate: The larger $L(\alpha^{*})$, the faster the convergence.

  Since $\frac{\|\mathbf{A}\|^2_F}{\sigma^2_{\max}(\mathbf{A})} \geq 1$, $H$ is a nonincreasing function of $\eta$ and we have $H(1) = 2 $, $H(\infty) \eqdef \lim_{\eta \rightarrow \infty} H(\eta) =  \tfrac{2\sigma^2_{\max}(\mathbf{A})}{\|\mathbf{A}\|^2_F}$.  In the asymptotic regime $\eta \rightarrow \infty,$ Algorithm \ref{alg:RSKA} becomes linearized Bregman algorithm for minimizing (\ref{eq:bpp}), and $H(\infty)$ is the rate of linearized Bregman cf. (\ref{eq:cv_lb}). This shows that the averaging method interpolates between the basic method and the linearized Bregman. By increasing $\eta$, the quantity $\frac{H(1)}{H(\infty)} = \frac{\|\mathbf{A}\|^2_F}{\sigma^2_{\max}(\mathbf{A})}$ controls the maximum (guaranteed) speedup in
the iteration complexity achievable. Comparing (\ref{eq:cf_rska}) with the convergence rate (\ref{eq:cv_rsk}) of the basic sparse Kaczmarz method, we get an improvement of $2L(\alpha^{*}) >1$ which shows that for the RSKA algorithm, we can get a speed-up even of order approximately $\frac{\|\mathbf{A}\|^2_F}{\sigma^2_{\max}(\mathbf{A})}$ compared to the rate of the basic sparse Kaczmarz algorithm (the also the comparison in Table~\ref{tab:complexities}).

For $\eta \geq \frac{\|\mathbf{A}\|^2_F}{\sigma^2_{\max}(\mathbf{A})},$ we get $H(\eta) \leq 2 H(\infty)$ , which is the performance of the full batch (up to a factor of 2). This means that it does not make sense to use a minibatch size larger than $\frac{\|\mathbf{A}\|^2_F}{\sigma^2_{\max}(\mathbf{A})}$. Moreover, notice that $H(\eta) \geq \frac{1}{\eta} H(1)$ for all $\eta$, show that the number of iterations does not decrease linearly in the minibatch size $\eta$. From a total complexity perspective cf. Table~\ref{tab:complexities}, this also means that in a computational regime where processing $\eta$ basic method
updates costs $\eta$ times as much as processing a single update, the decrease in iteration complexity cannot compensate for the increase in cost per iteration, which means that the choice $\eta = 1$ is optimal. On the other hand, if a parallel processor is available, a larger $\eta$ will be better.

\end{remark}

\begin{table}[htb]
  \centering
  \begin{tabular}{lccc}\toprule
    & RSK & RSKA & linBreg\\\midrule
    iteration complexity & $\mathcal{O} \bigg(2 \gamma \|\mathbf{A}\|^2_F \log (\tfrac{1}{\epsilon}) \bigg)$ &\quad   \quad $\mathcal{O} \bigg( \gamma \frac{\|\mathbf{A}\|^2_F}{L(\alpha^{*})} \log (\tfrac{1}{\epsilon}) \bigg)$ & $\mathcal{O} \bigg( 2 \gamma \|\mathbf{A}\|^2_2 \log (\tfrac{1}{\epsilon}) \bigg)$\\
    cost per iteration &  $\mathcal{O} (n )$ \quad& \quad$\mathcal{O} ( \eta n  )$ & $\mathcal{O} ( m n )$\\\bottomrule
  \end{tabular}
  \caption{Complexity of different methods.}
  \label{tab:complexities}
\end{table}

\subsection{Noisy right hand sides}
\label{sec:noisy-rhs}

In the noisy case we consider a consistent linear system $\mathbf{A}x=b$, but assume that instead of $b$ we only have access to some vector $b^{\delta}$ with $\norm[2]{b-b^{\delta}}\leq\delta$.
In the context if Kaczmarz methods, such errors in the right hand side have been considered in~\cite{needell2010randomized,zouzias2013randomized}. Since the system $\mathbf{A}x=b^{\delta}$ is most likely inconsistent, RSKA will not solve the optimization problem~\eqref{eq:RBPP} but hopefully the iterates still come close to the solution. If we assume a bound $\norm[2]{b-b^{\delta}}\leq \delta$ for $b$ with $\mathbf{A}\hat x = b$ we get the following result on the convergence of the method:
\begin{theorem}[Noisy case]
\label{th:RSKA_soisy}
Assume that instead of exact data $b \in \mathbf{Range}(\mathbf{A})$ only a noisy right hand side $b^{\delta} \in \RR^m$ with $\| b^{\delta} - b \|_2 \leq \delta$ is given. Consider $\eta>1$, $\epsilon > 0$, $\gamma$ as defined in equation \eqref{eqn:Gamma} (Lemma \eqref{lm:Gamma_for_Bregman_Ax_bound}), let Assumption~\ref{as:coupling} hold and assume that \begin{align}\label{eq:upper-bound-alpha_noisy}
0<\alpha < 2\frac{((1-\epsilon)\eta - \tfrac12 \sigma_{\max}(\mathbf{W}))\|A\|_{F}^{2}}{\sigma_{\max}(\mathbf{A})^{2}(\eta-1)}.
\end{align}If the iterates $x_k$ of the RSKA method from Algorithm \ref{alg:RSKA} are computed with $b$ replaced by $b^{\delta}$, then, with the contraction factor $a$, we have :
\begin{align}
\label{eq:cf_rska_noisy}
a = 1 - \frac{\alpha}{\gamma } \cdot \frac{  (1-\epsilon-\sigma_{\max}(\mathbf{T}))}{\|\mathbf{A}\|^2_F} \, \in (0,1),
\end{align}
and the expected rate of convergence is
\begin{align*}
\EE \left[ D_f^{x_{k+1}^*}(x_{k+1},\hat x) \right] &\leq a \cdot \EE \left[ D_f^{x_k^*}(x_k,\hat x) \right] + c \delta^2 \,\\
\EE \bigg[\|x_{k}-\hat x\|_2^2 \bigg]   & \le 2 \cdot a^{k}  \cdot f(\hat{x}) + \frac{c}{1-a} \delta^2.
\end{align*}
where
%\begin{align*}
    $c = \tfrac{\alpha}{\|A\|^2_F}\bigg(\sigma_{\max}(\mathbf{T}) + \tfrac{1}{\epsilon}\sigma_{\max}^2(\mathbf{T}^{'})\bigg), \quad \mathbf{T}^{'} = \mathbf{T} - \tfrac{1}{2}\mathbf{I}$
%\end{align*}
\end{theorem}
\begin{proof}
Assuming that a noisy observed data $b^{\delta} \in \mathbb{R}^m$ instead of $b$ with  $\|b^{\delta} - b\|_2 \leq \delta$ is given, where $b = A \hat x$. The update in this case is given by:
\begin{equation}
\begin{aligned}
\label{eq:rska_with_noise}
    x^{*}_{k+1} &= x^{*}_k - \tfrac{1}{\eta} \sum_{i \in \tau_k} w_i \tfrac{\langle a_{i}, x_k \rangle - b_{i}^{\delta}}{\|a_{i}\|^2_2} \cdot a_{i},\\
    x_{k+1} &= S_{\lambda}(x^{*}_{k+1})
    \end{aligned}
\end{equation}
where $\eta = |\tau_k|$,which in terms of matrix multiplication is equal to:
\begin{equation}
\begin{aligned}
\label{eq:rska1_noisy}
    x^{*}_{k+1} &= x^{*}_k - \mathbf{A}^T \cdot \mathbf{M}_k \cdot (\mathbf{A} x_k - b^{\delta}),\\
    x_{k+1} &= S_{\lambda}(x^{*}_{k+1})
    \end{aligned}
\end{equation}
We introduce the abbreviation 
\begin{equation*}
    x_k^{\delta} \eqdef \hat x + \mathbf{A}^T \cdot \mathbf{M}_k \cdot (b - b^{\delta})
\end{equation*}
and use Lemma \ref{lm:mp} with $f(x) = \lambda \|x\|_1 + \frac{1}{2}\|x\|_2^{2}$ and $\varPhi(x) = \langle \mathbf{A}^T \mathbf{M}_k(\mathbf{A} x_k - b^{\delta}), x - x_k \rangle$, and $y=x_{k}^{\delta}$ and  get
\begin{align*}
  D^{x_{k+1}^{*}}_{f}(x_{k+1}, x_{k}^{\delta}) 
 & \leq D^{x_k^{*}}_{f}(x_k, x_{k}^{\delta}) +  \varPhi(x_{k}^{\delta}) - \varPhi(x_{k+1}) - D^{x_k^{*}}_{f}(x_k, x_{k+1})  \\
  &= D^{x_k^{*}}_{f}(x_k, x_{k}^{\delta}) -   \langle \mathbf{A}^T \mathbf{M}_k(\mathbf{A} x_k - b^{\delta}), x_k - x_{k}^{\delta} \rangle \\
  & \qquad+ \langle \mathbf{A}^T \mathbf{M}_k(\mathbf{A} x_k - b^{\delta}),  x_k - x_{k+1} \rangle - D^{x_k^{*}}_{f}(x_k, x_{k+1})  \\
  & \leq D^{x_k^{*}}_{f}(x_k, x_{k}^{\delta}) -   \langle \mathbf{A}^T \mathbf{M}_k(\mathbf{A} x_k - b^{\delta}), x_k - x_{k}^{\delta} \rangle \\
  & \qquad + \|\mathbf{A}^T \mathbf{M}_k(\mathbf{A} x_k - b^{\delta})\| \cdot \|x_k - x_{k+1}\| - \tfrac{1}{2}\|x_k - x_{k+1}\|^2 \\
 & \leq D^{x_k^{*}}_{f}(x_k, x_{k}^{\delta}) -   \langle \mathbf{A}^T \mathbf{M}_k(\mathbf{A} x_k - b^{\delta}), x_k - x_{k}^{\delta} \rangle + \tfrac{1}{2}\|\mathbf{A}^T \mathbf{M}_k(\mathbf{A} x_k - b^{\delta})\|^2
 \vspace{-0.8cm}
\end{align*}
so that
\begin{align}\label{eq:est-bregman-noisy-1}
	D_f^{x_{k+1}^*}(x_{k+1},x_k^{\delta}) &\leq 
    D_f^{x_k^*}(x_k,x_k^{\delta}) - \langle \mathbf{A}^T \mathbf{M}_k (\mathbf{A} x_k - b^{\delta}) , x_k - x_k^{\delta} \rangle + \frac{1}{2} \Bigl\| \mathbf{A}^T \mathbf{M}_k  (\mathbf{A} x_k - b^{\delta}) \Bigr\|_2^2 
    \vspace{-0.8cm}
\end{align}
Unfolding the expression of $ D_f^{x_{k+1}^*}(x_{k+1},x_k^{\delta})$ and $D_f^{x_{k}^*}(x_{k},x_k^{\delta}),$ we get:
\begin{align*}
\vspace{-0.8cm}
	D_f^{x_{k+1}^*}(x_{k+1},\hat x) &\leq  D_f^{x_k^*}(x_k,\hat x) - \langle \mathbf{A}^T \mathbf{M}_k (\mathbf{A} x_k - b^{\delta}) , x_k - \hat x \rangle + \frac{1}{2} \Bigl\| \mathbf{A}^T \mathbf{M}_k  (\mathbf{A} x_k  -  b^{\delta}) \Bigr\|_2^2
	\vspace{-0.8cm}
\end{align*}
\noindent
On the other hand,
\begin{align*}
	\mathbb E_k \left[ \langle \mathbf{M}_k \cdot (\mathbf{A} x_k - b^{\delta}) , \mathbf{A}x_k - b \rangle \right]
	&= \frac{\alpha }{\|\mathbf{A}\|^2_F} \|\mathbf{A}x_k - b\|^2_2 + \frac{\alpha }{\|\mathbf{A}\|^2_F} \langle b - b^{\delta} , \mathbf{A}x_k - b \rangle
	\vspace{-0.8cm}
\end{align*}
\noindent
and
\begin{align*}
	\mathbb E_k \left[ \Bigl\| \mathbf{A}^T \cdot \mathbf{M}_k \cdot (\mathbf{A} x_k - b^{\delta}) \Bigr\|_2^2  \right] &= \langle  \mathbf{A} x_k - b , \mathbb E_k \left[\mathbf{M}_k^T\mathbf{A}\mathbf{A}^T \mathbf{M}_k \right]\cdot (\mathbf{A} x_k - b) \rangle \\
	&+  \langle  b - b^{\delta} , \mathbb E_k \left[\mathbf{M}_k^T\mathbf{A}\mathbf{A}^T \mathbf{M}_k \right]\cdot (b - b^{\delta}) \rangle \\ &+ 2 \langle  \mathbf{A} x_k - b , \mathbb E_k \left[\mathbf{M}_k^T\mathbf{A}\mathbf{A}^T \mathbf{M}_k \right]\cdot (b - b^{\delta}) \rangle
\end{align*}
In summary we have:
\begin{align}\label{eq:est-bregma-noisy-2}
  \begin{split}
    & -\mathbb E_k \left[ \langle \mathbf{A}^T \cdot \mathbf{M}_k \cdot (\mathbf{A} x_k - b^{\delta}) , x_k - \hat x \rangle \right] + \frac{1}{2} \mathbb E_k \left[\Bigl\| \mathbf{A}^T \cdot \mathbf{M}_k \cdot (\mathbf{A} x_k  -  b^{\delta}) \Bigr\|_2^2 \right] \\
    &= -\frac{\alpha }{\|\mathbf{A}\|^2_F} \|\mathbf{A}x_k - b\|^2_2 - \frac{\alpha }{\|\mathbf{A}\|^2_F} \langle b - b^{\delta} , \mathbf{A}x_k - b \rangle + \tfrac{1}{2}\langle  b - b^{\delta} , \mathbb E_k \left[\mathbf{M}_k^T\mathbf{A}\mathbf{A}^T \mathbf{M}_k \right] (b - b^{\delta}) \rangle \\
    &+  \tfrac{1}{2}\langle  \mathbf{A} x_k - b , \mathbb E_k \left[\mathbf{M}_k^T\mathbf{A}\mathbf{A}^T \mathbf{M}_k \right] (\mathbf{A} x_k - b) \rangle +  \langle  \mathbf{A} x_k - b , \mathbb E_k \left[\mathbf{M}_k^T\mathbf{A}\mathbf{A}^T \mathbf{M}_k \right]\cdot (b - b^{\delta}) \rangle \\
    &\leq -\frac{\alpha }{\|\mathbf{A}\|^2_F} \|\mathbf{A}x_k - b\|^2_2 + \tfrac{\alpha}{\|\mathbf{A}\|^2_F} \sigma_{\max}(\mathbf{T}) \|\mathbf{A}x_k - b\|^2_2 + \tfrac{\alpha}{\|\mathbf{A}\|^2_F} \sigma_{\max}(\mathbf{T}) \|b - b^{\delta}\|^2_2 \\ & +  \tfrac{2\alpha}{\|\mathbf{A}\|^2_F}\langle  \mathbf{A} x_k - b , \Bigg(\tfrac{1}{2\eta}  W -\tfrac{1}{2}I + \tfrac{\alpha}{2} (1-\tfrac{1}{\eta})  \frac{AA^T}{\|\mathbf{A}\|^2_F}\Bigg)\cdot (b - b^{\delta}) \rangle \\
    &\leq -\frac{\alpha }{\|\mathbf{A}\|^2_F} \|\mathbf{A}x_k - b\|^2_2 + \tfrac{\alpha}{\|\mathbf{A}\|^2_F} \sigma_{\max}(\mathbf{T}) \|\mathbf{A}x_k - b\|^2_2 + \tfrac{\alpha}{\|\mathbf{A}\|^2_F} \sigma_{\max}(\mathbf{T}) \|b - b^{\delta}\|^2_2 \\ & +  \tfrac{\alpha \epsilon}{\|\mathbf{A}\|^2_F} \|\mathbf{A} x_k - b\|^2_2 +  \tfrac{\alpha}{\epsilon\|\mathbf{A}\|^2_F} \Bigg\| \Bigg(\tfrac{1}{2\eta}  W -\tfrac{1}{2}I + \tfrac{\alpha}{2} (1-\tfrac{1}{\eta})  \frac{\mathbf{A}\mathbf{A}^T}{\|\mathbf{A}\|^2_F}\Bigg)\cdot (b - b^{\delta}) \Bigg\|^2_2 \\
    &= - \frac{\alpha}{\|\mathbf{A}\|^2_F} (1 - \epsilon - \sigma_{\max}(\mathbf{T}))\|\mathbf{A} x_k - b\|^2_2 + \frac{\alpha}{\|\mathbf{A}\|^2_F} ( \sigma_{\max}(\mathbf{T}) + \frac{1}{\epsilon}\sigma_{\max}^2(\mathbf{T}^{'}))\|b - b^{\delta}\|^2_2.
  \end{split}
\end{align}
To get an improvement after each iteration, we need $(1 - \epsilon - \sigma_{\max}(\mathbf{T})) >0$, i.e. $\sigma_{\max}(\mathbf{T}) < 1 - \epsilon$ which hold true because of \eqref{eq:upper-bound-alpha_noisy}. Combining~\eqref{eq:est-bregman-noisy-1} and~\eqref{eq:est-bregma-noisy-2} and using the error bound from Lemma~\ref{lm:Gamma_for_Bregman_Ax_bound} we arrive at 
\begin{align*}
\mathbb E_k \left[ D_f^{x_{k+1}^*}(x_{k+1},\hat x) \right] &\leq 
                                   \left( 1 - \tfrac{\alpha}{\gamma\|\mathbf{A}\|_{F}^{2}}(1-\epsilon-\sigma_{\max}(\mathbf{T})) \right)D_f^{x_k^*}(x_k,\hat x)\\
  & \qquad + \tfrac{\alpha}{\|\mathbf{A}\|^{2}_{F}}(\sigma_{\max}(\mathbf{T}) + \tfrac1\epsilon \sigma^{2}_{\max}(\mathbf{T}'))\delta^2,
\end{align*}
which concludes the proof. The inequality in terms of the norm is obtained by using the first inequality of (\ref{eq:D}) since $f$ is $1$-strongly convex.
\end{proof}

\begin{remark}
  Note that for $\mathbf{W} = \mathbf{I}, \eta=1, \alpha = 1$, we have $\mathbf{T} = \tfrac{1}{2}\mathbf{I}$ meaning $\mathbf{T}^{'} = 0$ and sending $\epsilon$ to zero give us $\frac{c}{1-a} = \gamma$ which recover the rate of the standard RSK in the noisy case showed in~\cite{SL19}.
\end{remark}

%\begin{comment}
\section{Numerical Experiments}
\label{sec:numerical_experiments}
We present several experiments to demonstrate the effectiveness of Algorithm \ref{alg:RSKA} under various conditions. In particular, we study the effects of the
relaxation parameter $\alpha$, the number of threads $\eta$, the sparsity parameter $\lambda$, the weight matrix $\mathbf{W}$, and the probability matrix $\mathbf{P}$. The simulations were performed in \texttt{Python} on an Intel Core i7 computer with 16GB RAM. We start by comparing several variants of RSKA algorithms with randomized Kaczmarz (RK) and randomized sparse Kaczmarz (RSK). We consider the following RSKA variants:
\begin{enumerate}
    \item v1: RSKA with $\mathbf{W}=\mathbf{I}$, i.e. $\alpha=1$, with a coupling such that $\mathbf{P}\mathbf{W}\mathbf{D}^{-2}= \frac{\alpha \mathbf{I}}{\|\mathbf{A}\|^2_F}$ i.e $\,\,$ $p_i = \frac{\|a_i\|^2_2}{\|\mathbf{A}\|^2_F}$.
    \item v2: RSKA with a uniform weight matrix $\mathbf{W} = \alpha^{*} \mathbf{I}$ i.e $p_i = \frac{\|a_i\|^2_2}{\|\mathbf{A}\|^2_F}$, where $\alpha^{*}$ is from Corollary~\ref{cor:linear_convergence}. 
    \item v3: RSKA with a general diagonal weight matrix $\mathbf{W}$ with diagonal entries $w_{i}$ sample i.i.d. from the uniform distribution on $[0,1]$ and $p_i = \frac{\|a_i\|^2_2}{\|\mathbf{A}\|^2_F}$.
    \item v4: RSKA with a general diagonal weight matrix $\mathbf{W}$ with diagonal entries $w_{i}$ sample i.i.d. from the uniform distribution on $[0,1]$ and set $p_{i}$ proportional to $\tfrac{\norm{a_{i}}^{2}}{w_{i}}$, i.e. we have $\mathbf{P}\mathbf{W}\mathbf{D}^{-2} = \alpha \mathbf{I}$ with $\alpha = \big(\sum_{j}\tfrac{\norm{a_{j}}^{2}}{w_{j}}\big)^{-1}$.
    \end{enumerate}
The rationale behind these choices are: Version v1 just chooses no weight and standard probabilities. In v2 we still choose standard probabilities but use a coupling of weights and probabilities with the uniform weight $\alpha^{*}$ of which we have seen in Corollary~\ref{cor:linear_convergence} that is has a certain optimality property. In v3 we still use standard probabilites but do not use a coupling of probabilties and weights. Since we do not have any indications on how to choose weights in any optimal way, we just used random weights here. In contrast to v3, we do use a coupling of weights and probabilities in v4, but again, since we do not have results on how to choose optimal weights, we choose random ones.
    
Note that RSK and RK are both special cases of RSKA, namely
\begin{enumerate}
    \item RSKA with $\mathbf{W}=\mathbf{I}$, i.e. $\alpha=1$, with a coupling such that $\mathbf{P}\mathbf{W}\mathbf{D}^{-2}= \frac{\alpha \mathbf{I}}{\|\mathbf{A}\|^2_F}$ i.e $\,\,$ $p_i = \frac{\|a_i\|^2_2}{\|\mathbf{A}\|^2_F}$ and $\eta=1$,
    \item RSKA with $\mathbf{W}=\mathbf{I}$, i.e. $\alpha=1$, with a coupling such that $\mathbf{P}\mathbf{W}\mathbf{D}^{-2}= \frac{\alpha \mathbf{I}}{\|\mathbf{A}\|^2_F}$ i.e $\,\,$ $p_i = \frac{\|a_i\|^2_2}{\|\mathbf{A}\|^2_F}$, $\eta=1$ and $\lambda=0$.
    \end{enumerate}

Synthetic data for the experiments is generated as follows:
All elements of the data matrix $\mathbf{A} \in \RR^{m\times n}$ are chosen independent and identically distributed from the standard normal distribution $\Ncal (0, 1)$. We constructed overdetermined, square, and underdetermined linear systems. To construct sparse solutions $\hat x \in \RR^n$ we choose $s$ indices from $\left\{ 1,\dots,n \right\}$ at random and placed zeros at these positions. and the corresponding right hand sides are $b = \mathbf{A}\hat x \in \RR^m$ while the respective noisy right hand sides are $b^{\delta}$ and are obtained by adding Gaussian noise (see Section~\ref{sec:noisy} below).
We generally chose $\eta = 1 + \frac{1}{10}\text{min}(m,n)$ for RSKA, unless something else is indicated. Note that in the overdetermined case with no noise there will be a unique solution $\hat x$ since with probability $1$ the matrices $\mathbf{A}$ have full rank, and so all methods are expected to converge to the same solution $\hat x$ in this case. 
    
For each experiment, we run independent trials each starting with the initial iterate $x_0=0$. We measure performance by plotting the relative residual error $\| \mathbf{A}x - b\|/\|b\|$ and the error $\|x_k - \hat x \|/\|\hat x\|$ against the number
of full iterations. Thick line shows mean over the total number of trials and shaded area which represent the standard deviation over the trials are plotted when appropriate. 

Figure \ref{fig:example} shows the result for a five times overdetermined and consistent system without noise where the value $\lambda = 1$ was used for RSK and RSKA. Note that the usual RK and RSKA variants performs consistently well over all trials, while the performance of RSK differs drastically between different instances. Moreover we observe experimentally that choosing the theoretically optimal overrelaxation parameter $\alpha^*$ from Corollary~\ref{cor:linear_convergence} for the RSKA v2 method give us faster convergence.

Figure \ref{fig:example1} and Figure \ref{fig:example2} shows the results respectively for a two times resp. five times underdetermined and consistent system without noise where the values $\lambda = 1$, resp. $\lambda = 3$ was used for RSK and RSKA. Methods like RSK and RSKA take advantage of the fact that the vectors $\hat x$ are very sparse. Moreover, in Figure \ref{fig:example1}, the RSK and RK methods do not reduce the residual as fast as the RSKA method. However, since the problem is underdetermined, the RK method does not converge to a sparse solution and hence, the error does not converge to zero. Figure \ref{fig:example3} and Figure \ref{fig:example4} show results for further values of $m$ and $n$ and show similar behavior as Figure \ref{fig:example1}.

% Figure \ref{fig:example13} and \ref{fig:example14} report the performance of RK, RSK and RSKA on the Mnist dataset available in the Tensorflow framework ~\cite{abadi2016tensorflow}, we randomly select a datapoint \textcolor{red}{i.e. a gray scale image in $\RR^{28 \times 28}$, we reshape it into a vector in $\RR^{784}$ and consider it as our $\hat x$. By counting the number of non zero components one can check that $\hat x$ is a sparse vector. We tried several settings with Gaussian matrix $\mathbf{A}$. The goal is to reconstruct the vector $\hat x$ using different methods. }. In Fig \ref{fig:example13}, we first used an underdetermined matrix $\mathbf{A}$ and show the 7 images which correspond to the original image and the reconstruction of different methods. \textcolor{red}{Once the reconstruction is done we reshape the last iterate of every method in 2D to be able to plot them as image}. Reconstruction for RSKA $v1$ and $v3$ are not visible because we used a large sparsity parameter for all RSKA variants and these methods returned zero as solution as final iterate. In Fig\ref{fig:example14},  we generate a square matrix $\mathbf{A}$, we used the same sparsity parameter for RSK, RSKA $v1$ and $v3$.

\begin{figure}[htb]%
%\vspace{-0.8cm}
    \centering
    \subfloat[\centering Relative Residual]{{\includegraphics[width=.45\linewidth]{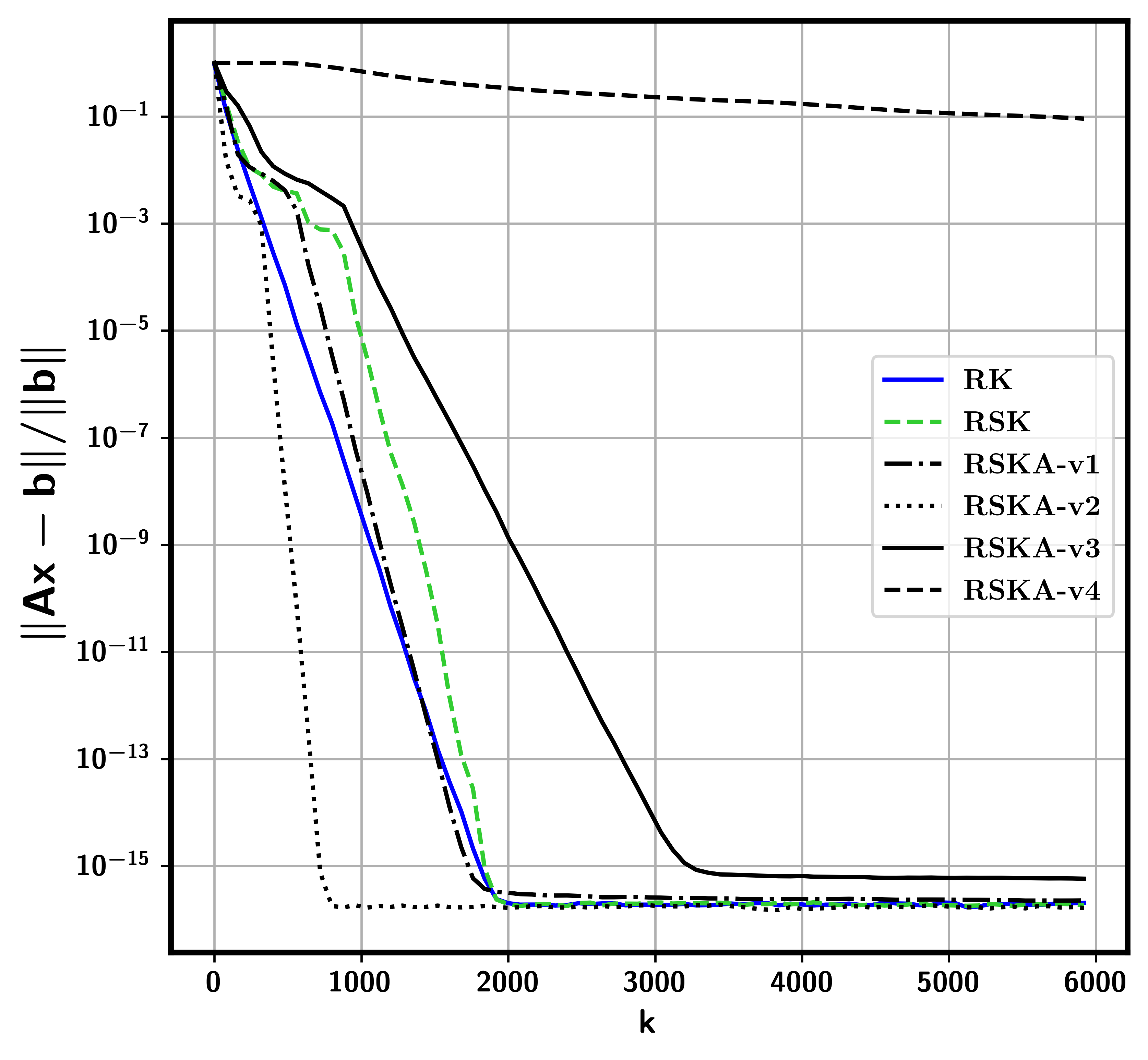} }}%
    \qquad
    \subfloat[\centering Error]{{\includegraphics[width=.45\linewidth]{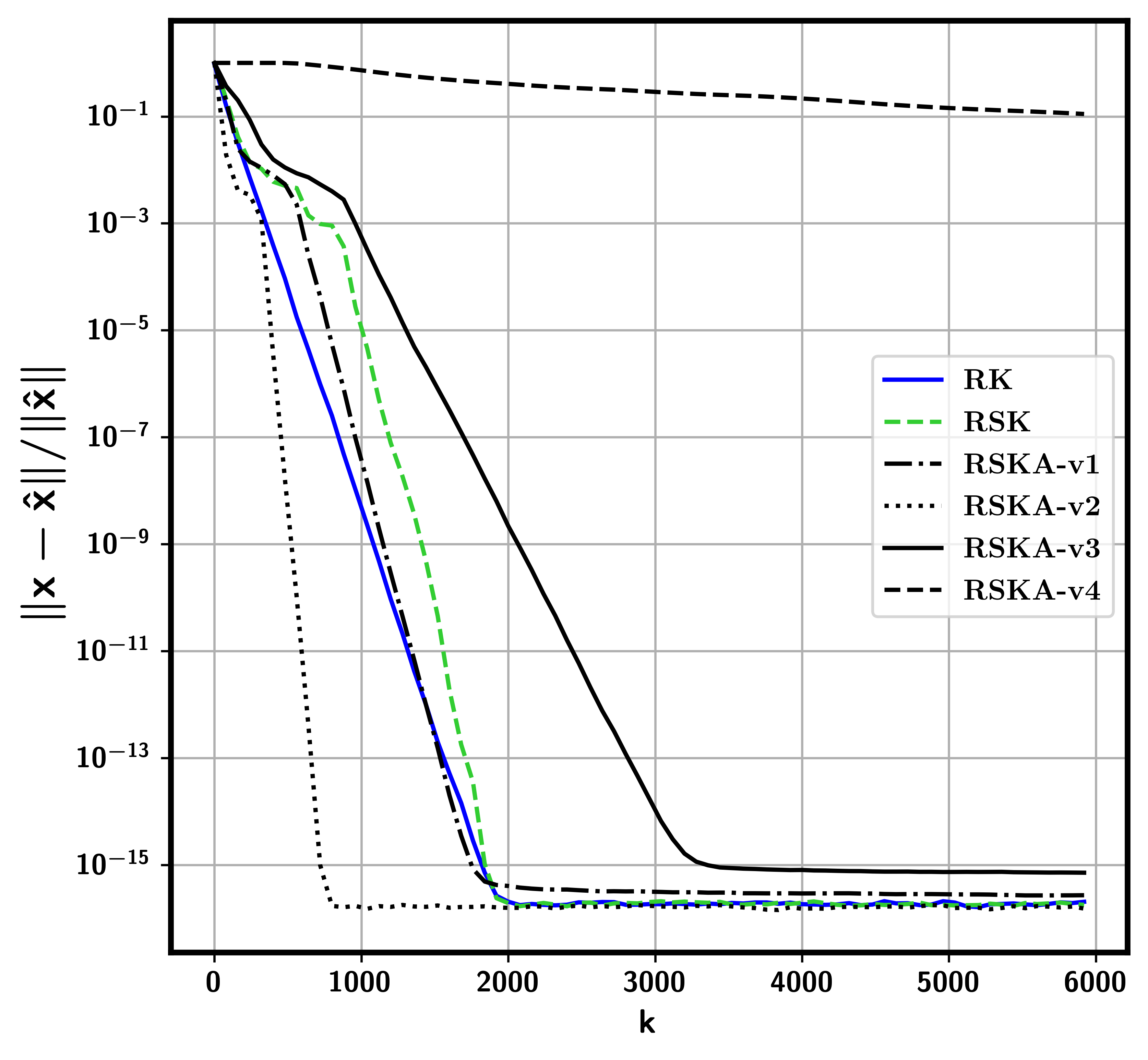} }}%
    \caption{\small A comparison of randomized Kaczmarz (blue), randomized sparse Kaczmarz (green) and RSKA method (black), $m = 100, n = 20,$ sparsity $s=10$, $\eta=11$, $\lambda=1$, no noise and 20 runs. Thick line shows mean over all trials.}%
    \label{fig:example}%
    \vspace{-0.8cm}
\end{figure}

\begin{figure}[htb]%
    \centering
    \subfloat[\centering Relative Residual]{{\includegraphics[width=.45\linewidth]{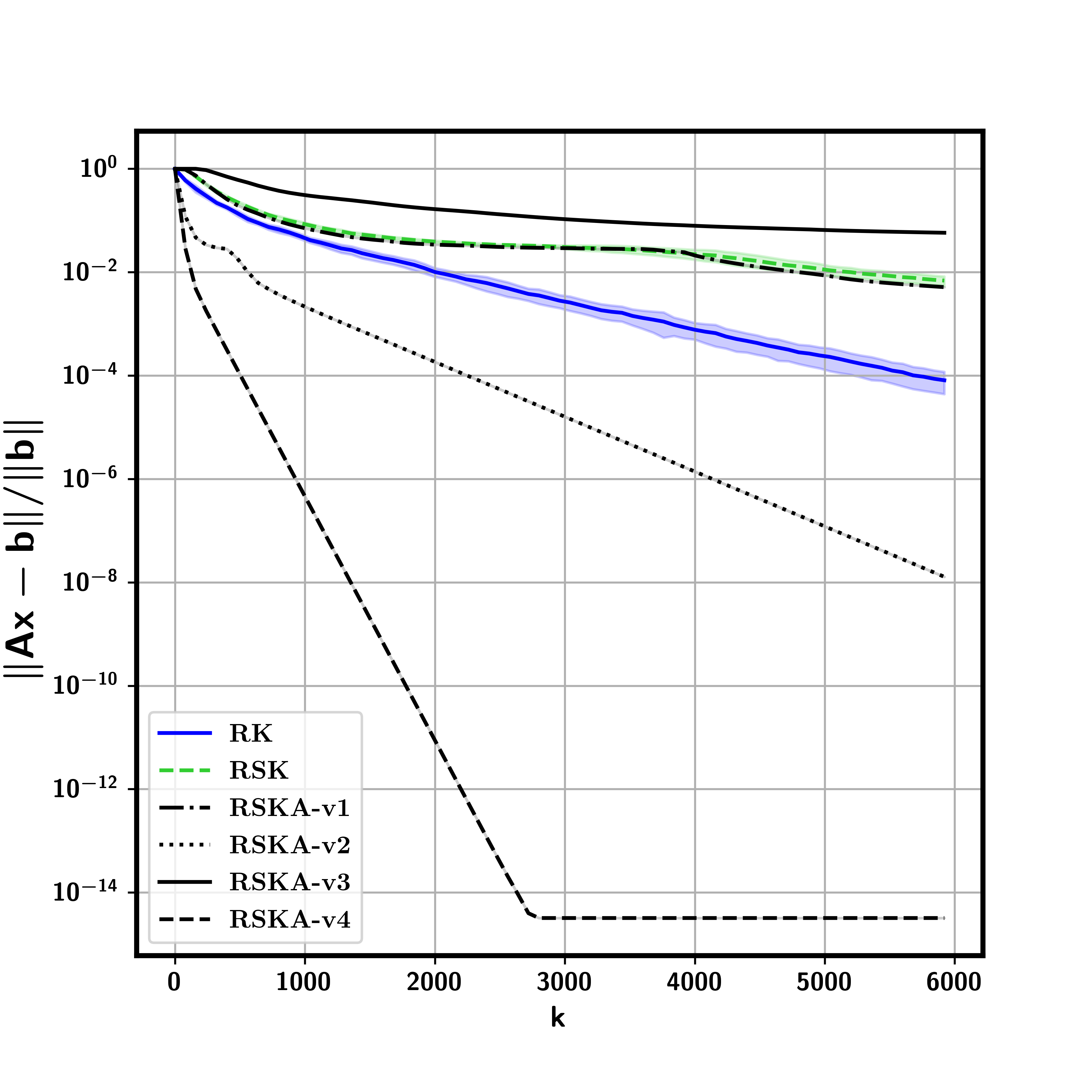} }}%
    \qquad
    \subfloat[\centering Error]{{\includegraphics[width=.45\linewidth]{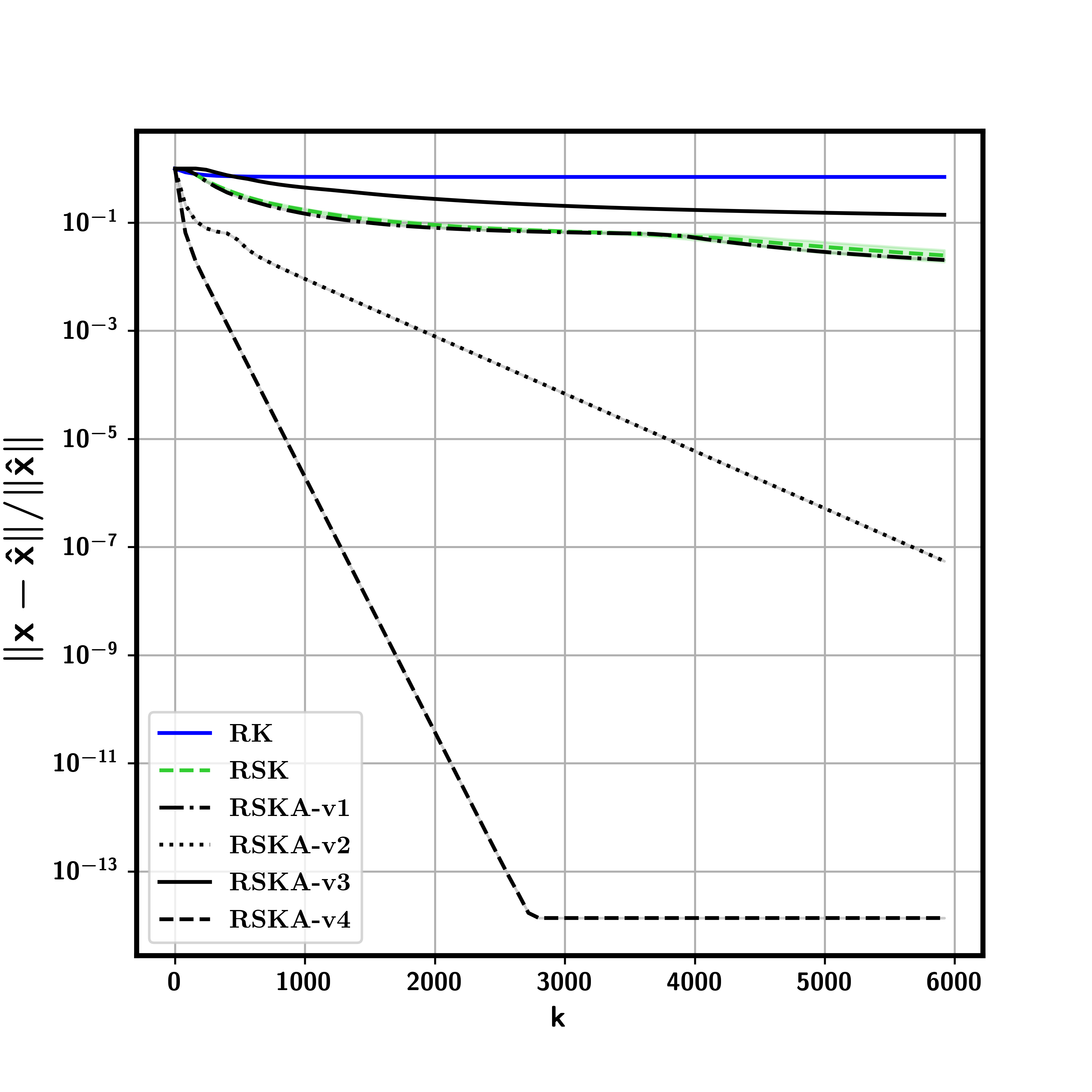} }}%
    \caption{\small A comparison of randomized Kaczmarz (blue), randomized sparse Kaczmarz (green) and RSKA method (black), $m = 100, n = 200,$ sparsity $s=10$, $\eta=11$, $\lambda=1$, no noise and 10 runs. Thick line shows mean over all trials, shaded area represent the standard deviation.}%
    \label{fig:example1}%
    \vspace{-0.8cm}
\end{figure}

\begin{figure}[htb]%
    \centering
    \subfloat[\centering Relative Residual]{{\includegraphics[width=.45\linewidth]{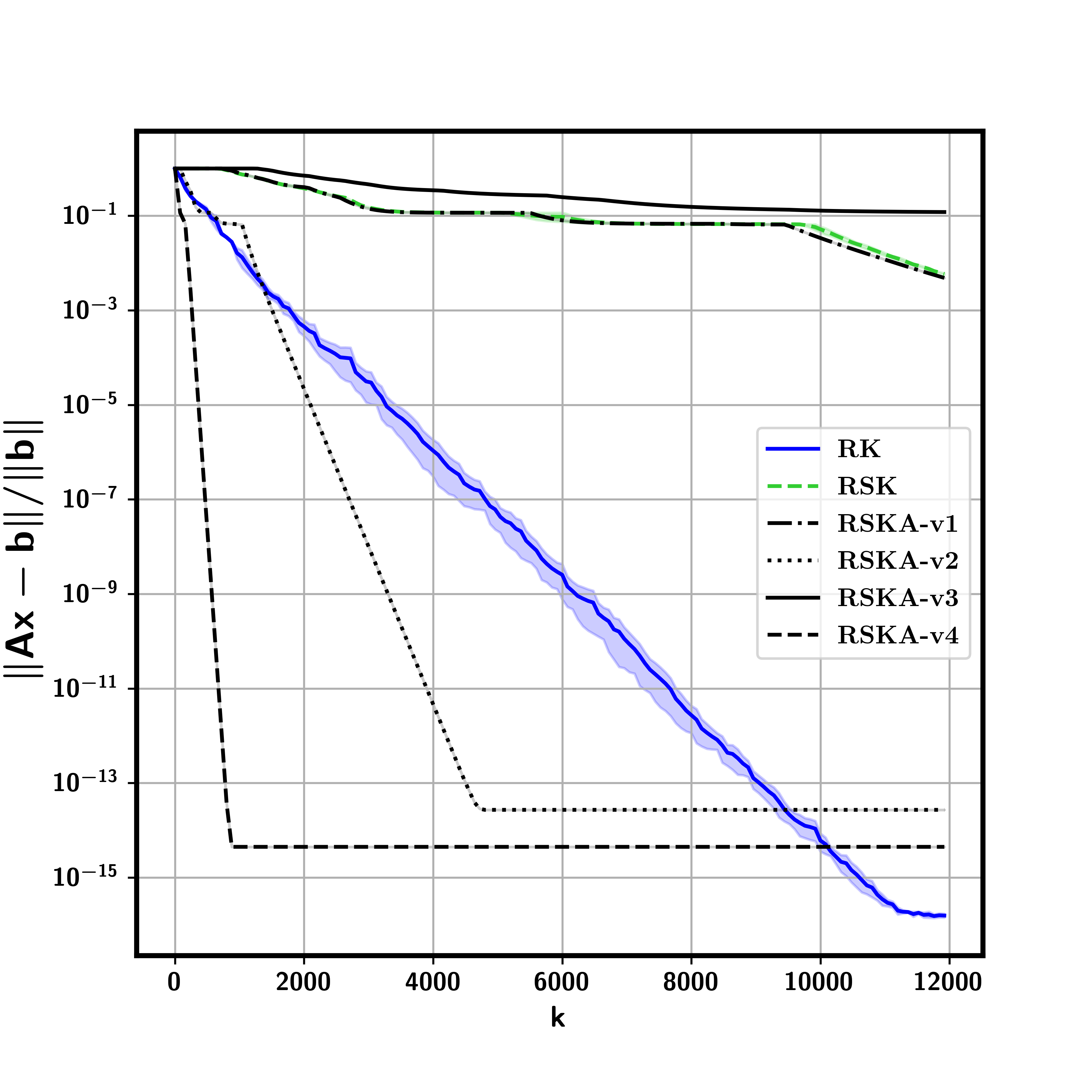} }}%
    \qquad
    \subfloat[\centering Error]{{\includegraphics[width=.45\linewidth]{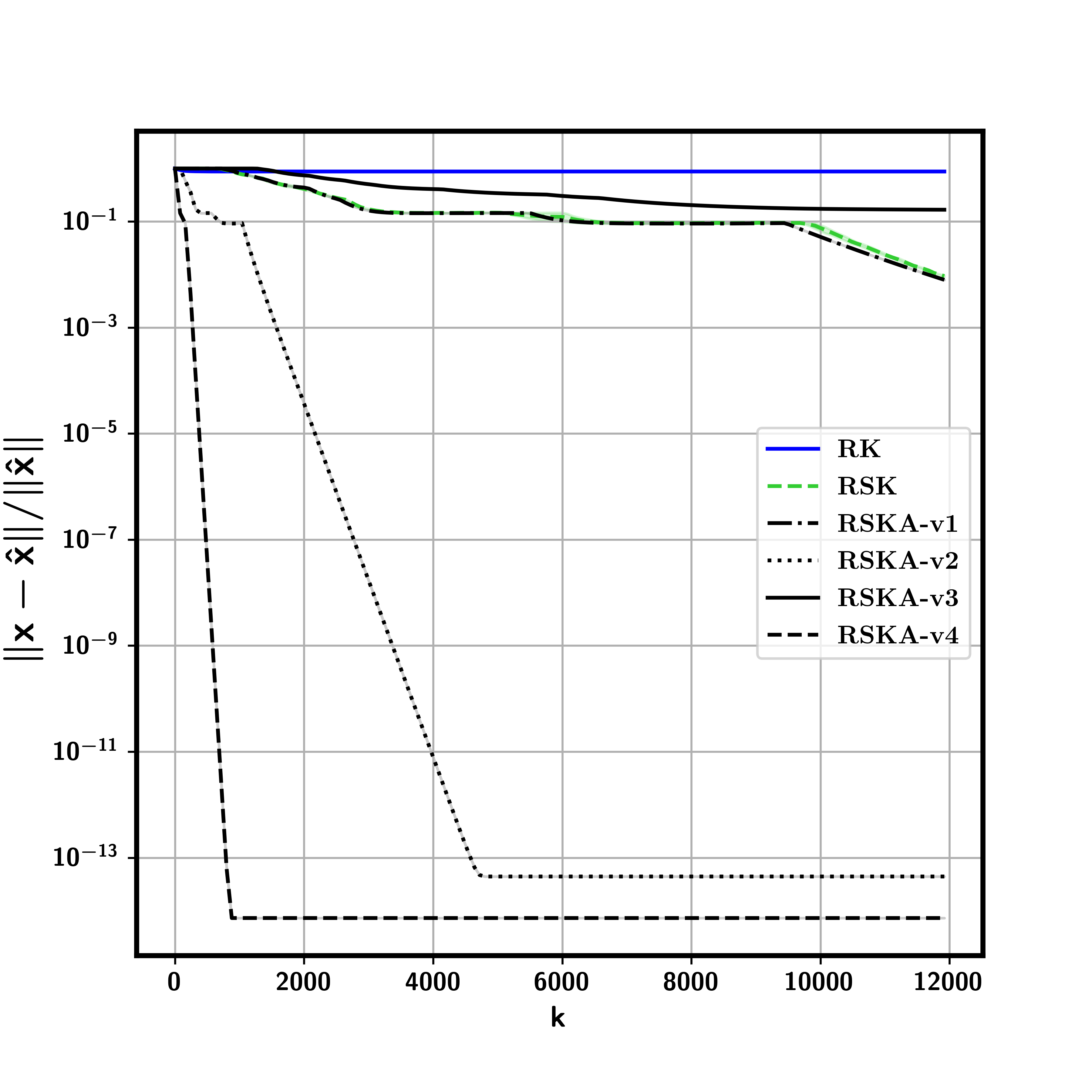} }}%
    \caption{\small A comparison of randomized Kaczmarz (blue), randomized sparse Kaczmarz (green) and RSKA method (black), $m = 100, n = 500,$ sparsity $s=10$, $\eta=11$, $\lambda=3$, no noise and 10 runs. Thick line shows mean over all trials, shaded area represent the standard deviation.}%
    \label{fig:example2}%
    \vspace{-0.4cm}
\end{figure}

\begin{figure}[htb]%
    \centering
    \subfloat[\centering Relative Residual]{{\includegraphics[width=.45\linewidth]{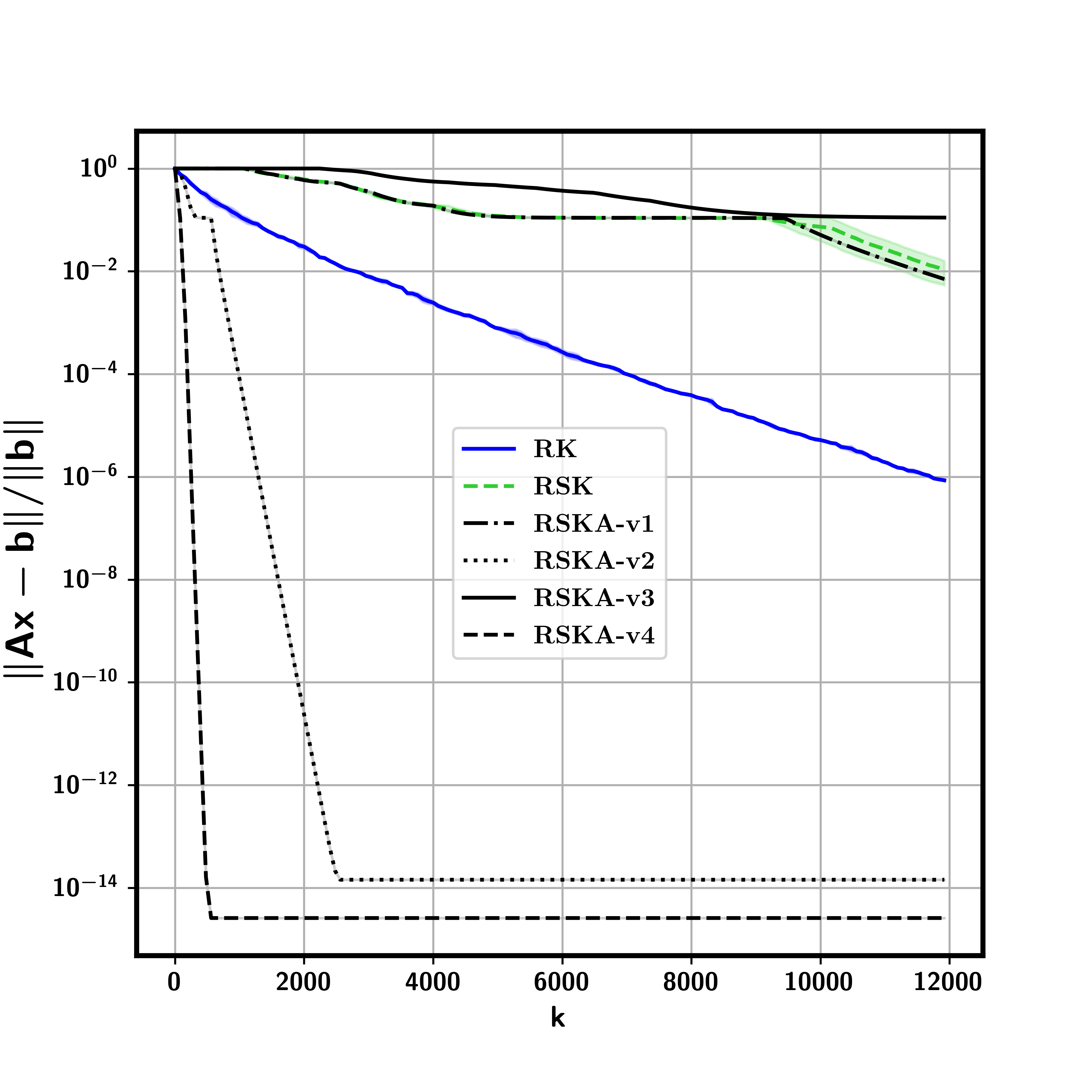} }}%
    \qquad
    \subfloat[\centering Error]{{\includegraphics[width=.45\linewidth]{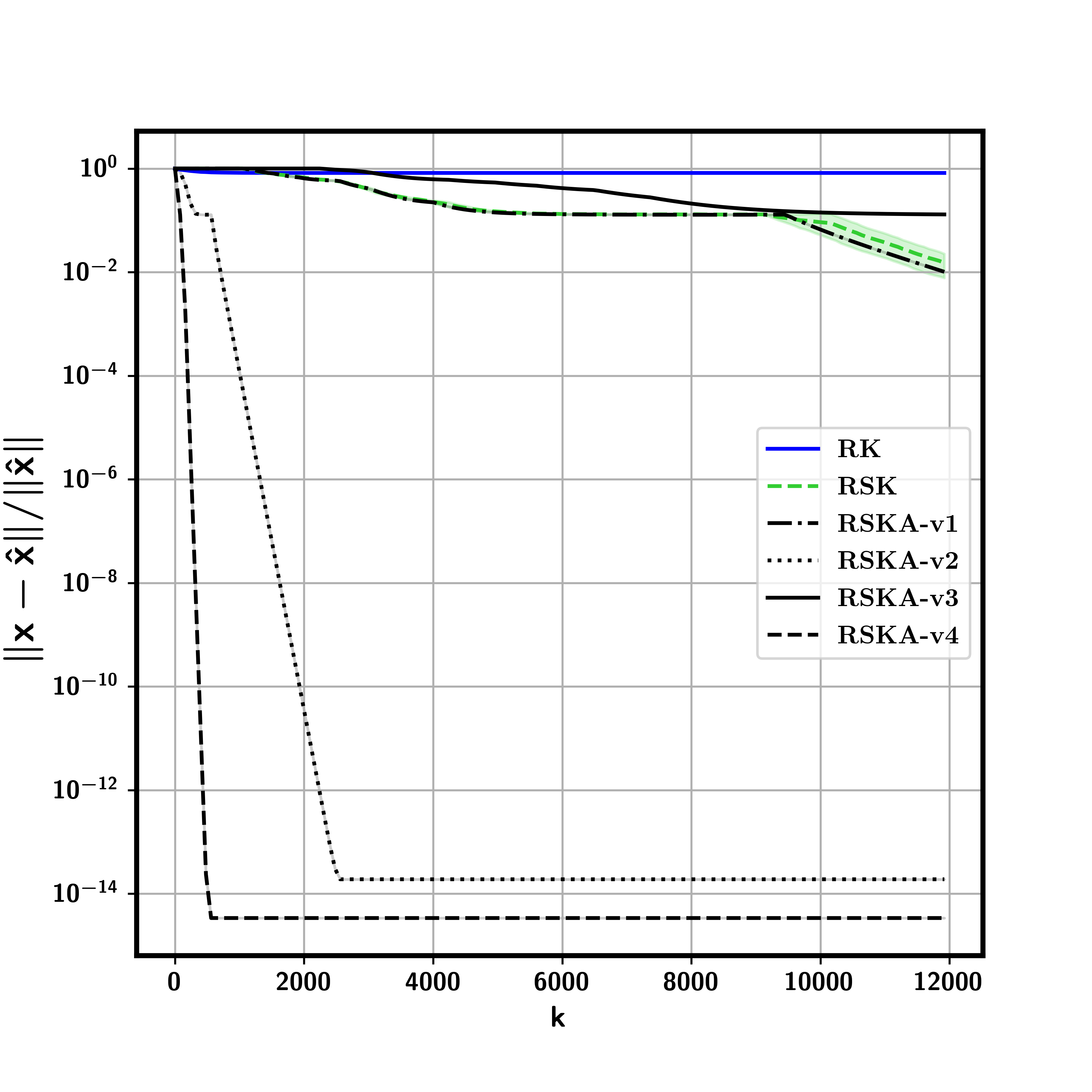} }}%
    \caption{\small A comparison of randomized Kaczmarz (blue), randomized sparse Kaczmarz (green) and RSKA method (black), $m = 200, n = 600,$ sparsity $s=10$, $\eta=21$, $\lambda=3$, no noise and 5 runs. Thick line shows mean over all trials, shaded area represent the standard deviation.}%
    \label{fig:example3}%
   \vspace{-0.8cm}
\end{figure}

\begin{figure}[htb]%
%\vspace{-0.8cm}
    \centering
    \subfloat[\centering Relative Residual]{{\includegraphics[width=.45\linewidth]{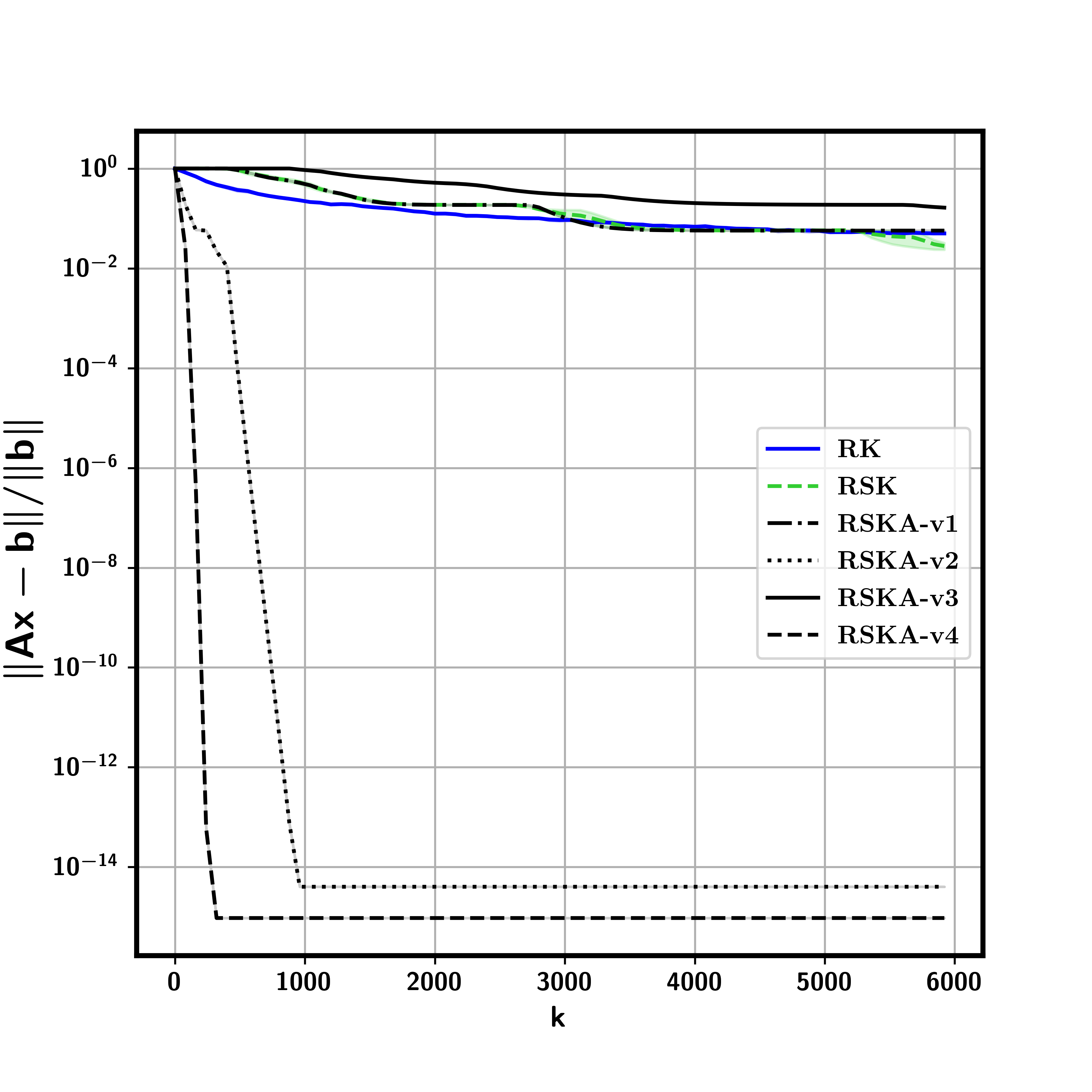} }}%
    \qquad
    \subfloat[\centering Error]{{\includegraphics[width=.45\linewidth]{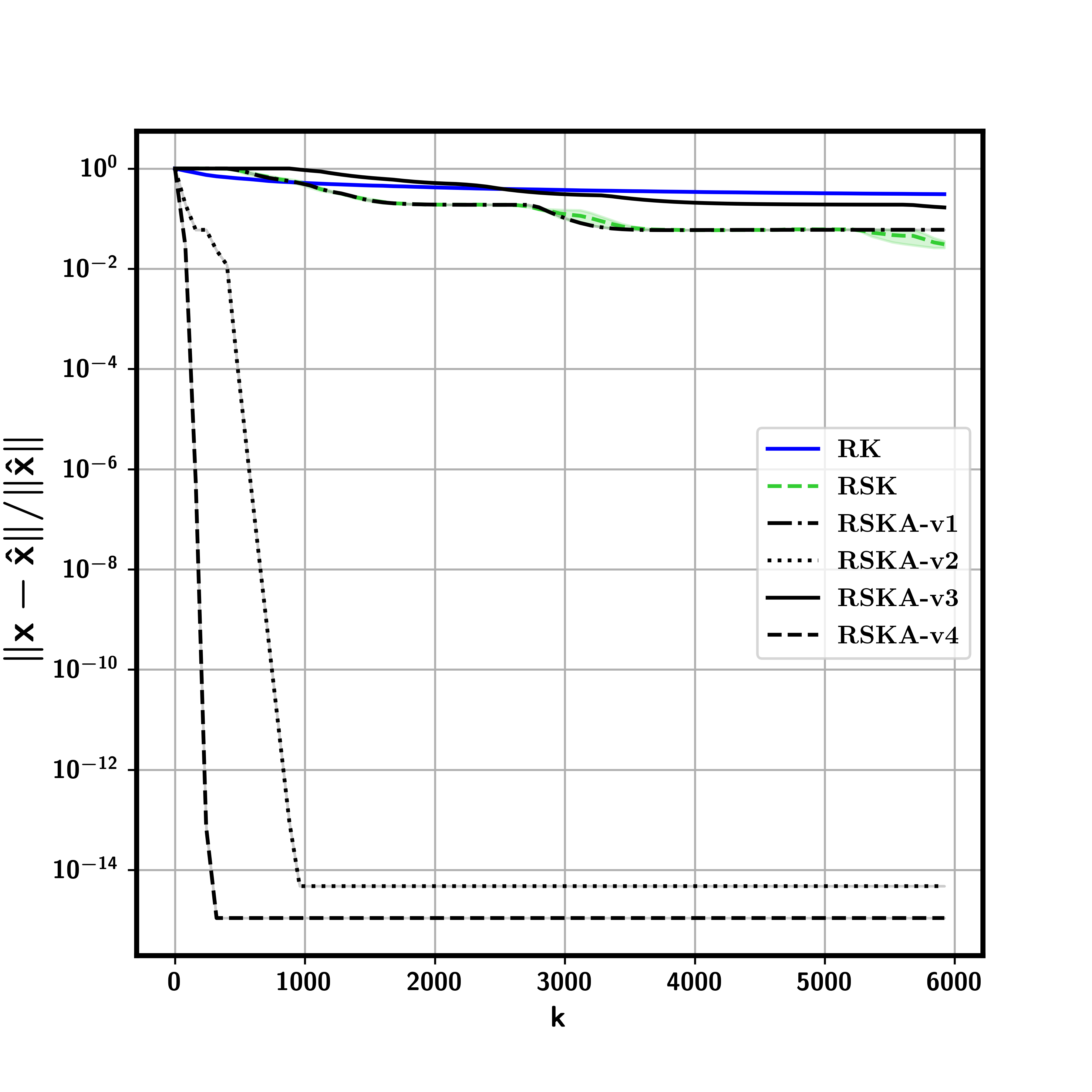} }}%
    \caption{\small A comparison of randomized Kaczmarz (blue), randomized sparse Kaczmarz (green) and RSKA method (black), $m = 300, n = 300,$ sparsity $s=10$, $\eta=31$, $\lambda=3$, no noise and 5 runs. Thick line shows mean over all trials, shaded area represent the standard deviation.}%
    \label{fig:example4}%
    %\vspace{-0.6cm}
\end{figure}

\subsection{The effect of the number of threads $\eta$}
In Figure \ref{fig:example5}, \ref{fig:example13}, \ref{fig:example6}, and \ref{fig:example14}, we see the effects of the number of threads $\eta$ in the error of Algorithm \ref{alg:RSKA} for the variant $v2$. We used under- and overdetermined and consistent system, with $\lambda \in \{0.01, 3\}$ . For the small $\lambda = 0.01$, the RSKA behaves almost like the standard randomized Kaczmarz with averaging from~\cite{moorman2021randomized}, while for the larger value $\lambda = 3$, we see the typical behavior for the sparse Kaczmarz method which stagnates from time to time and switches to  faster improvement inbetween~\cite{LSW14}. As the number of threads $\eta$ increases, we see a corresponding decrease in the number of iterations needed to reach a certain accuracy, however at some point increasing $\eta$ does not improve the method in accordance with Remark \ref{rmk:rmk_lc}. For smaller values of $\eta$ we roughly see an $\eta$-times speedup in the number of iterations. Thus it is clear that the averaging will pay off as soon as the updates in RSKA can be done in parallel. 
\begin{figure}[htb]%
    \centering
    \subfloat[\centering Relative Residual]{{\includegraphics[width=.45\linewidth]{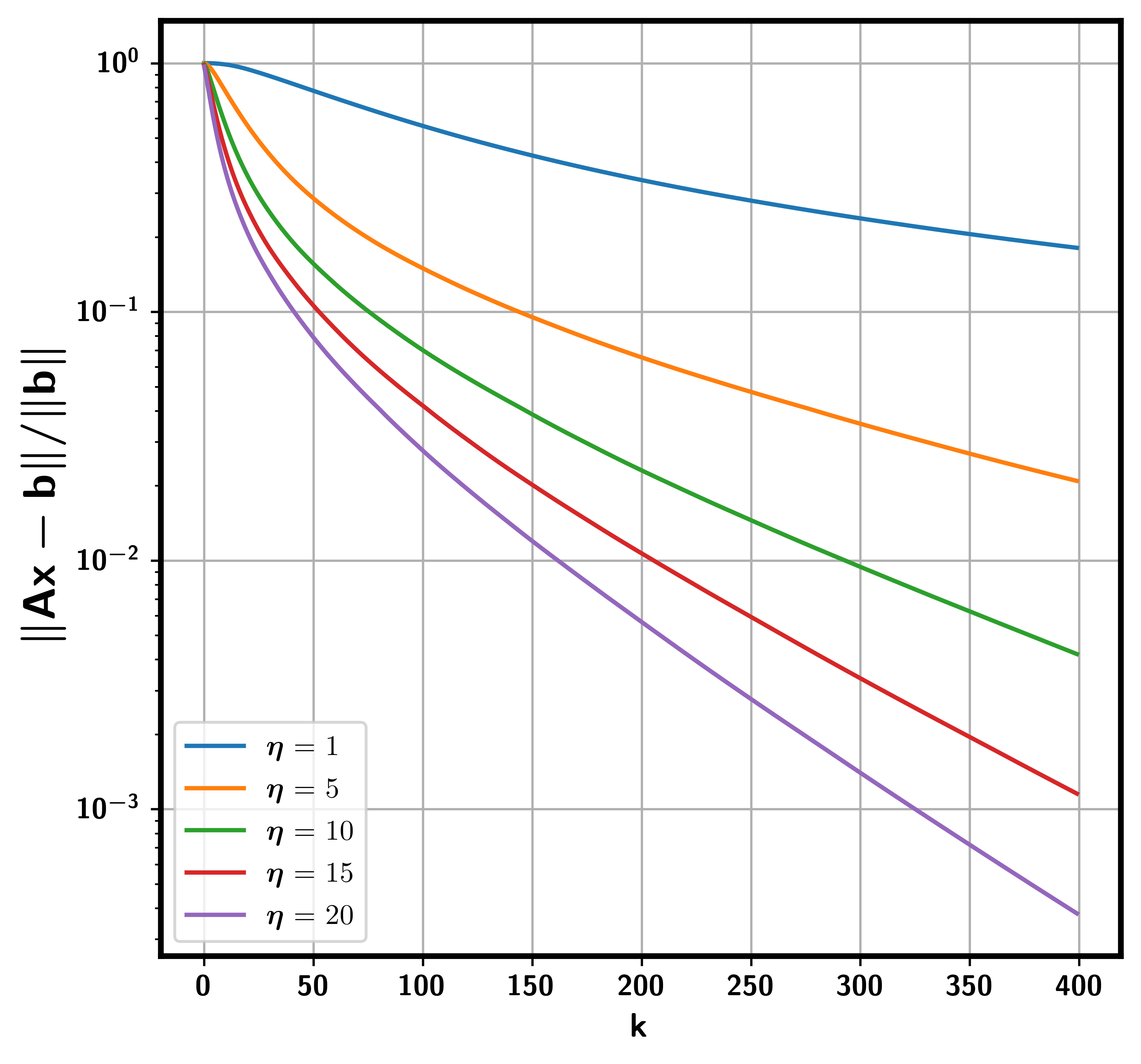} }}%
    \qquad
    \subfloat[\centering Error]{{\includegraphics[width=.47\linewidth]{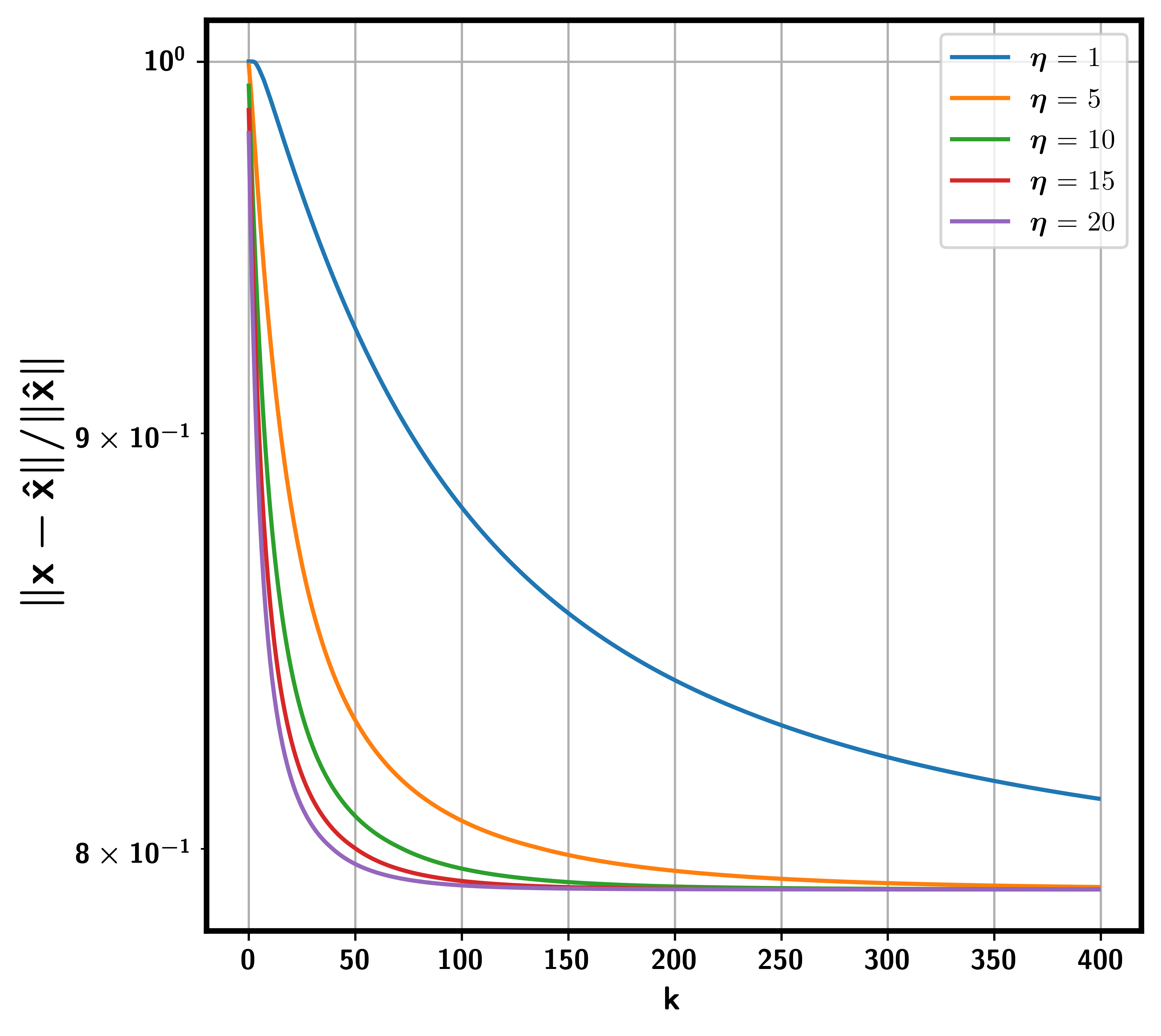} }}%
    \caption{\small The effect of the number of threads $\eta$ on the error and relative residual versus iteration for Algorithm \ref{alg:RSKA} on the variant $v2$. $m = 200, n = 600,$ sparsity $s=10$, $\lambda=0.01$, no noise.}%
    \label{fig:example5}%
    \vspace{-0.6cm}
\end{figure}

\begin{figure}[htb]%
    \centering
    \subfloat[\centering Relative Residual]{{\includegraphics[width=.45\linewidth]{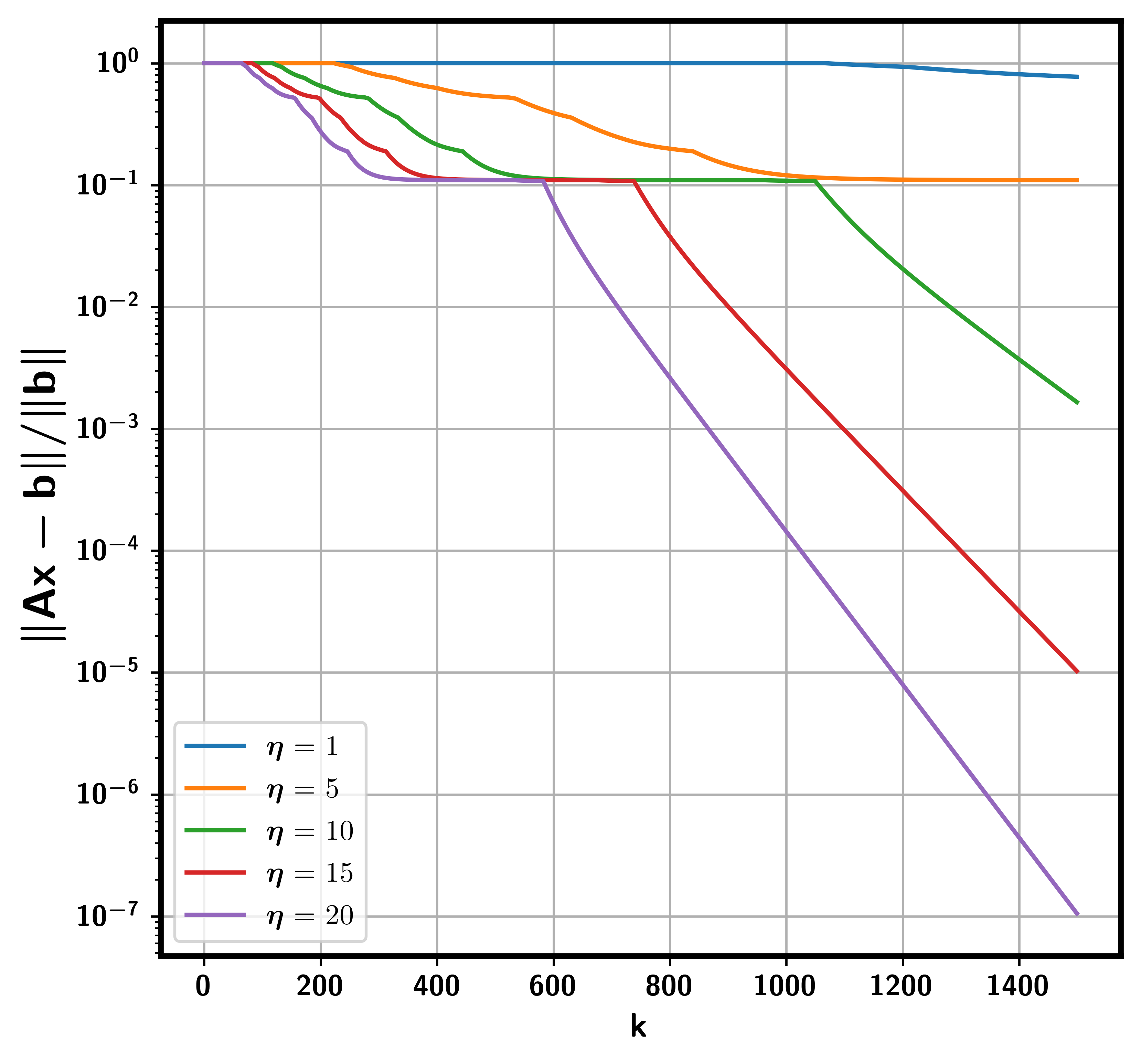} }}%
    \qquad
    \subfloat[\centering Error]{{\includegraphics[width=.45\linewidth]{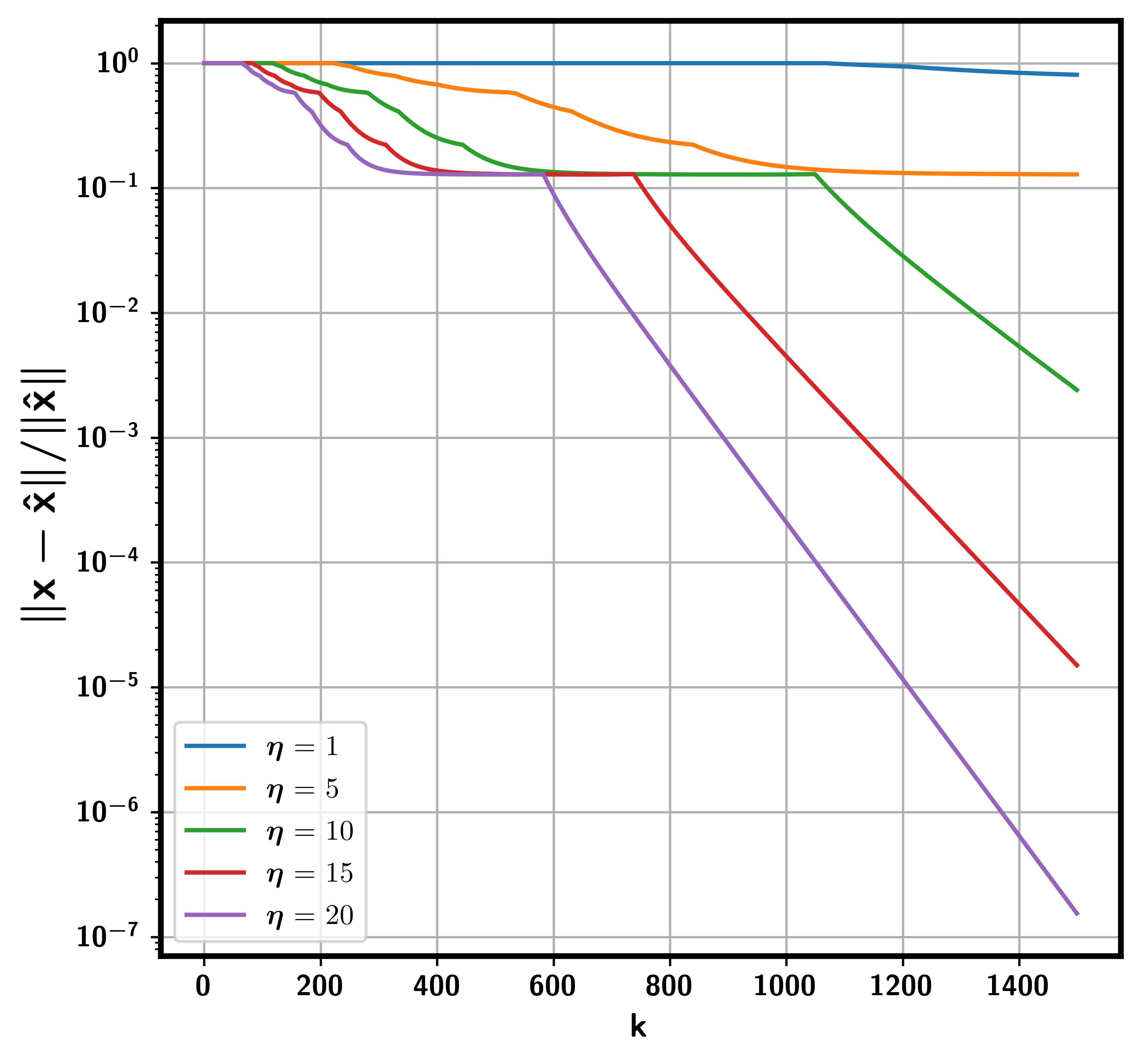} }}%
    \caption{\small The effect of the number of threads $\eta$ on the error and relative residual versus iteration for Algorithm \ref{alg:RSKA} on the variant $v2$. $m = 200, n = 600,$ sparsity $s=10$, $\lambda=3$, no noise.}%
    \label{fig:example13}%
    \vspace{-0.6cm}
\end{figure}

\begin{figure}[htb]%
    \centering
    \subfloat[\centering Relative Residual]{{\includegraphics[width=.45\linewidth]{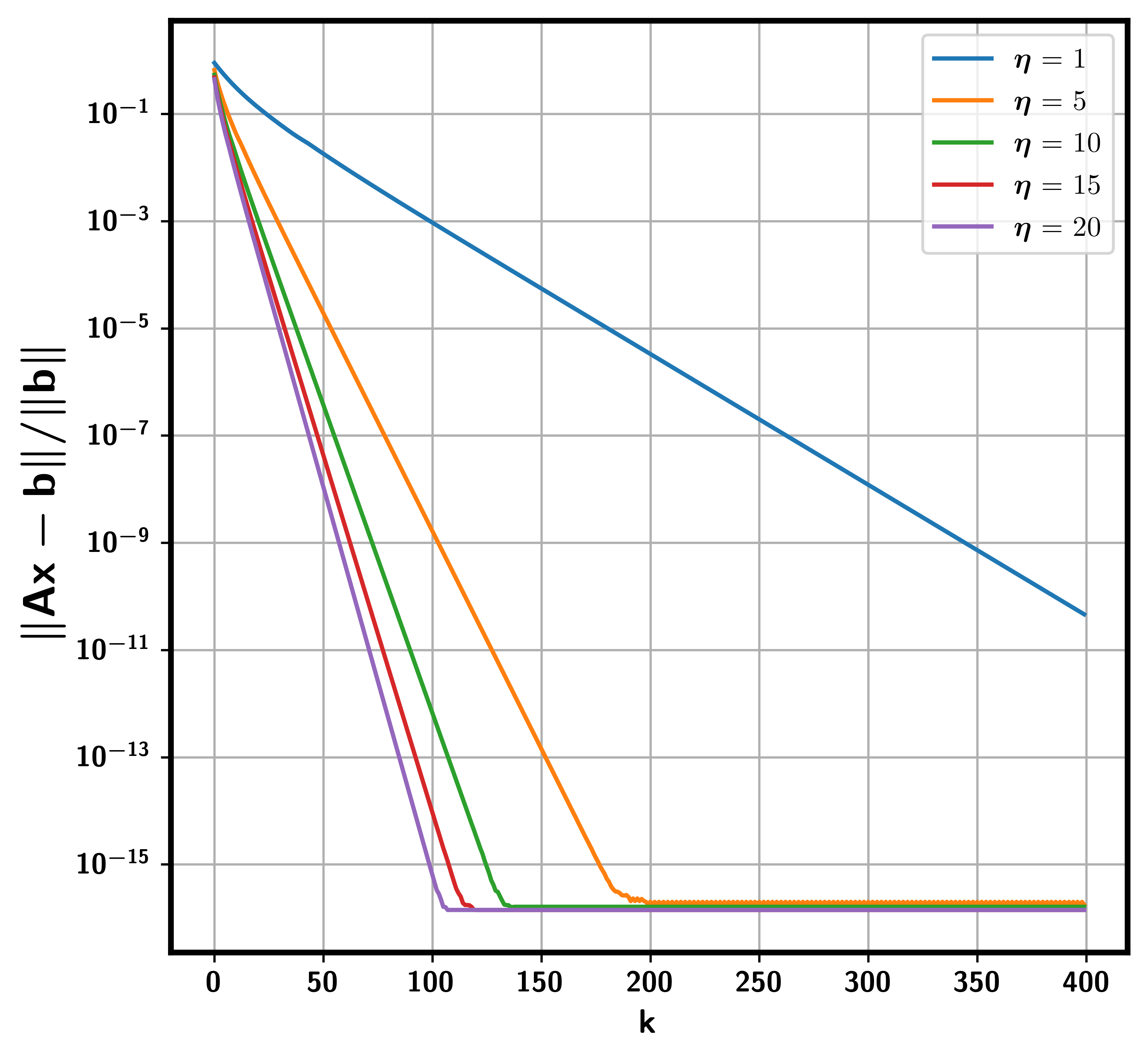} }}%width=2.5in
    \qquad
    \subfloat[\centering Error]{{\includegraphics[width=.45\linewidth]{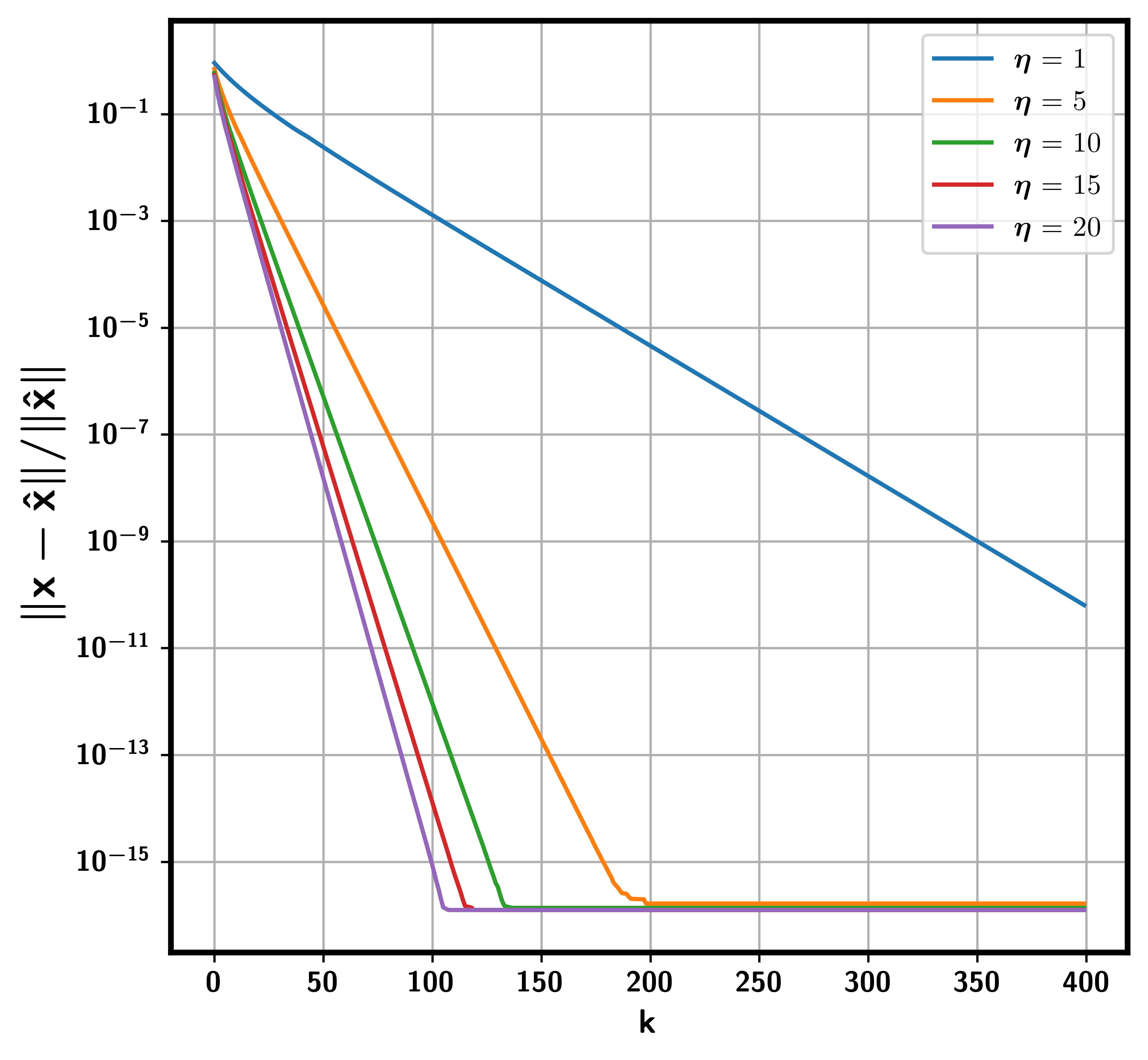} }}%
    \caption{\small The effect of the number of threads $\eta$ on the error and relative residual versus iteration for Algorithm \ref{alg:RSKA} on the variant $v2$. $m = 100, n = 10,$ sparsity $s=10$, $\lambda=0.01$, no noise.}%
    \label{fig:example6}%
    \vspace{-0.5cm}
\end{figure}

\begin{figure}[htb]%
    \centering
    \subfloat[\centering Relative Residual]{{\includegraphics[width=.45\linewidth]{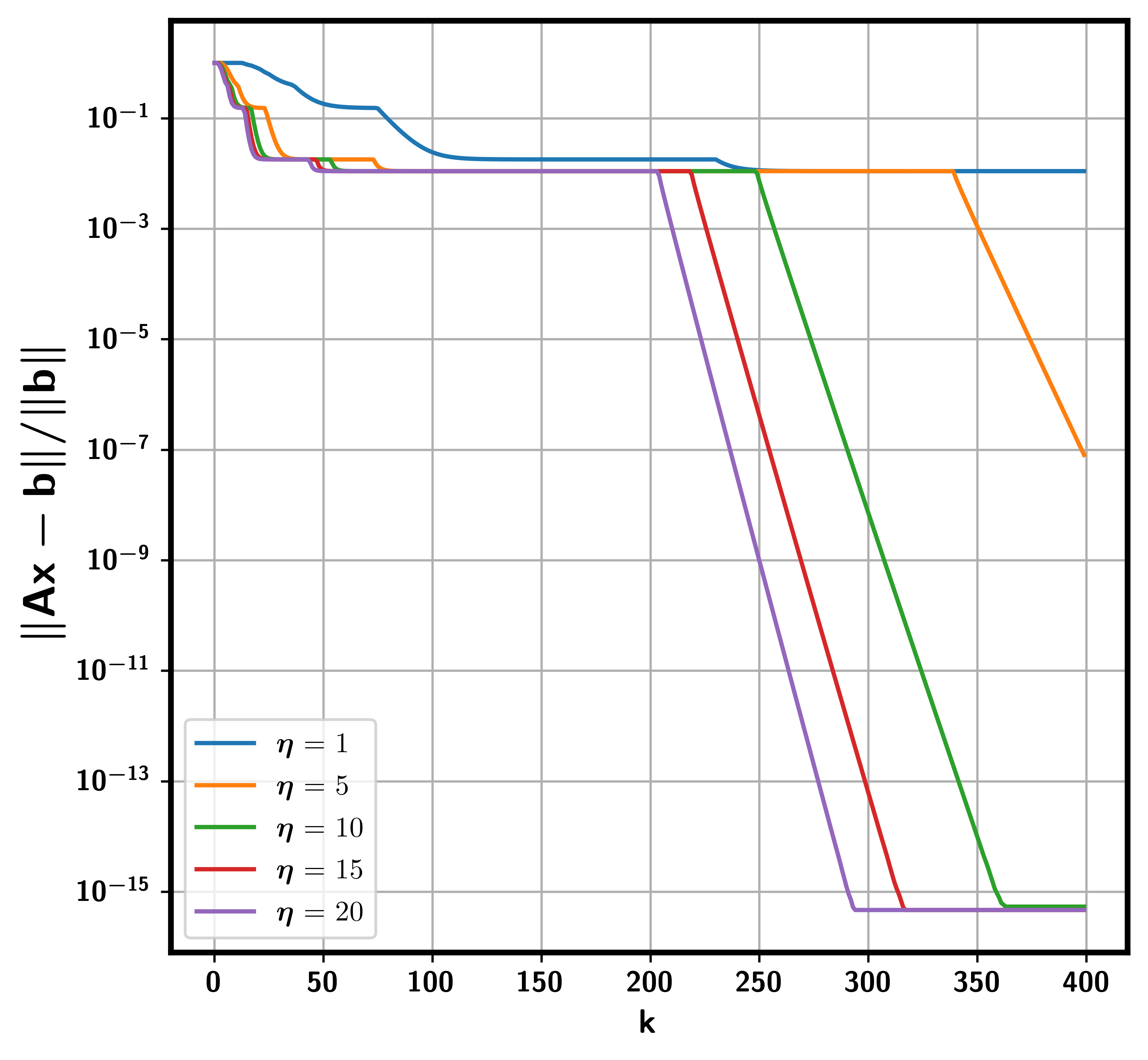} }}%width=2.5in
    \qquad
    \subfloat[\centering Error]{{\includegraphics[width=.45\linewidth]{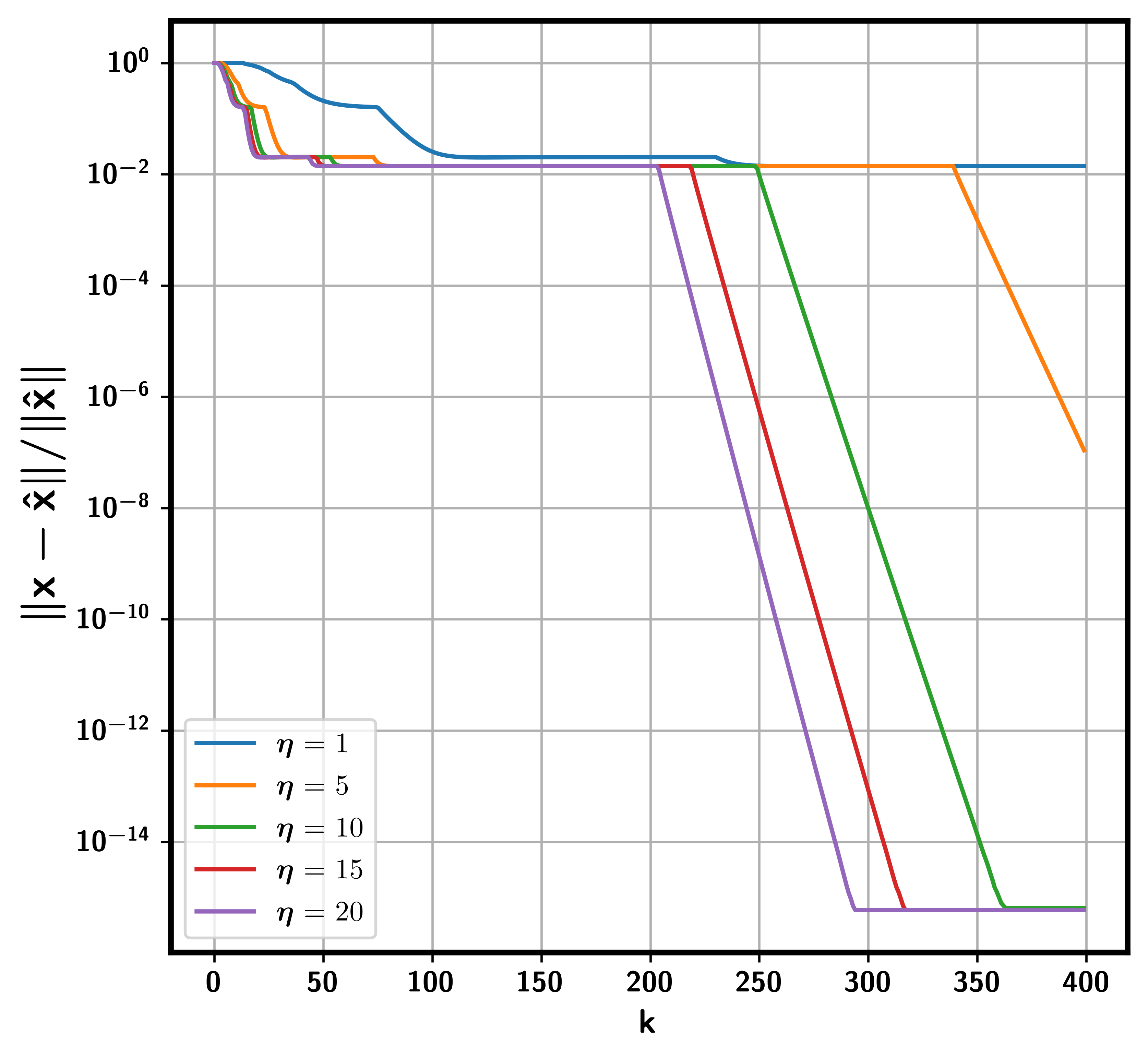} }}%
    \caption{\small The effect of the number of threads $\eta$ on the error and relative residual versus iteration for Algorithm \ref{alg:RSKA} on the variant $v2$. $m = 100, n = 10,$ sparsity $s=10$, $\lambda=3$, no noise.}%
    \label{fig:example14}%
    \vspace{-0.5cm}
\end{figure}

\subsection{The effect of the relaxation parameter $\alpha$}
\label{sec:effect-relaxation}
In Figure \ref{fig:example7}, we observe the effect on the convergence rate as we vary the relaxation parameter $\alpha$. We used an underdetermined and consistent system with $\lambda = 0.1$, $\eta = 8$ with variant $v2$. In fact increasing $\alpha$ allow us to get smaller error, however the method can ultimately diverge for some larger values of the relaxation parameter $\alpha$. This is observed in Figure \ref{fig:example8} which plot the relative residual and the error after 100 iterations on a smaller example for various relaxations parameters $\alpha$ and batch sizes $\eta$. The theoretically optimal parameter $\alpha^{*}$ from Corollary~\ref{cor:linear_convergence} is indicated as a dot. We used an overdetermined and consistent system with $\lambda = 1$ with variant $v2$.  The plots confirms that $\alpha$ can not be chosen too large, i.e there exists an $\eta$-depended upper bound for $\alpha$ which leads to convergence (cf. Theorem~\ref{th:RSKA}). However we also observe that a larger relaxation parameter than $\alpha^{*}$ from Corollary~\ref{cor:linear_convergence} leads to even faster convergence.
\begin{figure}[htb]%
    \centering
    \subfloat[\centering Relative Residual]{{\includegraphics[width=.45\linewidth]{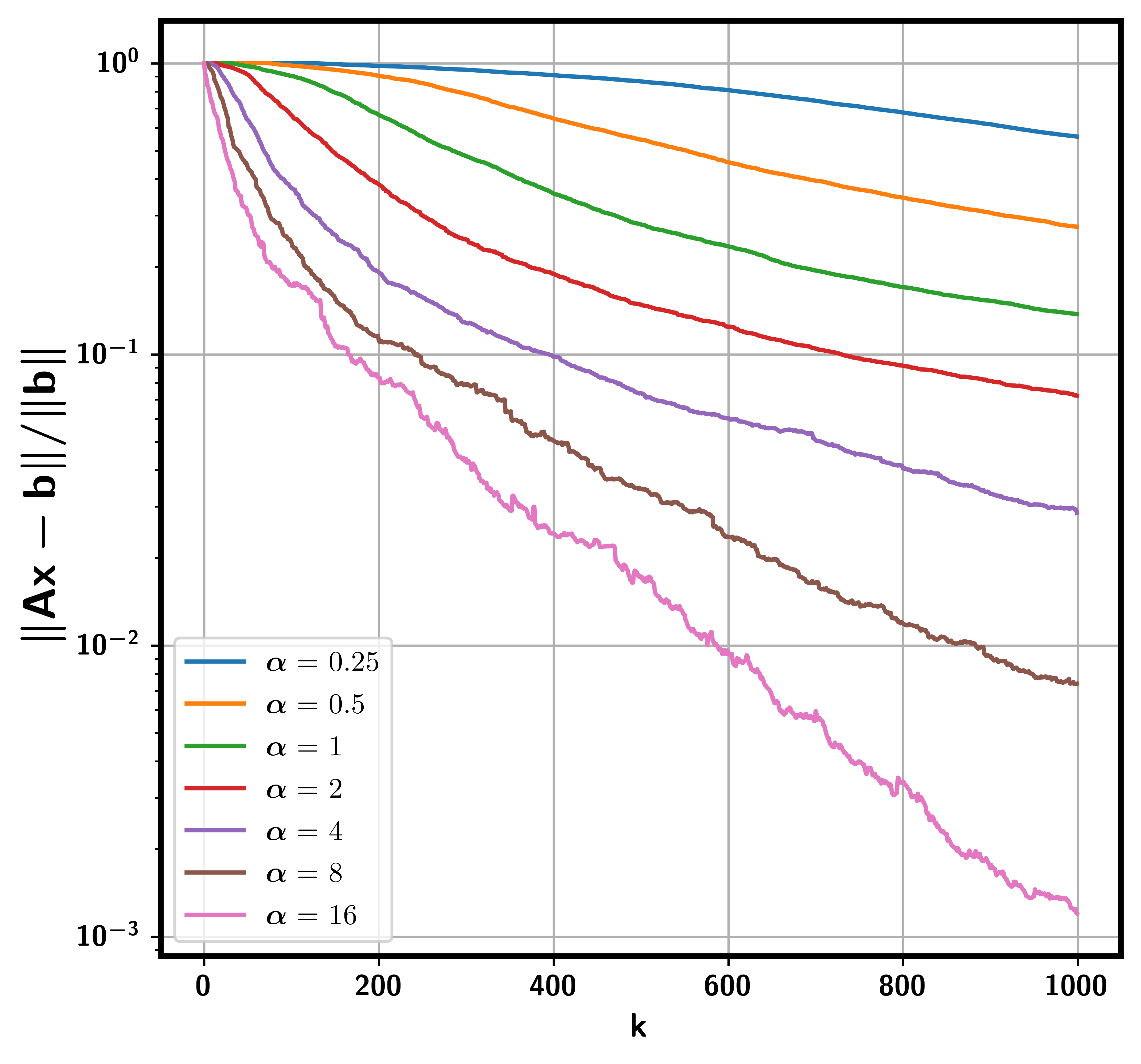} }}%
    %\qquad
    \subfloat[\centering Error]{{\includegraphics[width=.47\linewidth]{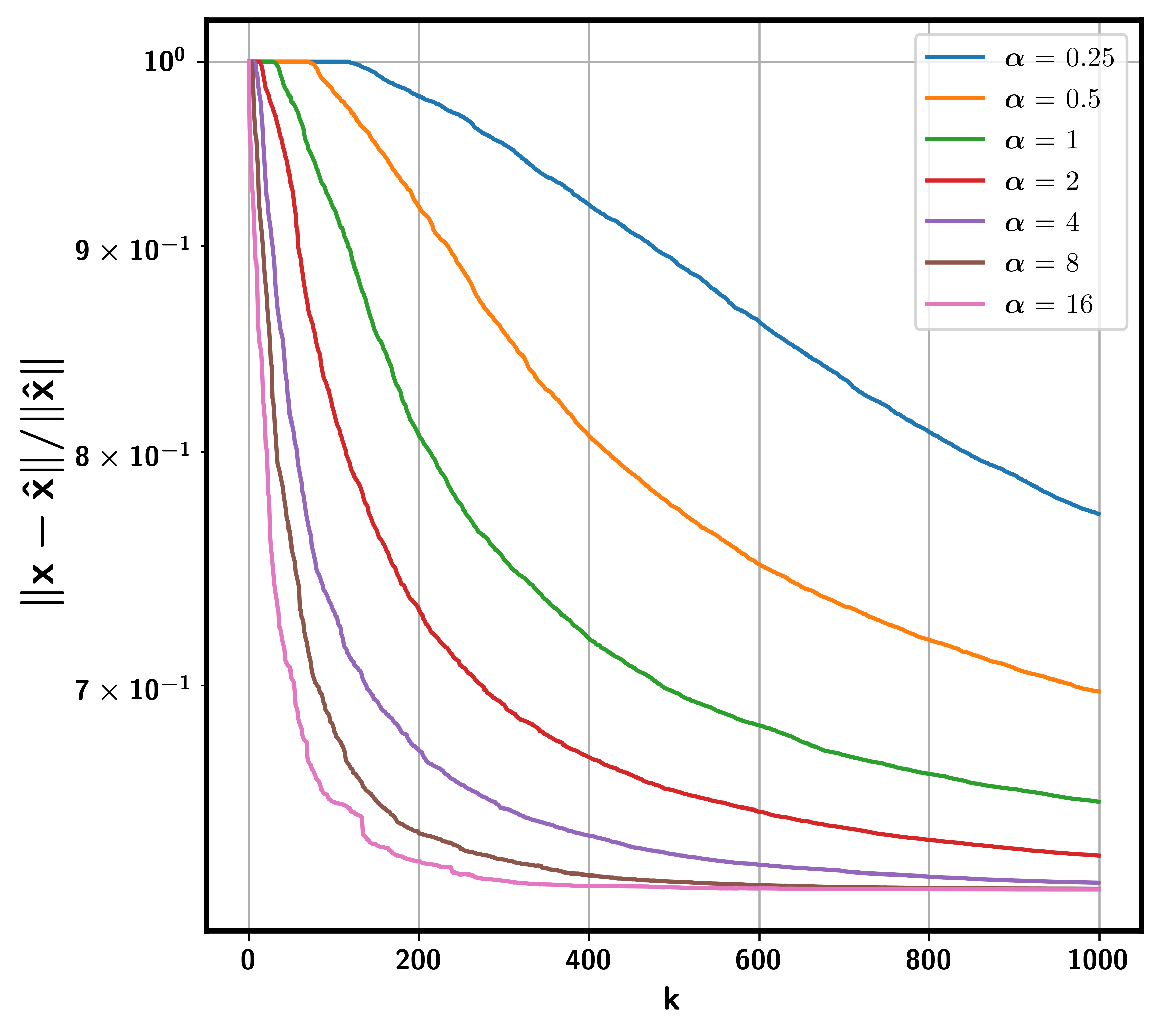} }}%
    \caption{\small The effect of the relaxation parameter $\alpha$ on the error and relative residual versus iteration for Algorithm \ref{alg:RSKA} on the variant $v2$. $m = 200, n = 600,$ sparsity $s=10$, $\eta=8$, $\lambda=0.1$, no noise.}%
    \label{fig:example7}%
    %\vspace{-0.4cm}
\end{figure}

% \begin{figure}[htb]%
%     \centering
%     \subfloat[\centering Relative Residual]{{\includegraphics[width=.45\linewidth]{figures/effect_of_alpha/residualsfor_eta_alpha1.png} }}%
%     %\qquad
%     \subfloat[\centering Error]{{\includegraphics[width=.45\linewidth]{figures/effect_of_alpha/errorfor_eta_alpha1.png} }}%
%     \caption{\small The effect of the relaxation parameter $\alpha$ for various values of $\eta$ on the error and relative residual after $1000$ iterations of Algorithm \ref{alg:RSKA} with the variant $v1$. $m = 100, n = 10,$ sparsity $s=10$, $\lambda=1$, no noise. Circle markers are estimates of the optimal relaxation parameter using Corollary \ref{cor:linear_convergence}.}%
%     \label{fig:example8}%
%     %\vspace{-0.8cm}
% \end{figure}

\begin{figure}[htb]%
    \centering
    \includegraphics[width=.5\linewidth]{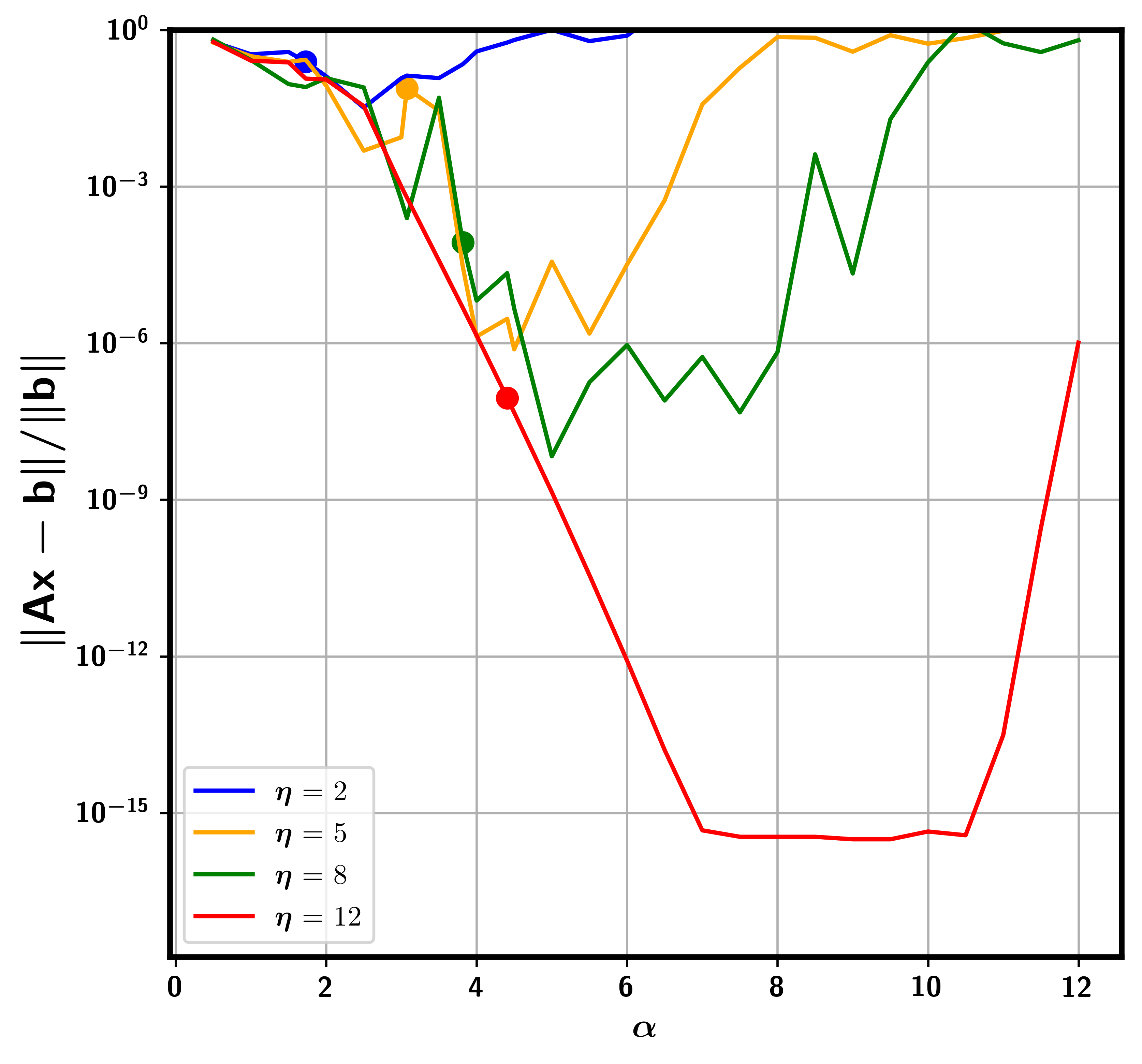}%
    %\qquad
    %\subfloat[\centering Error]{{\includegraphics[width=.45\linewidth]{figures/effect_of_alpha/errorfor_eta_alpha1.png} }}%
    \caption{\small The effect of the relaxation parameter $\alpha$ for various values of $\eta$ on the relative residual after $100$ iterations of Algorithm \ref{alg:RSKA} with the variant $v2$. $m = 100, n = 10,$ sparsity $s=10$, $\lambda=6$, no noise. Circle markers are estimates of the optimal relaxation parameter using Corollary \ref{cor:linear_convergence}.}%
    \label{fig:example8}%
    %\vspace{-0.8cm}
\end{figure}

\subsection{The effect of the sparsity parameter $\lambda$}
In Figure \ref{fig:example9}, we see the effects of the sparsity parameter $\lambda$ on the approximation error of Algorithm \ref{alg:RSKA} and the randomized Kaczmarz method (RK). We used an underdetermined and consistent system with $\eta = 21$ with variant $v2$ of RSKA. We observed that as you increase the sparsity parameter $\lambda$, the relative residual and the reconstruction error get worse and RSK is more affected by this behavior whereas RSKA keep it performance along different $\lambda$. The first row of Figure \ref{fig:example9} correspond to experiment where we run 1000 iterations and in the second row the methods are run for $(1 + \lambda)1000$ iterations for all $\lambda$.

\begin{figure}[htb]%
    \centering
    \subfloat[\centering Relative Residual]{{\includegraphics[width=.45\linewidth]{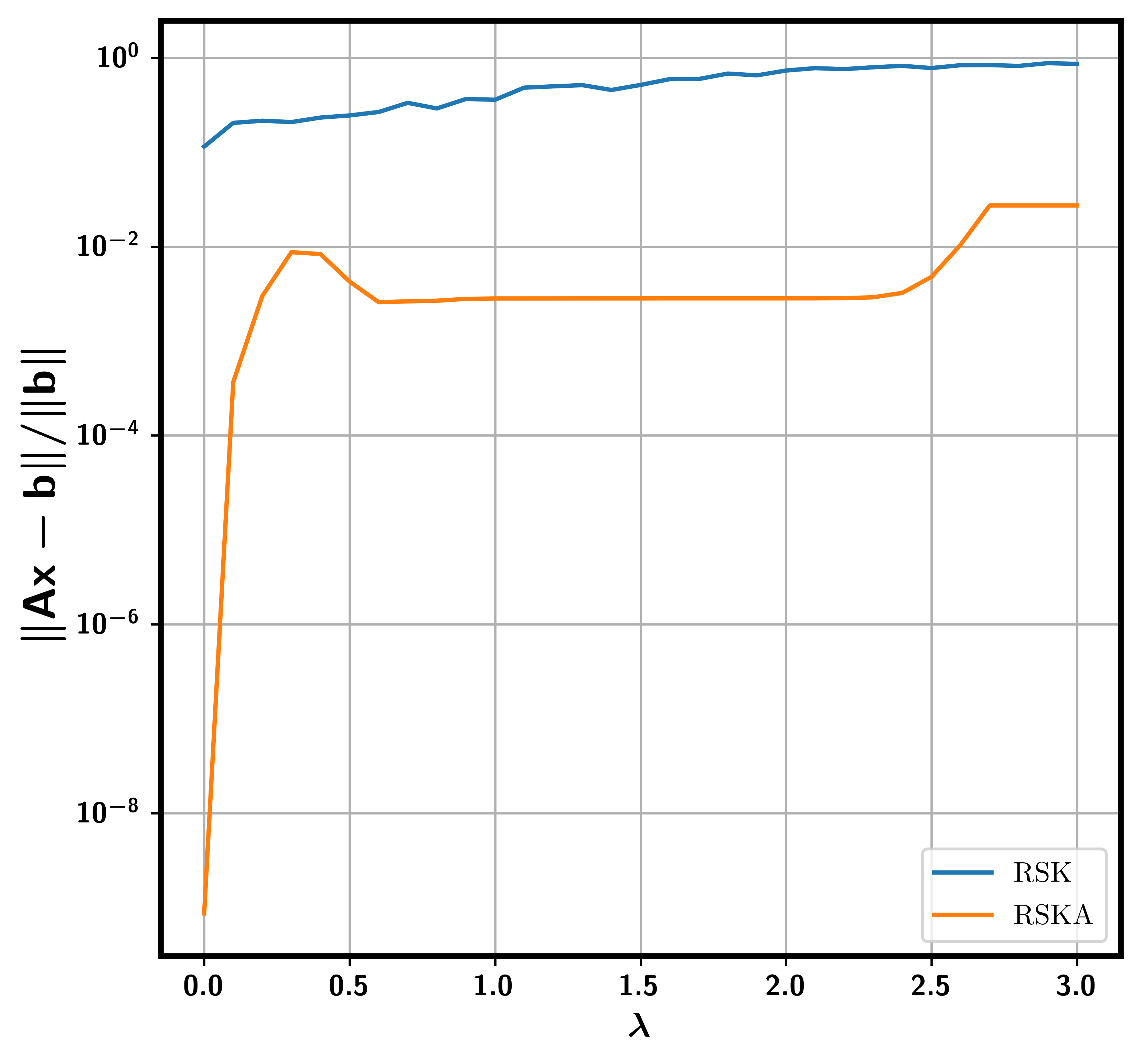} }}%
    \qquad
    \subfloat[\centering Error]{{\includegraphics[width=.45\linewidth]{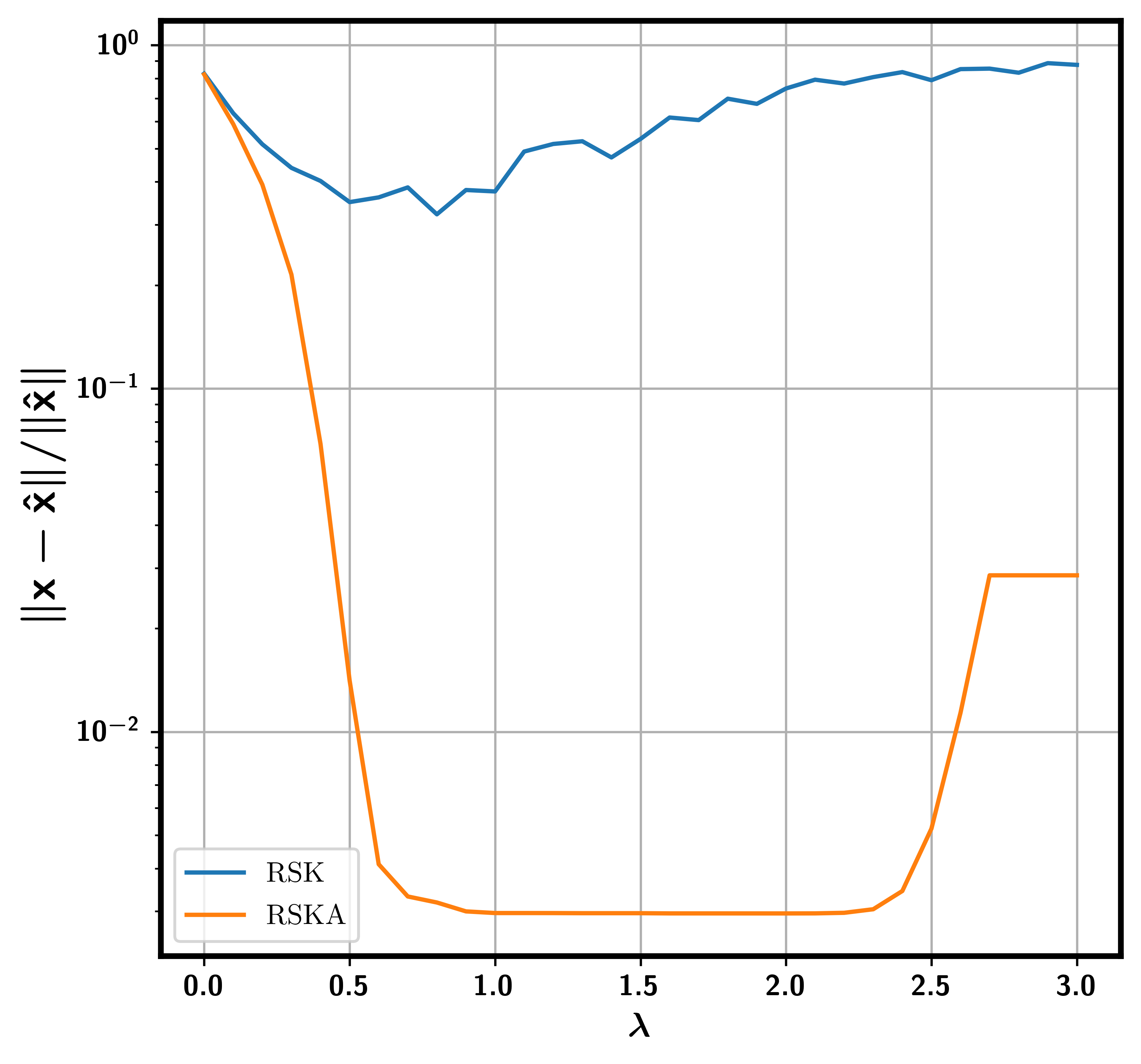} }}\\
    \subfloat[\centering Relative Residual]      {
    \includegraphics[width=.45\linewidth]{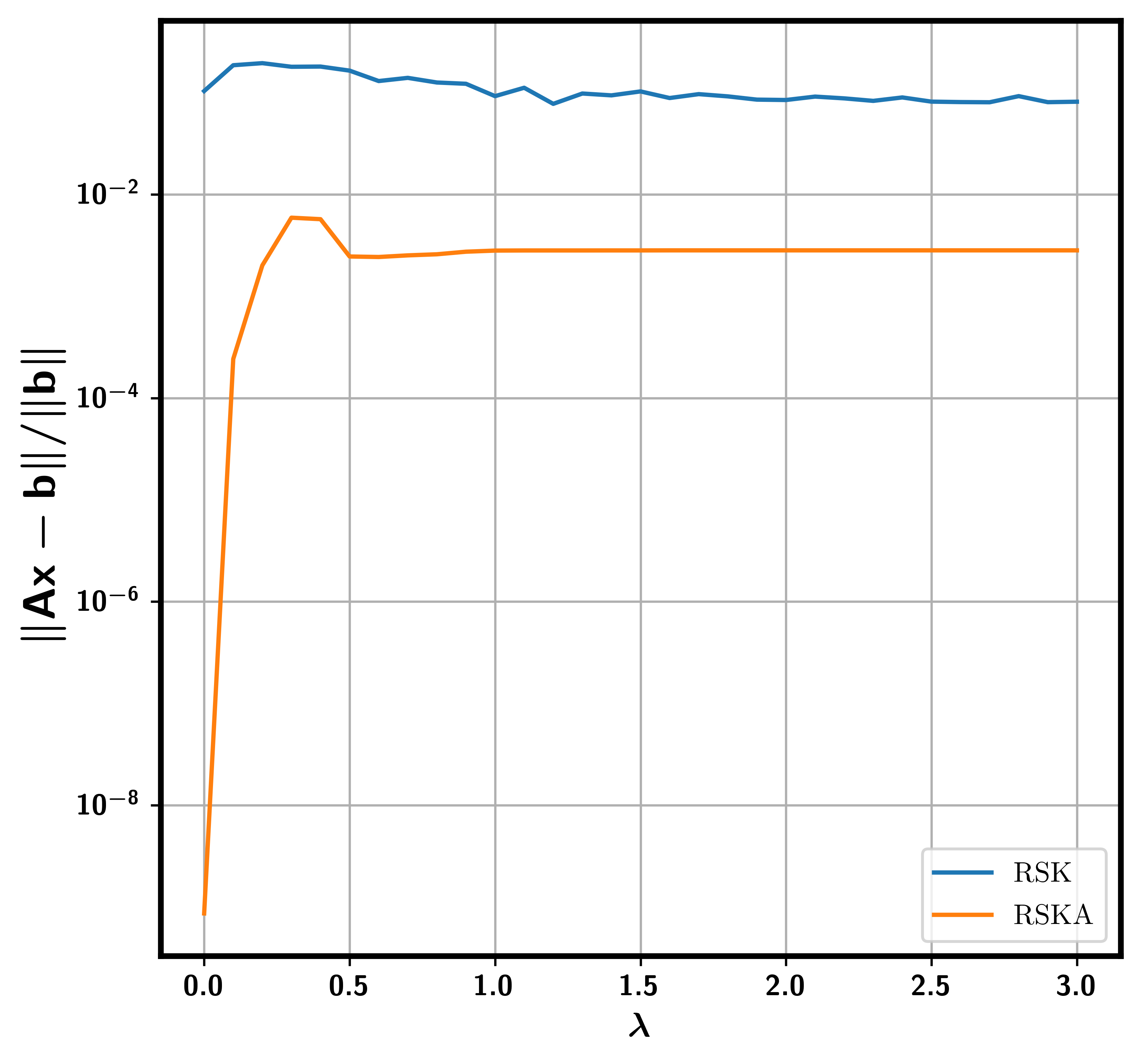}
    }\qquad
    \subfloat[\centering Error]      {
    \includegraphics[width=.45\linewidth]{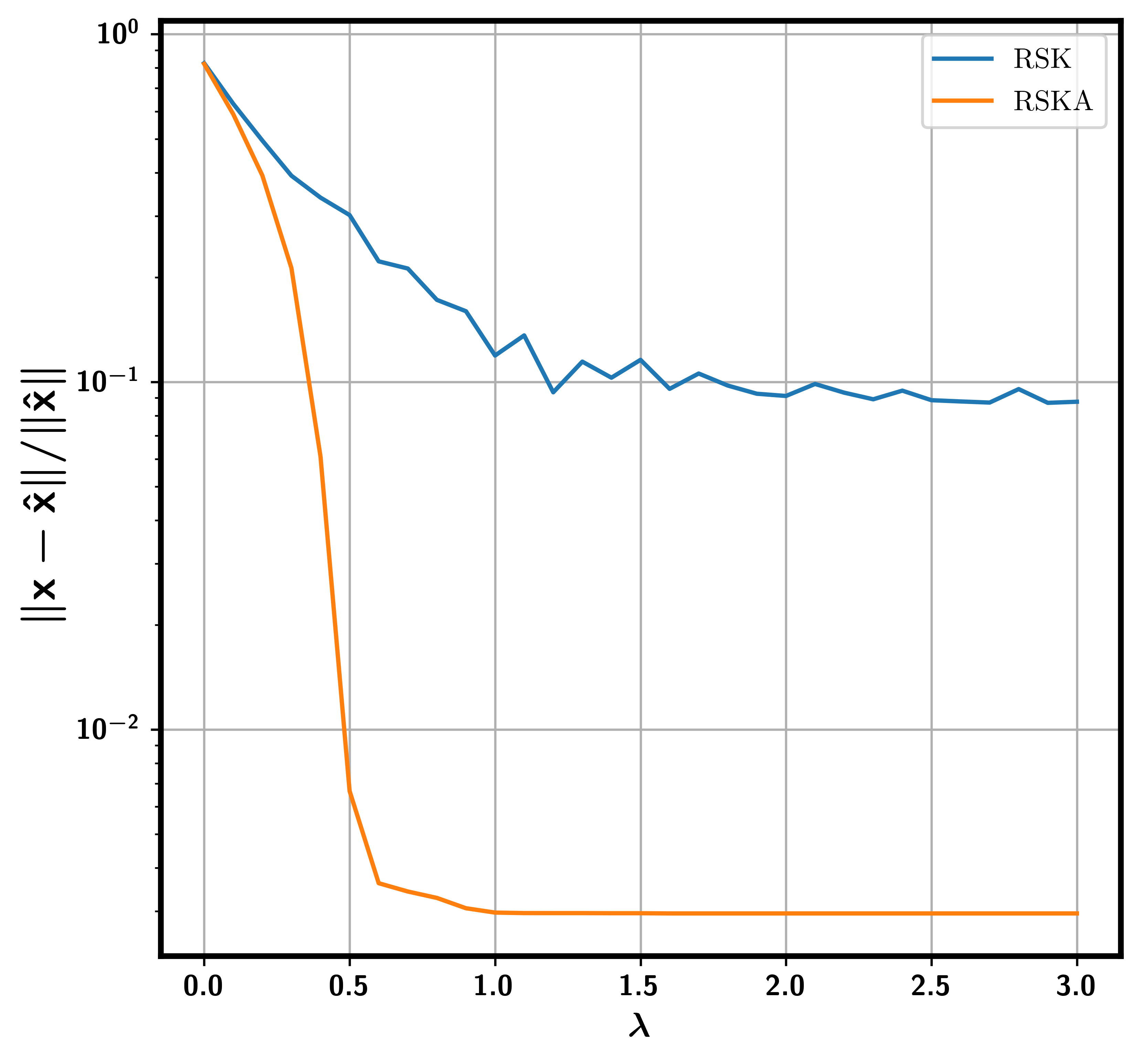}
    }
    \caption{\small A comparison of randomized sparse Kaczmarz (blue) and RSKA-$v2$ method (orange) in terms of the sparsity parameter $\lambda$. $m = 200, n = 600,$ sparsity $s=10$, $\eta=21$, no noise. In the first row $(a)$ and $(b)$, methods are run for 1000 iterations for all $\lambda$ whereas in the second row $(c)$ and $(d)$, methods are run for $(1 + \lambda)1000$ iterations for all $\lambda$.}%
    \label{fig:example9}%
    %\vspace{-0.5cm}
\end{figure}

\subsection{Noisy case}
\label{sec:noisy}
In this part we are interested in the effectiveness of the RSKA method on inconsistent systems. We construct a sparse $\hat x$ with normally distributed non-zero entries and set $b = \mathbf{A}\hat{x}, \quad b^{\epsilon} = b + \epsilon $ where $\epsilon$ is a random vector uniformly distributed on a sphere with radius $\ell$ that correspond to the relative noise level such that $\|b - b^{\epsilon}\|_2 \leq \ell$. 
Figures \ref{fig:example10}, \ref{fig:example11} and \ref{fig:example12} show the results for noisy right hand sides, all with $\lambda=1$ for RSK and RSKA. Figure \ref{fig:example11} uses a five times overdetermined system with $10 \%$ relative noise, Fig. \ref{fig:example12} has the same noise level and a five times underdetermined system. In the underdetermined case, all methods consistently stagnate at a residual level which is comparable to the noise level, however, in all settings, RSKA variants achieves faster convergence than RSK which in turn is faster than RK. Regarding the reconstruction error, RSKA and RSK achieve reconstructions with an error in the size of the noise level, while RSKA achieves an even lower reconstruction error.
\begin{figure}[htb]%
\vspace{-0.4cm}
    \centering
    \subfloat[\centering Relative Residual]{{\includegraphics[width=.45\linewidth]{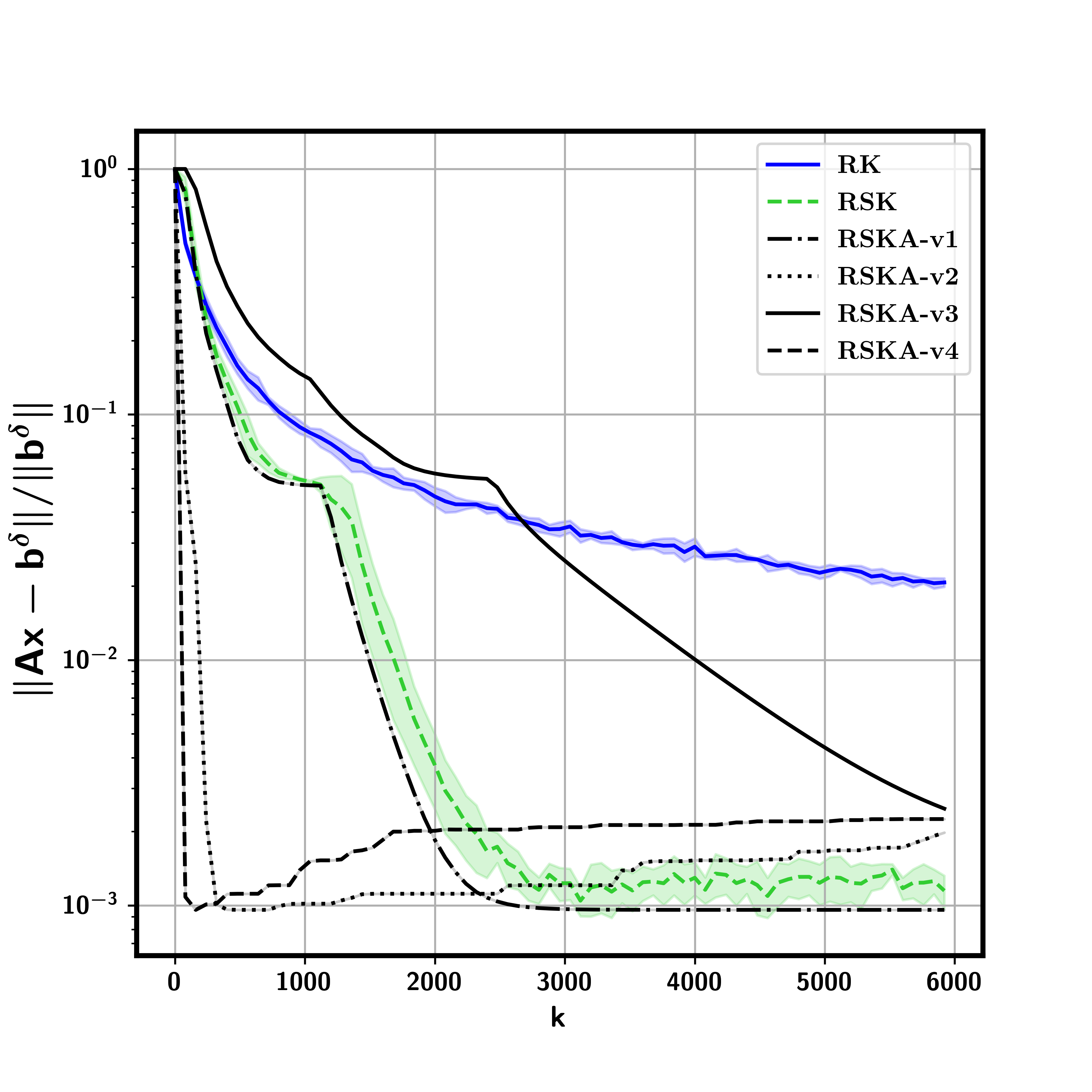} }}%
    \qquad
    \subfloat[\centering Error]{{\includegraphics[width=.45\linewidth]{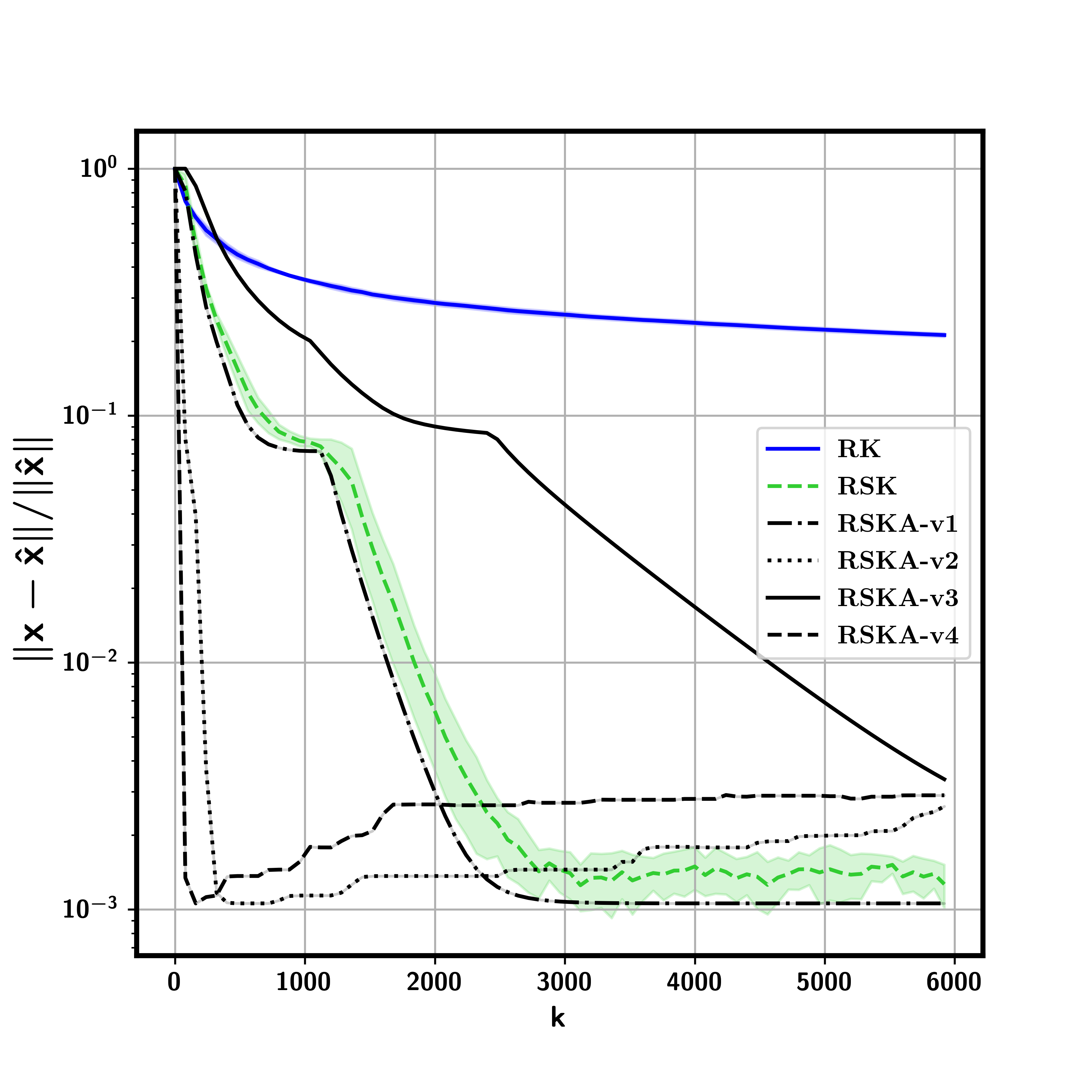} }}%
    \caption{\small A comparison of randomized Kaczmarz (blue), randomized sparse Kaczmarz (green) and RSKA method (black), $m = 100, n = 100,$ sparsity $s=10$, $\eta=11$, $\lambda=1$, noise level $l=0.1$ and 5 runs. Thick line shows mean over all trials, shaded area represent the standard deviation.}%
    \label{fig:example10}%
    \vspace{-0.6cm}
\end{figure}

\begin{figure}[htb]%
    \centering
    \subfloat[\centering Relative Residual]{{\includegraphics[width=.45\linewidth]{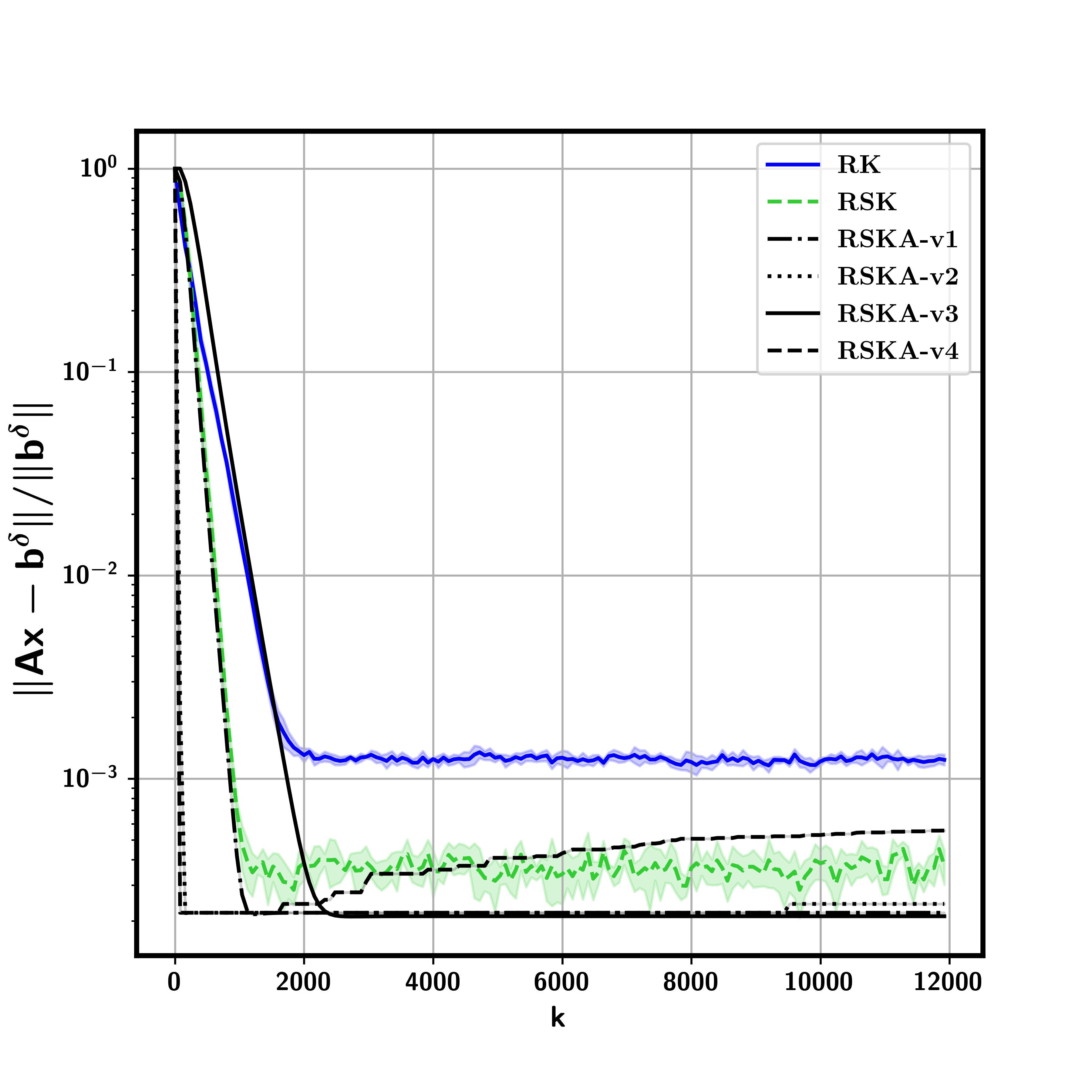} }}%
    \qquad
    \subfloat[\centering Error]{{\includegraphics[width=.45\linewidth]{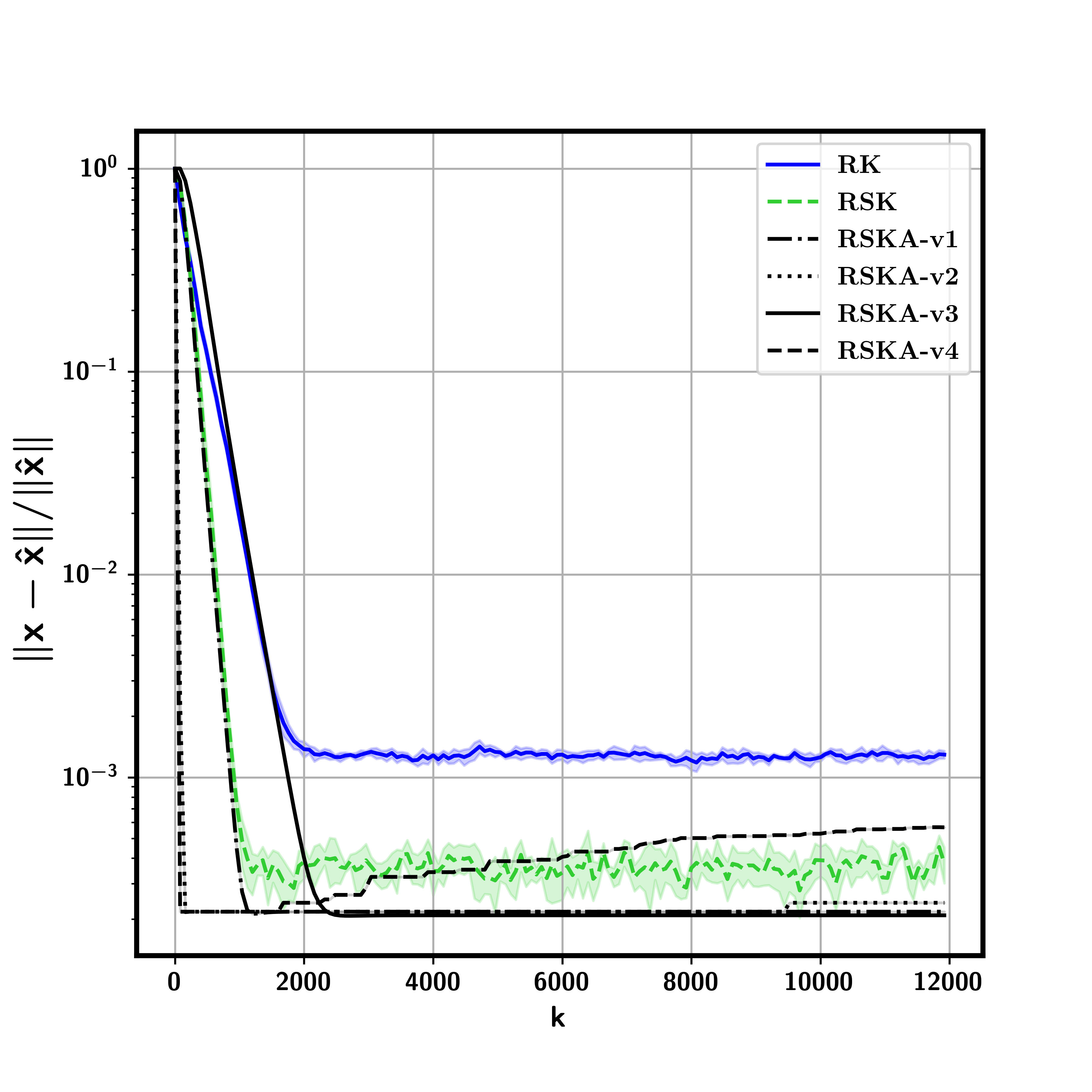} }}%
    \caption{\small A comparison of randomized Kaczmarz (blue), randomized sparse Kaczmarz (green) and RSKA method (black), $m = 500, n = 100,$ sparsity $s=10$, $\eta=11$, $\lambda=1$, noise level $l=0.1$ and 5 runs. Thick line shows mean over all trials, shaded area represent the standard deviation.}%
    \label{fig:example11}%
    %\vspace{-0.6cm}
\end{figure}

\begin{figure}[htb]%
    \centering
    \subfloat[\centering Relative Residual]{{\includegraphics[width=.45\linewidth]{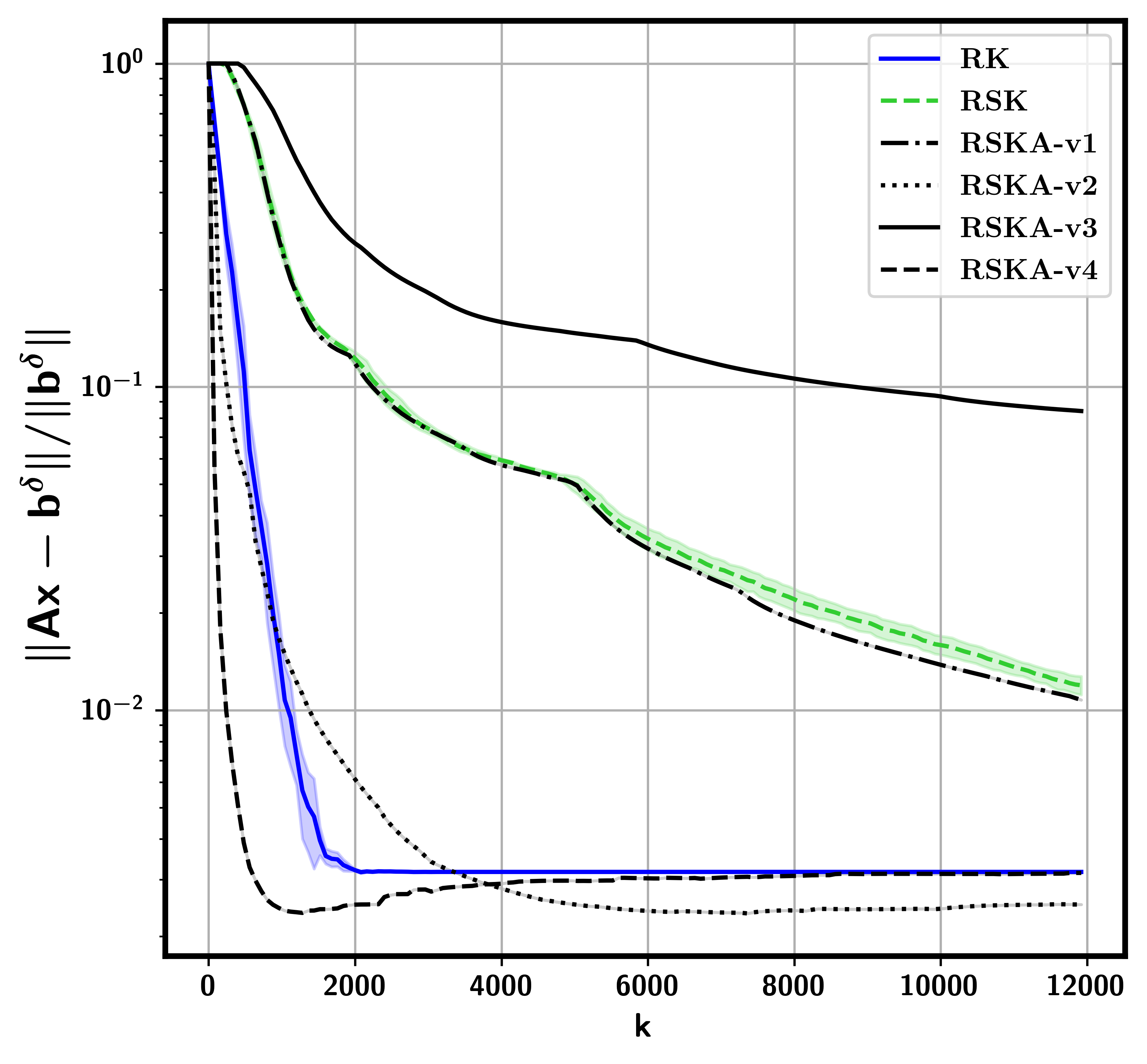} }}%
    \qquad
    \subfloat[\centering Error]{{\includegraphics[width=.45\linewidth]{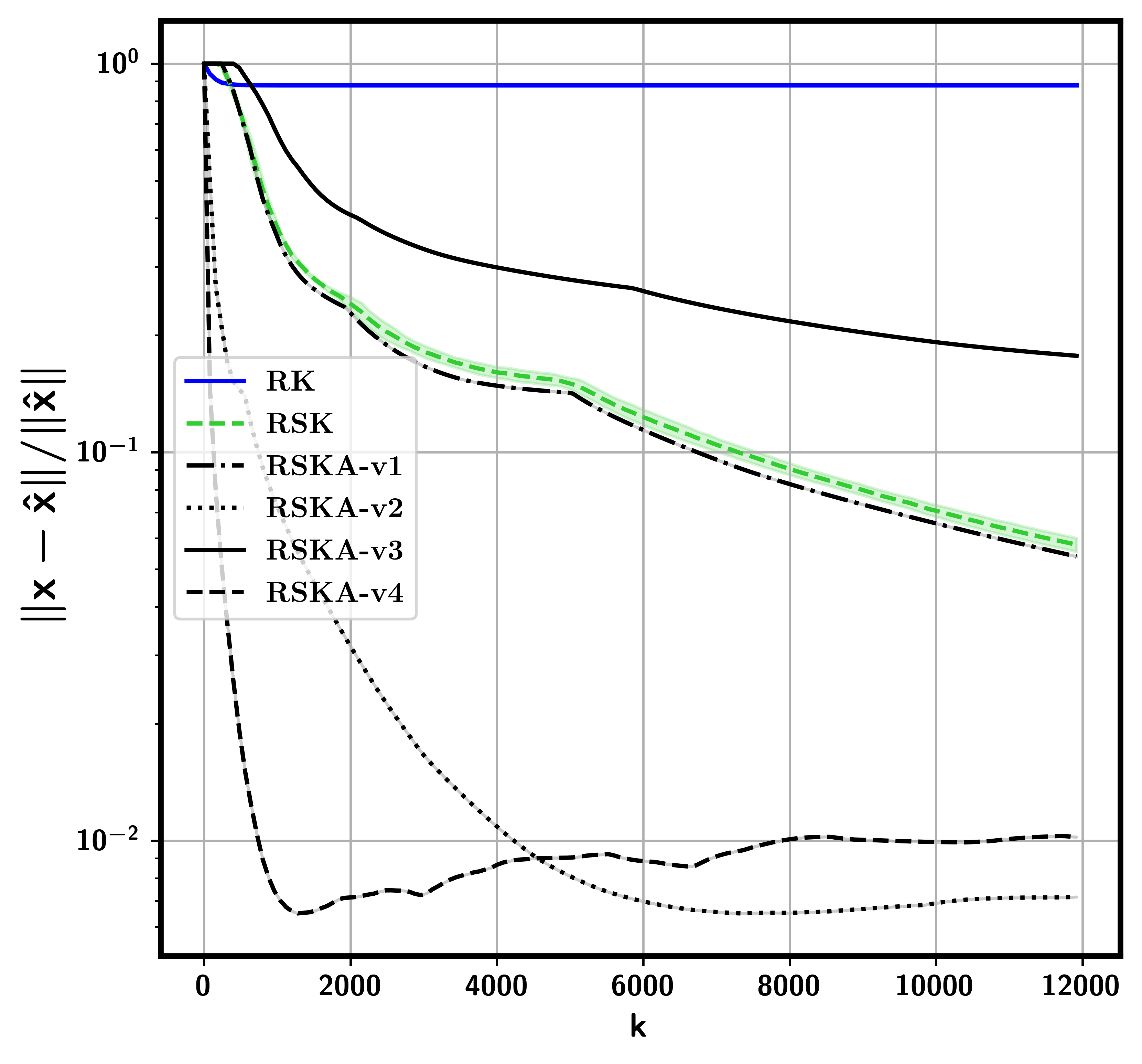} }}%
    \caption{\small A comparison of randomized Kaczmarz (blue), randomized sparse Kaczmarz (green) and RSKA method (black), $m = 100, n = 500,$ sparsity $s=10$, $\eta=11$, $\lambda=1$, noise level $l=0.1$ and 5 runs. Thick line shows mean over all trials, shaded area represent the standard deviation.}%
    \label{fig:example12}%
\end{figure}
%\vspace{-0.5cm}
%\end{comment}

\section{Conclusion}
\label{sec:conclusion}
We proved that the iterates of the randomized sparse Kaczmarz with averaging method (Algorithm \ref{alg:RSKA}) are expected to converge linearly for consistent linear systems. Moreover we show that the iterates reach an error threshold in the order of the noise-level in the noisy case. We gave a general error bound in terms of the sparsity parameter $\lambda$, the number of threads $\eta$ and a relaxation parameter $\alpha$. Numerical experiments show that the method performs consistently well over a range of  values of $\lambda$ (which is different for the version without averaging), very good reconstruction quality as $\lambda$ increases, confirm the theoretical results, and demonstrate the benefit of using Algorithm \ref{alg:RSKA} to recover sparse solutions of linear systems, even in  the noisy case. We demonstrate that the rate of convergence for Algorithm \ref{alg:RSKA} improve both in theory and practice as the number of threads $\eta$ increases. Moreover, we derive an optimal value for the relaxation parameter $\alpha$ which gives the fastest convergence speed and our numerical experiments indicate that this optimal value for $\alpha$ (in v2 and v4 of the algorithms in Section~\ref{sec:numerical_experiments}) does indeed provide fast convergence.

\paragraph{Conflict of interest.}
The authors declare that they have no conflict of interest.

\paragraph{Availability of data and materials}
The data for the numerical example are randomly generated matrices.
%while the data for Figure~\ref{fig:example13} and~\ref{fig:example14} comes from the Mnist data set which is freely available from~\url{http://yann.lecun.com/exdb/mnist/}.

\paragraph{Code availability.}
The computational code is only prototypal, but it is available from the authors upon request.

% \begin{acknowledgements}
% If you'd like to thank anyone, place your comments here.
% \end{acknowledgements}

\bibliographystyle{spmpsci}      % mathematics and physical sciences
\bibliography{references}   % name your BibTeX data base

\end{document}